\documentclass[a4paper,twoside,11pt,leqno]{amsart}
\usepackage{fullpage}
\usepackage{amsmath} \usepackage{amssymb} \usepackage{amsopn}
\usepackage{amsthm} \usepackage{amsfonts} \usepackage{mathrsfs}
\usepackage{mathabx}
\usepackage{stmaryrd}

\usepackage{verbatim}

\usepackage[OT2,T1]{fontenc}

\usepackage[pdfpagelabels,pdftex]{hyperref}
\usepackage[dvips]{graphicx} \usepackage[usenames]{color}
\usepackage[all]{xy}

\usepackage{setspace}
\usepackage{amssymb}
\usepackage{euscript}
\usepackage{cite}
\usepackage[all]{xy}
\usepackage{tabularx}

\theoremstyle{plain}
\usepackage{amsfonts}
\usepackage{graphicx}
\usepackage{amsmath}

\hypersetup{
  pdftitle={},
  pdfauthor={},
  pdfsubject={},
  pdfkeywords={},
  colorlinks=false,
  breaklinks=true,
  bookmarksopen=true,
  bookmarksnumbered=true,
  pdfpagemode=UseOutlines,
  plainpages=false}

\newcommand{\bF}{{\mathbb F}}
\newcommand{\bG}{{\mathbb G}}

\newcommand{\bP}{{\mathbb P}}
\newcommand{\bQ}{{\mathbb Q}}
\newcommand{\bR}{{\mathbb R}}

\newcommand{\bZ}{{\mathbb Z}}

\newcommand{\cA}{{\mathscr A}}
\newcommand{\cB}{{\mathscr B}}
\newcommand{\cC}{{\mathscr C}}

\newcommand{\cF}{{\mathscr F}}

\newcommand{\cN}{{\mathscr N}}
\newcommand{\cO}{{\mathscr O}}

\newcommand{\cZ}{{\mathscr Z}}

\newcommand{\caO}{{\mathcal O}}

\newcommand{\fJ}{{\mathfrak J}}

\newcommand{\fM}{{\mathfrak M}}

\newcommand{\fT}{{\mathfrak T}}

\newcommand{\fp}{{\mathfrak p}}

\DeclareSymbolFont{cyrletters}{OT2}{wncyr}{m}{n}
\DeclareMathSymbol{\Sha}{\mathalpha}{cyrletters}{"58}

\DeclareMathOperator{\pr}{pr}

\DeclareMathOperator{\Hom}{Hom}

\DeclareMathOperator{\im}{im}

\DeclareMathOperator{\GL}{GL}
\DeclareMathOperator{\SL}{SL}

\newcommand{\tr}{{\rm tr}}

\newcommand{\matzz}[4]{\left(
\begin{array}{cc} #1 & #2 \\ #3 & #4 \end{array} \right)}

\DeclareMathOperator{\Spec}{Spec}

\newcommand{\inv}{{\rm inv}}
\DeclareMathOperator{\Frob}{Frob}

\DeclareMathOperator{\R}{R}

\DeclareMathOperator{\Gal}{Gal}
\DeclareMathOperator{\Rep}{Rep}

\DeclareMathOperator{\sgn}{sgn}
\newcommand{\lengthofv}{n}

\makeatletter
\newtheorem*{rep@theorem}{\rep@title}
\newcommand{\newreptheorem}[2]{%
\newenvironment{rep#1}[1]{%
 \def\rep@title{#2 \ref{##1}}%
 \begin{rep@theorem}}%
 {\end{rep@theorem}}}
\makeatother

\newtheorem{thm}{Theorem}[section]

\newtheorem{prop}[thm]{Proposition}

\newtheorem{cor}[thm]{Corollary}

\newtheorem{lm}[thm]{Lemma}

\newtheorem{conj}[thm]{Conjecture}

\newreptheorem{thm}{Theorem}
\newreptheorem{prop}{Proposition}
\newreptheorem{cor}{Corollary}

\theoremstyle{definition}
\newtheorem{Def}[thm]{Definition}

\newtheorem{notat}[thm]{Notation}

\newtheorem{rem}[thm]{Remark}

\newenvironment{pro*}[1][Proof]{{\it{#1:}} }{}

\newcommand\Indd{\mathop{ \rm Ind}}

\newcommand\cIndd{\mathop{ \rm c-Ind}}

\newcommand\rar{ \rightarrow }
\newcommand\tar{ \twoheadrightarrow }
\newcommand\har{ \hookrightarrow }
\newcommand\Rar{ \Rightarrow }
\newcommand\LRar{ \Leftrightarrow }
\newcommand\longrar{\longrightarrow}

\newcommand\charac{\mathop{ \rm char}}

\newcommand{\sm}{{\,\smallsetminus\,}}
\newcommand\N{{\rm N}}

\newcommand\coh{{\rm H}}

\newcommand\Ind{\mathop{ \rm Ind}}

\newcommand{\diag}{{\rm diag}}

\newcounter{absatzcounter}[section]
\numberwithin{equation}{section}

\begin{document}

\title{Ramified automorphic induction and zero-dimensional affine Deligne-Lusztig varieties}
\author{Alexander B. Ivanov}
\address{Alexander B. Ivanov,  Fakult\"at f\"ur Mathematik der Technischen Universit\"at M\"unchen - M11, Boltzmannstr. 3, 85748 Garching, Germany}
\email{ivanov@ma.tum.de}
\urladdr{http://http://www-m11.ma.tum.de/dr-alexander-ivanov/}
\thanks{The author was partially supported by ERC starting grant 277889 ''Moduli spaces of local $G$-shtukas''.}


\begin{abstract}
To any connected reductive group $G$ over a non-archimedean local field $F$ (of characteristic $p > 0$) and to any maximal torus $T$ of $G$, we attach a family of extended affine Deligne-Lusztig varieties (and families of torsors over them) over the residue field of $F$. This construction generalizes affine Deligne-Lusztig varieties of Rapoport, which are attached only to unramified tori of $G$. Via this construction, we can attach to any maximal torus $T$ of $G$ and any character of $T$ a representation of $G$. This procedure should conjecturally realize the automorphic induction from $T$ to $G$. 

For $G = \GL_2$, we prove that our construction indeed realizes the automorphic induction from at most tamely ramified tori. Moreover, if the torus is purely tamely ramified, then the varieties realizing this correspondence are zero-dimensional and reduced, i.e., just disjoint unions of points. 
\end{abstract}

\maketitle

\section{Introduction}

Let $G$ be a connected reductive group over a finite field $\bF_q$ and let $\sigma$ be the Frobenius over $\bF_q$. Then there is a natural correspondence, which to any pair $(T, \chi)$ consisting of a maximal $\bF_q$-torus $T$ of $G$ and a character $\chi$ of $T(\bF_q)$ in general position, associates an irreducible representation $R_T^{\chi}$ of $G(\bF_q)$. Moreover, $R_T^{\chi}$ is cuspidal, whenever $T$ is anisotropic modulo the center of $G$. This was conjectured in 1968 by MacDonald and proven in 1976 by Deligne and Lusztig in their celebrated paper \cite{DL}. They defined $R_T^{\chi}$ as the alternating sum of the $\ell$-adic cohomology with compact support of an \'etale cover of a Deligne-Lusztig variety. This last is just the variety of all Borel subgroups of $G$, which are in a fixed relative position with their $\sigma$-translates.

Let now $G$ be a connected reductive group over a non-archimedean local field  $F$ with residue field $\bF_q$ and let $\ell \neq \charac(\bF_q)$ be a prime. For simplicity (to avoid dealing with endoscopy phenomena, etc.) let us assume here that $G = \GL_n$. Then there is a correspondence, similar to the one explained above, which associates to any pair $(T, \chi)$ consisting of a maximal $F$-torus $T$ of $G$ and a smooth $\overline{\bQ}_{\ell}^{\times}$-character of $T(F)$ in general position, an irreducible representation $R_T^{\chi}$ of $G(F)$ over $\overline{\bQ}_{\ell}$. Moreover, $R_T^{\chi}$ is supercuspidal, whenever $T$ is anisotropic modulo the center of $G$. Such a correspondence is a special case of the more general principle of \emph{automorphic induction} for $G$, which is closely related to the local Langlands correspondence. 

Let $G$ again be arbitrary. Roughly, one can divide all geometrical objects attached to $G$, in the cohomology of which one has tried to realize the automorphic induction, into two types:

\begin{itemize}
\item[(i)] Varieties (or rigid or adic spaces) over $\Spec F$ equipped with integral models over $\Spec \caO_F$ and special fibers over $\bF_q$.
\item[(ii)] Reduced varieties over $\bF_q$.
\end{itemize}

Constructions of type (ii) are purely in characteristic $p$, i.e., over $\bF_q$, and only the reduced structure is relevant.
Up to now, constructions of type (ii) only existed for unramified tori of $G$ (except for a construction by Stasinski for $\GL_n$ and $\SL_n$, see below), which was a serious drawback. This article contributes to the automorphic induction over local fields by introducing a new construction of type (ii), which works for all tori and all reductive groups (in the equal characteristic case). For $G = \GL_2$ we prove that our construction indeed realizes the $\ell$-adic automorphic induction for all at most tamely ramified tori. Here a very intriguing phenomenon occurs: the constructed varieties attached to the tamely ramified torus of $\GL_2$ turn out to be zero-dimensional and reduced, more precisely, they are just discrete unions of $\bF_q$-rational points and the automorphic induction is realized in their zeroth cohomology groups $\coh_c^0(-, \overline{\bQ}_{\ell})$ with coefficients in the constant sheaf $\overline{\bQ}_{\ell}$.

We recall some of the existing unramified constructions of type (ii). A first such construction was suggested in 1977 by Lusztig \cite{Lu}. Its variants were studied by Boyarchenko, Boyarchenko-Weinstein and Chan in \cite{Bo}, \cite{BW}, \cite{Ch}. A different, but apparently related approach via higher level covers of Rapoport's affine Deligne-Lusztig varieties was studied by the author in \cite{Iv2} (in the equal characteristic case). The nature of all these constructions is strongly related to the classical Deligne-Lusztig construction explained in the beginning. In particular, if $\breve{F}$ denotes the completion of the unramified closure of $F$, $\sigma$ the Frobenius of $\breve{F}/F$, and $b \in G(\breve{F})$ is some element, then an affine Deligne-Lusztig variety attached to these data can be seen as the subvariety of the affine flag manifold of $G$, consisting of all Iwahori subgroups of $G(\breve{F})$ being in a fixed relative position with their $b\sigma$-translates.


\subsection*{Main construction} 
We will define the \emph{extended affine Deligne-Lusztig varieties} and torsors naturally attached to them in Section \ref{sec:general_section} below. Roughly, the construction goes as follows. Let $F = \bF_q((t))$ be a local field of characteristic $p$. Let $G$ be a connected reductive group over $F$. Let $\fT$ be a maximal $F$-torus of $G$. Let $\breve{E}/F$ be the completion of the maximal unramified extension of the splitting field $E$ of $\fT$ and let $\Sigma$ be a set of generators of $\Gal_{\breve{E}/F}$. Let $\underline{w}$ be a map from $\Sigma$ to the set of all possible relative positions of Iwahori subgroups of $G(\breve{E})$ and let $b \in G(\breve{E})$. Then the extended affine Deligne-Lusztig set attached to $\underline{w}$ and $b$ is the subset $X_{\underline{w}}(b)$ of the affine flag manifold $\cF$ of $G_{\breve{E}}$ consisting of all Iwahori subgroups, whose relative position to their $b\gamma$-translate is equal to $\underline{w}(\gamma)$ for all $\gamma \in \Sigma$.

Now we turn to torsors of $X_{\underline{w}}(b)$. By a level $f$ we essentially mean a congruence subgroup $I^f$ of some $\Gal_{\breve{E}/F}$-stable Iwahori subgroup $I$ of $G(\breve{E})$. Attached to such $f$, there is a natural cover $\cF^f \tar \cF$ of the affine flag manifold. Then to any lift $\underline{w}_f$ of $\underline{w}$ to a function into an appropriate space of relative positions of level $f$, we naturally attach a subset $X_{\underline{w}_f}^f(b)$ of $\cF^f$, which lies over $X_{\underline{w}}(b)$. In many cases, $X_{\underline{w}}(b)$ and $X_{\underline{w}_f}^f(b)$ can be given a scheme structure, turning them into reduced schemes locally of finite type over a finite extension of $\bF_q$. Moreover, we obtain two natural commuting group actions

\begin{equation*} 
J_b(F) \,\,\, \rotatebox[origin=c]{-90}{$\circlearrowright$} \,\,\, X_{\underline{w}_f}^f(b) \,\, \rotatebox[origin=c]{90}{$\circlearrowleft$} \,\, \tilde{I}_{f, \underline{w}_f} \tar \fT(F).
\end{equation*}

\noindent Here $J_b$ is the $\Sigma$-stabilizer of $b$, i.e., the algebraic group over $F$ defined by

\[ J_b(R) := \{ g \in G(R \otimes_F \breve{E}) \colon g^{-1} b \gamma(g) = b \,\, \forall \gamma \in \Sigma \} \]

\noindent for an $F$-algebra $R$, and $\tilde{I}_{f, \underline{w}_f}$ is a certain subgroup of $G(\breve{E})$, which depends on $\underline{w}_f$ (but not on $b$) and admits $\fT(F)$ as a natural quotient, if $\underline{w}_f$ is appropriate. Further, $X_{\underline{w}_f}^f(b)$ is in a natural way a torsor over $X_{\underline{w}}(b)$ under a certain subquotient of $\tilde{I}_{f, \underline{w}_f}$, which is an algebraic group of finite type over a finite extension of $\bF_q$.

The construction explained above generalizes the unramified construction from \cite{Iv2}, i.e., if one chooses $\fT$ to be an unramified maximal torus of $G$ (i.e., $\breve{E} = \breve{F}$) and $\Sigma$ to be the set  with one element containing only the Frobenius of $\breve{F}/F$, then the corresponding Iwahori-level variety $X_{\underline{w}}(b)$ will be just the affine Deligne-Lusztig variety of Rapoport, and the torsors $X_{\underline{w}_f}^f(b)$ over it will be precisely the torsors defined in \cite{Iv2}.

We wish to point out that in 2011 Stasinski made in \cite{St} the first attempt to define some varieties (of type (ii)) attached to ramified tori. He worked over finite rings $\bF_q[t]/(t^r)$ and was interested in the representation theory of the finite group $G(\bF_q[t]/(t^r))$. For $G = \GL_n, \SL_n$, he constructed extended Deligne-Lusztig varieties (hence our choice of terminology) attached to $G(\bF_q[t]/(t^r))$ and any maximal torus in $G = \GL_n, \SL_n$. This construction is technically involved, and, in particular, works a priori only for $G = \GL_n, \SL_n$. Also, there are issues about defining higher-level torsors. Nevertheless, the main ideas of his and our constructions seem to coincide: in both cases, one defines a variety by fixing the relative position with respect to many automorphisms of an extension of $F$, and not only with respect to the Frobenius. Moreover, the first example of a zero-dimensional variety (attached to $\SL_2(\bF_q[t]/(t^2))$), realizing interesting representations occurs in \cite{St}.


\subsection*{Affine Deligne-Lusztig induction} We can use the $\ell$-adic cohomology of $X_{\underline{w}_f}^f(b)$ to define the following map, which we call the \emph{affine Deligne-Lusztig induction}:

\begin{equation}\label{eq:adlv_induction_map} 
R = R_{f, \underline{w}_f, b} \colon \Rep_{\overline{\bQ}_{\ell}} (\tilde{I}_{f, \underline{w}_f}/I^f) \rar \Rep_{\overline{\bQ}_{\ell}} (J_b(F)) , \quad \chi \mapsto \sum_i (-1)^i \coh_c^i(X_{\underline{w}_f}^f(b), \overline{\bQ}_{\ell})[\chi] 
\end{equation}

\noindent between the categories of smooth $\overline{\bQ}_{\ell}$-representations. In particular, whenever $\underline{w}_f$ is such that $\fT(F)$ is a quotient of $I_{f,\underline{w}_f}$, characters of $\fT(F)$ (of level bounded by $f$) induce, after inflation to $\tilde{I}_{f, \underline{w}_f}$, representations of $J_b(F)$. We formulate the following conjecture here only for $\GL_n$ and $b = 1$, in which case $J_1(F) = G(F)$. For more general reductive groups $G$ we expect that a similar statement is true, but that some endoscopy phenomena occur.

\begin{conj}\label{conj:adlv_induction} 
Let $G = \GL_n$ and $b = 1$. The collection of maps \eqref{eq:adlv_induction_map} satisfies the following properties:
\begin{itemize}
\item[(A)] If $\fT$ is anisotropic modulo the center of $G$, and $\chi$ is a character of $\fT(F)$ in sufficiently general position, then there are $f,\underline{w}_f$, such that $\tilde{I}_{f, \underline{w}_f} \tar \fT(F)$ and $R_{f, \underline{w}_f, 1}(\chi)$ is an irreducible supercuspidal representation of $G(F)$. 
\item[(B)] The map $\chi \mapsto R_{f, \underline{w}_f, 1}(\chi)$ in (A) is injective up to Galois conjugation.
\item[(C)] The map $\chi \mapsto R_{f, \underline{w}_f, 1}(\chi)$ in (A) coincides with the realization of the automorphic induction constructed via cuspidal types by Bushnell, Kutzko and others (see \cite{BK}).
\end{itemize}
\end{conj}

\noindent Further, we expect the following two facts to be true:

\begin{itemize}
\item[(D)] If $\fT$ is unramified, then $X_{\underline{w}_f}^f(1)$ is isomorphic (up to a possible unessential defect) to the reduction of the open affinoid in the Lubin-Tate perfectoid space constructed by Boyarchenko and Weinstein in \cite{BW}.
\item[(E)] If $\fT$ is purely tamely ramified, then $\underline{w}_f$ can be chosen such that $X_{\underline{w}_f}^f(1)$ is a zero-dimensional reduced scheme (disjoint union of points).
\end{itemize}

The evidence for Conjecture \ref{conj:adlv_induction} and expectations (D) and (E) is build up mainly on the at most tamely ramified $\GL_2$-case (see Theorem \ref{thm:A} below), and the analogy with the classical Deligne-Lusztig induction. We discuss some further evidence below in this introduction.


\subsection*{Case $G = \GL_2$.} 
Let $\bP_2(F)$ be the set of all isomorphism classes of admissible pairs $(E/F,\chi)$ attached to at most tamely ramified quadratic extensions $E/F$. Note that if $\fT \subseteq G$ is a torus with $\fT(F) = E^{\times}$, then characters of $\fT(F)$ in general position up to Galois conjugation are in 1:1-correspondence with the subset of minimal pairs. Let $\cA_2^{\rm tame}(F)$ be the set of isomorphism classes of irreducible supercuspidal representations of $G(F)$, which are additionally assumed to be unramified if $\charac F = 2$. Then the tame parametrization theorem (\cite{BH} 20.2 Theorem) shows the existence of a certain well-behaved bijection
\begin{equation} \label{eq:tame_param_for_GL_2}
\bP_2(F) \stackrel{\sim}{\rar} \cA_2^{\rm tame}(F), \quad (E/F, \chi) \mapsto \pi_{\chi}. 
\end{equation}

\noindent The following theorem shows Conjecture \ref{conj:adlv_induction} and expectation (E) for $\GL_2$ for all at most tamely ramified tori in $\GL_2$.

\begin{thm}[rough statement; cf. Theorem \ref{thm:hard_version_main_result} and \cite{Iv2} Theorem 4.3] \label{thm:A}
Let $G = \GL_2$. Let $\fT$ be a non-split maximal torus of $G$. Let $\breve{E}/F$ be the completion of the unramified closure of the splitting field $E$ of $\fT$. Then there are choices of $\Sigma \subseteq \Gal_{\breve{E}/F}, f, \underline{w}_f$ such that the corresponding maps $R = R_{f,\underline{w}_f,1}$ from \eqref{eq:adlv_induction_map} realize the automorphic induction from $\fT$ to $G$. Moreover, they induce a bijection 
\begin{equation} \label{eq:R_for_GL_2}
\bP_2(F) \stackrel{\sim}{\rar} \cA_2^{\rm tame}(F), \quad (E/F, \chi) \mapsto R_{\chi}, 
\end{equation}

\noindent and one has $R_{\chi} \cong \pi_{\chi}$, i.e., the maps \eqref{eq:R_for_GL_2} and \eqref{eq:tame_param_for_GL_2} coincide. In the case of the ramified torus, $\underline{w}_f$ can be chosen such that the varieties $X_{\underline{w}_f}^f(1)$ are disjoint unions of $\bF_q$-points.
\end{thm}

If $\fT$ is tamely ramified with splitting field $E$, it suffices to work with $E$ instead of $\breve{E}$ to obtain the same results (see Remark \ref{rem:forget_unram_and_comps}(i)). A similar statement is not true if $\fT$ is unramified.


\subsection*{Ramified torus in $\GL_2$} Besides the construction explained above, the main result of this article is Theorem \ref{thm:hard_version_main_result} (the tamely ramified case of Theorem \ref{thm:A}; the unramified case was proven in \cite{Iv2}). Its proof is analogous to the proof in the unramified case \cite{Iv2} Theorem 4.3. Roughly speaking, there are three instruments used in the proof:

\begin{itemize}
\item[(1)] a trace formula, which in the present case reduces, due to zero-dimensionality, to (quite involved) point-counting arguments. 
\item[(2)] the theory of elementary modifications of characters of $E^{\times}$, developed in Section \ref{sec:some_character_theory}. It allows to prove our second main result, Theorem \ref{thm:Xi_chi_restriction_to_Etimes}. It gives a precise description of the restriction of a certain cuspidal type (from which $R_{\chi}$ is induced) to an $E^{\times}$-representation. From this we deduce the injectivity of \eqref{eq:R_for_GL_2}
\item[(3)] the theory of cuspidal types of Bushnell, Kutzko, Henniart and others. We need it to compare $R_{\chi}$ with $\pi_{\chi}$ and, in particular, to show surjectivity of \eqref{eq:R_for_GL_2}.
\end{itemize}


\subsection*{Some evidence for Conjecture \ref{conj:adlv_induction}} \label{sec:evidence_for_conj} 
Let now $G = \GL_n$ and let $\fT$ be a maximal $F$-torus of $G$. Here are some heuristic reasons, justifying Conjecture \ref{conj:adlv_induction} and the expectations (D) and (E):
\begin{itemize}
\item[$\bullet$] If $G = \GL_2$ and $\fT$ is at most tamely ramified, (A),(B),(C),(E) hold (see Theorem \ref{thm:hard_version_main_result} below and \cite{Iv2} Theorem 4.3) 
\item[$\bullet$] The construction is completely analogous to the classical Deligne-Lusztig induction.
\item[$\bullet$] (D) becomes evident for $G = \GL_2$, $b$ superbasic by looking at the explicit defining equations (see \cite{Iv2} Section 3.6).
\item[$\bullet$] Assume $\fT$ is unramified, and let $\Sigma := \{\sigma\}$ consist only of the Frobenius, i.e., the corresponding varieties are torsors of level $f$ over the usual affine Deligne-Lusztig varieties $X_w(1)$, where $w$ is an element of the extended affine Weyl group $\tilde{W}$ of $G$. By \cite{GH} Proposition 2.2.1, if $w$ is contained in a finite Weyl subgroup of $\tilde{W}$, then essentially, $X_w(1) \cong \coprod_{g \in G(F)/G(\caO_F)} g X_w$, where the union is taken over translates of classical Deligne-Lusztig varieties $X_w$. This implies a similar decomposition for level-$f$-covers $X_{\underline{w}_f}^f(1) = \coprod_{g \in G(F)/G(\caO_F)} g Y^f_{\underline{w}_f}$ with $Y^f_{\underline{w}_f}$ of finite type. If $Z$ denotes the center of $G(F)$, we deduce
\[ \coh_c^{\ast}(X_{\underline{w}_f}^f(1), \overline{\bQ}_{\ell})[\chi] = \cIndd\nolimits_{ZG(\caO_F)}^{G(F)} \coh_c^{\ast}(Y^f_{\underline{w}_f}, \overline{\bQ}_{\ell})[\chi]. \]

\noindent On the other side, it follows from the theory of cuspidal types (see \cite{BK}) that any supercuspidal representation is compactly induction from a cuspidal type, thus we have, in particular, a natural family of supercuspidal representations of $G(F)$, which are all of the form $\cIndd_{ZG(\caO_F)}^{G(F)} \Xi$, where $\Xi$ is some cuspidal inducing datum. Thus if the torus is unramified, the conjecture should boil down to statements about smooth representations of the group $ZG(\caO_F)$, which is compact modulo center. For $n = 2$, this holds also for the tamely ramified torus and is part of our strategy of the proof of Theorem \ref{thm:A}.

\item[$\bullet$] Concerning expectation (E), we remark that if some lift $\sigma^{\prime}$ of the Frobenius $\sigma \in \Gal_{\breve{F}/F}$ lies in $\Sigma$ and $\underline{w}(\sigma^{\prime}) = 1$, then the extended affine Deligne-Luszstig variety $X_{\underline{w}}(1)$ of Iwahori level is a disjoint union of points (for any connected reductive $G$). 
\end{itemize}


\subsection*{Outline} In Section \ref{sec:general_section} we define the extended affine Deligne-Lusztig varieties and their covers for arbitrary connected reductive groups. In Section \ref{sec:GL_2_expl_computations} we compute certain of those varieties for $G = \GL_2$ explicitly. Based on these computations, in Section \ref{sec:repth1}, we state and prove our main results about tamely ramified automorphic induction for $\GL_2$. The proofs of all results from Section \ref{sec:repth1}, which contain any trace computations, are postponed to Section \ref{sec:applications_of_the_trace_formula}. 

\subsection*{Acknowledgements} The author is very grateful to Christian Liedtke, Stephan Neupert, Peter Scholze and Eva Viehmann for helpful discussions concerning this work. He is especially grateful to Eva Viehmann for valuable comments concerning a preliminary version of this manuscript. The author was partially supported by ERC starting grant 277889 ''Moduli spaces of local $G$-shtukas''.


\section{Extended affine Deligne-Lusztig varieties of higher level} \label{sec:general_section}

\subsection{Preliminaries}\label{sec:prelims_in_gen_def}

\subsubsection{Basic notation} \label{sec:field_extensions} Let $k$ be the finite field with $q$ elements and let $\bar{k}$ be an algebraic closure of $k$. Let $F = k((t))$. Let $\breve{F} = \bar{k}((t))$ be the completion of the maximal unramified extension of $F$. Let $E/F$ be a finite separable extension of $F$, such that $\breve{E} := E\breve{F}$ is the completion of a Galois extension of $F$. We denote by $u$ a uniformizer of $E$ and set $k_E = \bar{k} \cap E$ (inside some fixed algebraic closure of $F$). We have the identifications $E = k_E((u))$ and $\breve{E} = \bar{k}((u))$. For a Galois extension $M/L$ we denote by $\Gal_{M/L}$ its Galois group.


\subsubsection{Group theoretic data}  Let $G$ be a connected reductive group over $F$. Let $S_0$ be a maximal split torus in $G_{\breve{F}}$. By \cite{BT2} 5.1.12, $S_0$ can be chosen to be defined over $F$. Let $T_0 := \cZ_{G_{\breve{F}}}(S_0)$ be the centralizer of $S_0$. By Steinberg's theorem, $G_{\breve{F}}$ is quasi-split, hence $T_0$ is a maximal torus. Then the base change $T := T_0 \otimes_{\breve{F}} \breve{E}$ is a maximal torus of $G_{\breve{E}}$. Let $S^{\prime}$ be a maximal $\breve{E}$-split subtorus of $T$, containing $S := S_0 \otimes_{\breve{F}} \breve{E}$. 
We consider the root system $\Phi := \Phi(G_{\breve{E}},S^{\prime})$. For $a \in \Phi$, we write $U_a$ for the corresponding root subgroup of $G_{\breve{E}}$. Moreover, we write $U_0 := T$.  


\subsubsection{Bruhat-Tits buildings} \label{sec:BT_general_stuff} For any finite extension $L$ of $F$ or $\breve{F}$, let $\cB_L$ denote the Bruhat-Tits building of $G$ over $L$. It always exists by \cite{BT},\cite{BT2}\S4,\cite{Ro} Chap. 5 and \cite{MSV}. 
If $L \subseteq M$ are two such extensions such that $M/L$ is Galois, then $\Gal_{M/L}$ acts on $\cB_M$. Moreover, there is a unique embedding of $\cB_L$ into $\cB_M$ in the sense of \cite{Ro} Definition 2.5.1. Indeed, the centralizer $T$ of $S_0$ is a maximal torus, hence abelian, and its derived group is trivial. This allows to apply \cite{Ro} Theorem 2.5.6, to show that there is a unique such embedding. Note that if $M/L$ is ramified, then $\cB_L$, $\cB_M^{\Gal_{M/L}}$ are not equal as simplicial complexes. However, if $M/L$ is Galois tamely ramified, then $\cB_L = \cB_M^{\Gal_{M/L}}$ as subsets, as follows from \cite{Ro} 5.1.1 (see also \cite{Pr}).



\subsubsection{Apartments and Galois-stable alcoves} Let $\cA_{S^{\prime}}$ be the apartment of $\cB_{\breve{E}}$ corresponding to $S^{\prime}$. Via the embedding $\cB_{\breve{F}} \har \cB_{\breve{E}}$ it contains the apartment $\cA_{S_0}$ of $\cB_{\breve{F}}$ corresponding to $S_0$. The restriction of any root $a$ in $\Phi$ to $S$ is non-trivial (indeed, otherwise $U_a$ would lie in the centralizer $\cZ_{G_{\breve{E}}}(S)$ of $S$, but taking the centralizer commutes with base change, hence $\cZ_{G_{\breve{E}}}(S) = \cZ_{G_{\breve{F}}}(S_0) \otimes_F \breve{F} = T$. This leads to a contradiction). This means that $\cA_{S_0}$ is not contained in a wall of $\cA_{S^{\prime}}$. Take some alcove $\underline{a}$ of $\cA_{S^{\prime}}$ which contains a point $x_0$ of $\cA_{S_0}$ in its interior. Then $\Gal_{\breve{E}/\breve{F}}$ fixes $x_0$. As $S_0$ is defined over $F$, $\Gal_{\breve{F}/F}$ acts on $\cA_{S_0}$ and we \emph{assume} that $x_0$ is $\Gal_{\breve{F}/F}$-stable. Then $\underline{a}$ is also $\Gal_{\breve{E}/F}$-stable. Let $I$ denote the associated $\Gal_{\breve{E}/F}$-stable Iwahori subgroup of $G(\breve{E})$.

\subsubsection{Filtrations on root subgroups}  Let $x$ be a vertex of $\underline{a}$. Let $\tilde{\bR} := \bR \cup \{ r+ \colon r \in \bR \} \cup \{ \infty \}$ be the ordered monoid as in \cite{BT} 6.4.1. Then for all $a \in \Phi$, $x$ defines a filtration of $U_a(\breve{E})$ by subgroups $U_a(\breve{E})_{x,r}$ for $r \in \tilde{\bR}$.

\subsubsection{Filtration on the torus} We choose an admissible schematic filtration on tori in the sense of Yu \cite{Yu} \S4. This gives a filtration $U_0(\breve{E})_r = T(\breve{E})_r$ on $T(\breve{E})$. If $G$ satisfies the condition (T) from \cite{Yu} 4.7.1 (in particular, if $G$ is either simply connected or adjoint or split over a tamely ramified extension), then this filtration is independent of the choice of the admissible schematic filtration and coincides with the Moy-Prasad filtration \cite{Yu} Lemma 4.7.4. 

\subsubsection{Smooth models of root subgroups} \label{sec:smooth_models_of_roots_sgrs} Let $f \colon \Phi \cup \{0\} \rar \tilde{\bR}_{\geq 0} \sm \{ \infty \}$ be a concave function, i.e., $f(\sum_i a_i) \geq \sum_i f(a_i)$, for all $a_i \in \Phi \cup \{0\}$, such that $\sum_i a_i \in \Phi \cup \{0\}$ (see \cite{BT} 6.4). Following \cite{Yu}, let $G(\breve{E})_{x,f}$ be the subgroup of $G(\breve{E})$ generated by $U_a(\breve{E})_{x,f(a)}$ for $a \in \Phi \cup \{0\}$. We refer to $f$ as a \emph{level} and to $G(\breve{E})_{x,f}$ as the corresponding \emph{level subgroup}. By \cite{Yu} Theorem 8.3, there is a unique smooth model $\underline{G}_{x,f}$ of $G_{\breve{E}}$ over $\caO_{\breve{E}}$ such that $\underline{G}_{x,f}(\caO_{\breve{E}}) = G(\breve{E})_{x,f}$. Moreover, if $G(\breve{E})_{x,f}$ is $\Gal_{\breve{E}/E}$-stable, then $\underline{G}_{x,f}$ descends to a group scheme over $\caO_{E}$ (see \cite{Yu} Section 9.1).



For a concave function $f$, we write $I^f := G(\breve{E})_{x,f}$. Let $f_I$ denote the concave function such that $G(\breve{E})_{x,f} = I$. If $f \geq f_I$, then $I^f \subseteq I$.

\subsubsection{Loop groups and covers of the affine flag variety} \label{sec:loops_and_covers_of_aff_flag_var} Let $LG$ denote the functor on the category of $k_E$-algebras, 

\[ LG \colon R \mapsto G(R((u))). \]

\noindent Let $f \geq f_I$ be some level, such that $I^f$ is $\Gal_{\breve{E}/E}$-stable. Let $L^+\underline{G}_{x,f}$  be the functor on the category of $k_E$-algebras,

\[ L^+\underline{G}_{x,f} \colon R \mapsto \underline{G}_{x,f}(R\llbracket u \rrbracket). \]

\noindent Then by \cite{PR} Theorem 1.4 the quotient of fpqc-sheaves

\[ \cF^f := LG/L^+\underline{G}_{x,f} \]

\noindent is represented by an Ind-$k_E$-scheme of Ind-finite type over $k_E$ and its $\bar{k}$-points are $\cF^f(\bar{k}) = G(\breve{E})/G(\breve{E})_{x,f}$. Moreover, if $g \geq f$ are two concave functions satisfying the above assumptions, then there is a natural projection $\cF^g \tar \cF^f$. We write $\cF$ instead of $\cF^{f_I}$. This is just the affine flag manifold of $G$.

\subsubsection{Actions on $\cF^f$}\label{actions_on_Ff_in_general}
Let $f \geq f_I$ be some level. By construction, $LG$ acts on $\cF^f$ by left multiplication. In particular, $G(\breve{E}) = LG(\bar{k})$ acts on the $\bar{k}$-valued points $\cF^f(\bar{k}) = G(\breve{E})/G(\breve{E})_{x,f}$. Assume now that $I^f$ is normal in $I$ and let $Z$ be the center of $G$. Then $Z(\breve{E})I$ acts by right multiplication on $\cF^f(\bar{k})$.

\subsubsection{Extended affine Weyl group and Iwahori-Bruhat decomposition} \label{sec:ext_aff_Weyl_group_IBdec} Let $\cN_T$ be the normalizer of $T$ in $G$. Let $\tilde{W}$ be the extended affine Weyl group of $G_{\breve{E}}$ associated with $S^{\prime}$. Let $W := \cN_T(\breve{E})/T(\breve{E})$ be the finite Weyl group. Then $\tilde{W}$ sits in the short exact sequence 

\[0 \rar X_{\ast}(T_{\breve{E}})_{\Gal_{\breve{E}}} \rar \tilde{W} \rar W \rar 0 \]

\noindent (here $\Gal_{\breve{E}}$ denotes the absolute Galois group of $\breve{E}$). The Iwahori-Bruhat decomposition states that 

\[ G(\breve{E}) = \coprod_{w \in \tilde{W}} I\dot{w}I, \]

\noindent where $\dot{w}$ is any lift of $w$ to $\cN_T(\breve{E})$. 

\subsubsection{Double coset decomposition at level $f$}\label{sec:double_cosets_dec_in_level_f} Let $f \geq f_I$ be some fixed level. Consider the set of double cosets

\[ D_{G_{\breve{E}},f} := I^f \backslash G(\breve{E}) / I^f. \]

\noindent If $g \geq f$, then there is a natural projection $D_{G_{\breve{E}},g} \tar D_{G_{\breve{E}},f}$. In particular, we have the natural projection $D_{G_{\breve{E}},f} \tar D_{G_{\breve{E}},f_I} =  I \backslash G(\breve{E}) / I \cong \tilde{W}$. For many $w$, $f$ the fiber $D_{G,f}(w)$ of this projection can be given the structure of a finite-dimensional affine variety over $\bar{k}$. This will be discussed in detail in a future work. Below we only need the case $G = \GL_2$, where an explicit parametrization can be given (see Section \ref{sec:param_double_cosets_GL_2}).

\subsubsection{Relative position} Let $f$ be as in Section \ref{sec:double_cosets_dec_in_level_f}. We define the map

\[ \inv^f \colon \cF^f(\bar{k}) \times \cF^f(\bar{k}) \rar D_{G_{\breve{E}},f} \]

\noindent on $\bar{k}$-points by $\inv^f(xI^f, yI^f) = w_f$, where $w_f$ is the unique double $I^f$-coset containing $x^{-1}y$.


\subsection{Extended affine Deligne-Lusztig varieties of higher level}\label{sec:ext_adlv_higher_level} \mbox{}

\subsubsection{Main definition}

\begin{Def}\label{Def:general_ADLV_def}
Let $G$ be a connected reductive group over $F$. Fix the following data:
\begin{itemize}
\item a finite separable extension $E/F$, such that $\breve{E} = E\breve{F}$ is the completion of a Galois extension of $F$, and such that $G(\breve{E})$ contains a $\Gal_{\breve{E}/F}$-stable Iwahori subgroup $I$
\item a subset $\Sigma \subseteq \Gal_{\breve{E}/F}$ generating $\Gal_{\breve{E}/F}$
\item a concave function $f \colon \Phi \cup \{ 0 \} \rar \tilde{\bR}_{\geq 0} \sm \{ \infty \}$, such that $I^f$ is $\Gal_{\breve{E}/F}$-stable
\item a function $\underline{w}_f \colon \Sigma \rar D_{G_{\breve{E}}, f}$ 
\item an element $b \in G(\breve{E})$.
\end{itemize}

\noindent We define the \emph{extended affine Deligne-Lusztig set} $X_{\underline{w}_f}^f(b)$ attached to $(G, E/F, f, \Sigma, \underline{w}_f, b)$ as the subset

\[ X_{\underline{w}_f}^{\Sigma,f}(b) := \{x \in \cF^f(\bar{k}) \colon \inv^f(x, b\gamma(x)) = \underline{w}_f(\gamma) \,\, \forall \gamma \in \Sigma   \} \subseteq \cF^f(\bar{k}). \]
\end{Def}


\subsubsection{Left action by the $\Gal_{\breve{E}/F}$-stabilizer of $b$} For $b \in G(\breve{E})$, let $J_b$ be the $\Sigma$-stabilizer of $b$, i.e., the algebraic group over $F$ defined by 

\[ J_b(R) := \{ g \in G(R \otimes_F \breve{E}) \colon g^{-1} b \gamma(g) = b \,\, \forall \gamma \in \Sigma \} \]

\noindent for any $F$-algebra $R$. Then $J_b(F)$ acts on $X_{\underline{w}_f}^f(b)$ by left multiplication for any $f$ and any $\underline{w}_f$. If $g \geq f$ and $\underline{w}_g$ lies over $\underline{w}_f$, then $X_{\underline{w}_g}^g(b)$ lies over $X_{\underline{w}_f}^f(b)$ and the $J_b(F)$-actions are compatible.


\subsubsection{Right action on $X_{\underline{w}_f}^f(b)$ by the stabilizer of $\underline{w}_f$} \label{sec:right_action_in_general} 
Let $w \in \tilde{W}$. If $I^f$ is normal in $I$, then $Z(\breve{E})I/I^f$ acts on $D_{G_{\breve{E}}, f}$ by left and right multiplication and we obtain a (right) action of $Z(\breve{E})I/I^f$ on the set of maps $\Sigma \rar D_{G_{\breve{E}}, f}$ by $(\underline{w}_f,i)\mapsto \underline{w}_f.i$, where $(\underline{w}_f.i)(\gamma) := i^{-1} \underline{w}_f(\gamma) \gamma(i)$ for any $i \in Z(\breve{E})I/I^f$, $\gamma \in \Sigma$. This inflates to an action of $Z(\breve{E})I$ on the same set. 

\begin{lm} \label{lm:right_acting_I_general_G} Let $(G, E/F, \Sigma, f, \underline{w}_f, b)$ be as in Definition \ref{Def:general_ADLV_def}. Assume that $f \geq f_I$ and that $I^f$ is normal in $I$.  For $i \in Z(\breve{E})I$, the map $xI^f \mapsto xiI^f$ defines an isomorphism $X_{\underline{w}_f}^f(b) \stackrel{\sim}{\rar} X_{\underline{w}_f.i}^f(b)$. 
\end{lm}

\begin{proof}
The proof is an easy computation. 
\end{proof}

Thus if $I^f$ is normal in $I$, the group

\[ \tilde{I}_{f,\underline{w}_f} := \{ i \in Z(\breve{E})I \colon \underline{w}_f.i = \underline{w}_f \} \]

\noindent acts on $X_{\underline{w}_f}^f(b)$ by right multiplication. We have the subgroup $I_{f,\underline{w}_f} := \tilde{I}_{f,\underline{w}_f} \cap I$. It is clear that $I_{f,\underline{w}_f} \supseteq I^f$ and that the right $\tilde{I}_{f,\underline{w}_f}$-action on $X_{\underline{w}_f}^f(b)$ factors through an action of $\tilde{I}_{f,\underline{w}_f}/I^f$. 


\begin{lm}
$X_{\underline{w}_f}^f(b)$ is a $I_{f,\underline{w}_f}$-torsor over the underlying Iwahori-level set (resp. variety, if a variety structure is provided) $X_{\underline{w}}(b)$.
\end{lm}
\begin{proof} This follows immediately from Lemma \ref{lm:right_acting_I_general_G}.
\end{proof}

\subsubsection{Scheme structure} A priori it is not clear that $X_{\underline{w}_f}^f(b)$ can be equipped with a scheme structure in a natural way. In the following case this is possible.

\begin{prop} 
Let $(G, E/F, \Sigma, f, \underline{w}_f, b)$ be as in Definition \ref{Def:general_ADLV_def} with $f \geq f_I$. Assume that $\Sigma$ is finite and contains a lift of a power of the Frobenius element in $\Gal_{\breve{F}/F}$ and that the action of $I^f$ on $I\dot{\underline{w}}_f(\gamma)I/I^f$ by left multiplication possesses a geometric quotient for any $\gamma \in \Sigma$, where $\dot{\underline{w}}_f(\gamma)$ is any preimage of $\underline{w}_f(\gamma)$ in $G(\breve{E})$. Then the subset $X_{\underline{w}_f}^f(b) \subseteq \cF^f$ is locally closed, and hence can be equipped with the induced reduced sub-Ind-scheme structure. 
\end{prop}

\begin{proof} We write $X_{\underline{w}_f}^{\Sigma,f}(b)$ for $X_{\underline{w}_f}^f(b)$. Let $\underline{w}$ be the composition of $\underline{w}_f$ with the projection $D_{G_{\breve{E}},f} \tar \tilde{W}$. If $\tilde{\sigma}$ is a lift in $\Gal(\breve{E}/F)$ of a power of the Frobenius in $\Gal_{\breve{F}/F}$, then \cite{HV} Corollary 6.5 shows that $X_{\underline{w}(\tilde{\sigma})}^{\{ \tilde{\sigma} \},f_I}(b)$ is a locally closed subset of $\cF$. Moreover, it is a scheme locally of finite type over a finite extension of $k_E$. Now, $X_{\underline{w}}^{\Sigma,f_I}(b)$ is the subset of $X_{\underline{w}(\tilde{\sigma})}^{\{ \tilde{\sigma} \},f_I}(b)$ cut out by the finitely many locally closed conditions $x^{-1}b\gamma(x) \in \underline{w}(\gamma)$. This shows that $X_{\underline{w}}^{\Sigma,f_I}(b)$ is locally closed and locally of finite type over $k_E$.

Consider the preimage $\tilde{X}$ of $X_{\underline{w}}^{\Sigma,f_I}(b)$ under $\cF^m \tar \cF$. By \cite{PR} Theorem 1.4, the projection $\beta \colon LG \tar \cF^m$ admits sections locally for the \'etale topology. Let $U \rar \tilde{X}$ be \'etale, such that there is a section $s \colon U \rar \beta^{-1}(U)$ of $\beta$. Consider the composition of the two morphisms

\[ \psi \colon U \rar \prod_{\gamma \in \Sigma} \beta^{-1}(U) \times U \rar \prod_{\gamma \in \Sigma} \cF^f, \]

\noindent where the first is given by $x \mapsto (s(x)^{-1}, b \gamma(x))_{\gamma \in \Sigma}$ and the second is just the componentwise restriction of the left multiplication action of $G(\breve{E})$ on $\cF^f$. As $U$ lies over $\tilde{X}$, this composed morphism factors through the inclusion $\prod_{\gamma \in \Sigma} I\dot{\underline{w}}_f(\gamma)I/I^f \subseteq \prod_{\gamma \in \Sigma} \cF^f$. Let $\pi_{\gamma} \colon I\dot{\underline{w}}_f(\gamma)I/I^f \rar D_{G_{\breve{E}},f}(\underline{w}(\gamma))$ denote the geometric quotient with respect to the left multiplication action by $I^f$. Let $\pi = \prod_{\gamma \in \Sigma} \pi_{\gamma}$. Then the composition 

\begin{equation}\label{eq:etale_loc_equation_ADLV_gen} 
U \, \stackrel{\psi}{\rar} \, \prod_{\gamma \in \Sigma} IwI/I^f \, \stackrel{\pi}{\tar} \, \prod_{\gamma \in \Sigma} D_{G_{\breve{E}},f}(\underline{w}(\gamma)). 
\end{equation}

\noindent is independent of the choice of the section $s$. Moreover, it sends a $\bar{k}$-point $x$ to the tuple $(I^f x^{-1}b \gamma(x) I^f)_{\gamma \in \Sigma}$. Thus, \'etale locally, $X_{\underline{w}_f}^{\Sigma,f}(b)$ is just the preimage of a $\bar{k}$-point under the composite morphism \eqref{eq:etale_loc_equation_ADLV_gen}. This finishes the proof.
\end{proof}

The condition about the existence of geometric quotients is satisfied in many cases. This will be studied in detail in a future work.


\section{$\GL_2$, tamely ramified case: geometry} \label{sec:GL_2_expl_computations}

\emph{From here and until the end of the article we set $G = \GL_2$ and assume that $\charac k \neq 2$}. After fixing some notation in Section \ref{sec:GL_2_expl_computations_some_prelims}, we study some extended affine Deligne-Lusztig varieties of Iwahori level in Section \ref{sec:GL_2_ADLV_at_Iwahori_level} and of higher levels in Section \ref{sec:expl_structure_of_ADLV_tam_ram}.


\subsection{Some preliminaries in the $\GL_2$-case} \label{sec:GL_2_expl_computations_some_prelims}\mbox{}

\subsubsection{Basic notation} Let $\breve{F}/F = k((t))$, $k$ be as in Section \ref{sec:field_extensions} with $\charac k \neq 2$. Let $E/F$ be a tamely ramified degree $2$ extension and let $\breve{E} := E \breve{F}$. We can find an uniformizer $u \in E$ such that $u^2 = t$. Then $E = k((u))$, $\breve{E} = \bar{k}((u))$. For an algebraic extension $M$ of $F$, we denote by $\caO_M$ resp. $\fp_M$ its ring of integers resp. its maximal ideal. We have $\caO_{E} = k \llbracket u \rrbracket$, $\caO_{\breve{E}} = \bar{k}\llbracket u \rrbracket$. The Galois group of $\breve{E}/F$ is generated by the two commuting elements $\sigma$,$\tau$ given by $\sigma(\sum_i a_i u^i) = \sum_i a_i^q u_i$ and $\tau(\sum_i a_i u^i) = \sum_i (-1)^i a_i u^i$. We set $\Sigma := \{\sigma, \tau \}$\footnote{the more canonical choice of all Frobenius lifts $\Sigma^{\prime} := \{ \sigma, \sigma\tau \}$ was suggested to the author by P. Scholze. At least in the cases we study in this article, this choice will lead to the same results as $\Sigma$ from the text.}.

\subsubsection{Level subgroups.} We use the standard Iwahori subgroup $I \subseteq G(\breve{E})$ and the filtration of it given for $m \geq 0$ by

\[ I^m := \matzz{1 + \fp_{\breve{E}}^{m+1}}{\fp_{\breve{E}}^m}{\fp_{\breve{E}}^{m+1}}{1 + \fp_{\breve{E}}^{m+1}} \subseteq I := \matzz{\caO_{\breve{E}}^{\times}}{\caO_{\breve{E}}}{\fp_{\breve{E}}}{\caO_{\breve{E}}^{\times}}. \] 

\noindent We write $\cF$ for the affine flag manifold of $G_{\breve{E}}$ and $\cF^m$ for its cover corresponding to $I^m$ (see Section \ref{sec:loops_and_covers_of_aff_flag_var}).


\subsubsection{Subgroups of $G(F)$} \label{sec:subgroups_of_GF} Consider the $\caO_F$-subalgebra  
\[\fJ := \matzz{\caO_F}{\caO_F}{\fp_F}{\caO_F}\] 

\noindent of $M_2(\caO_F)$. Then the units $U_{\fJ}$ of $\fJ$ form a compact subgroup of $G(F)$. Note that $U_{\fJ} = I^{\langle\Gal_{\breve{E}/F}\rangle}$. Further, we fix the embedding of $F$-algebras

\[ \iota \colon E \har M_2(F), \qquad \iota(u) = \varpi := \matzz{}{1}{t}{} \]

\noindent (here and further, omitted entries are zeros). Via $\iota$ we consider $E^{\times}$ as a subgroup of $G(F)$. The center of $G(F)$ is $\iota(F^{\times})$. Usually we omit $\iota$ from the notation and write $E^{\times} \subseteq G(F)$, etc. We have $U_E = U_{\fJ} \cap E^{\times}$. 


\subsubsection{Root subgroups.} 

The extended set of roots $\Phi \cup \{ 0 \}$ consists of three elements. Denote by $+$ resp. $-$ the positive resp. the negative root. For $\ast \in \Phi \cup \{0\}$, we denote by

\[e_{\ast} \colon U_{\ast} \rar G \]

\noindent the embedding of the root subgroup. Thus, for $a \in \breve{E}$, $e_+(a)=\matzz{1}{a}{}{1}$, $e_0(c,d)=\matzz{c}{}{}{d}$, etc.



\subsubsection{Slices of positive loops} \label{sec:slices_of_positive_loops}

Consider the additive group $\bG_a$ over $\breve{E}$. The group $\bG_a(\breve{E})$ has a filtration by subgroups $\bG_a(\breve{E})_{\lambda} := u^{\lambda} \bar{k}\llbracket u  \rrbracket$ for $\lambda \in \bZ_{\geq 0}$. There is a unique smooth model $\bG_{a,\lambda}$ of $\bG_a$ over $\caO_{\breve{E}}$, such that $\bG_{a,\lambda}(\caO_{\breve{E}}) = \bG_a(\breve{E})_{\lambda}$. For any $\mu \leq \lambda$, there exists a unique morphism $\bG_{a,\lambda} \rar \bG_{a,\mu}$, inducing the natural embedding $u^{\lambda}k\llbracket u \rrbracket \har u^{\mu}k\llbracket u \rrbracket$ (see \cite{BT2} Section 1.7). Let $L^+$ denote the positive loop group functor from $\bar{k}\llbracket u  \rrbracket$-schemes to $\bar{k}$-schemes. For non-negative integers $\mu \leq \lambda$, we define

\[ L_{[\mu,\lambda]} \bG_a := L^+\bG_{a,\mu}/L^+\bG_{a,\lambda+1}. \]

\noindent This is a smooth $\bar{k}$-group of finite type and we have canonically $L_{[\mu,\lambda]} \bG_a (\bar{k}) = u^{\mu}\bar{k}[u]/u^{\lambda+1}\bar{k}[u]$. Replacing $\bG_a$ by $\bG_m$ and using the filtration on $\bG_m(\breve{E}) = \bar{k}((u))^{\times}$ given by $\bG_m(\breve{E})_0 = \bar{k}\llbracket u  \rrbracket^{\times}$, $\bG_m(\breve{E})_{\lambda} = 1 + u^{\lambda} \bar{k}\llbracket u  \rrbracket$ for $\lambda > 0$, we obtain in exactly the same way the $\bar{k}$-groups $L_{[\mu,\lambda]} \bG_m$. All these groups uniquely descend to smooth group schemes over $k$.

Let now $0 \leq \mu \leq \lambda \leq \lambda^{\prime}$. We have the natural projection, which comes from reduction $\mod u^{\lambda+1}$:

\[ p_{\lambda,\lambda^{\prime}} \colon L_{[\mu,\lambda^{\prime}]} \bG_a \rar L_{[\mu,\lambda]} \bG_a. \]

\noindent Moreover, there are group-theoretic sections

\[ s_{\lambda,\lambda^{\prime}} \colon L_{[\mu,\lambda]} \bG_a \rar  L_{[\mu,\lambda^{\prime}]} \bG_a \]

\noindent of $p_{\lambda,\lambda^{\prime}}$, sending $\sum_{i = 0}^{\lambda} a_i u^i$ to $\sum_{i = 0}^{\lambda} a_i u^i + \sum_{i = \lambda + 1}^{\lambda^{\prime}} 0u^i$. Then the image of $s_{\lambda,\lambda^{\prime}}$ is a closed subgroup scheme of $L_{[\mu,\lambda^{\prime}]} \bG_a$ and we denote it by $L_{[\mu,\lambda^{\prime}]}^{\leq \lambda} \bG_a$.  For $a \in L_{[\mu,\lambda^{\prime}]} \bG_a$, we use the shortcut notation $a|_{\lambda} := s_{\lambda,\lambda^{\prime}}(p_{\lambda,\lambda^{\prime}}(a))$.


\subsubsection{Schubert cells}\label{sec:param_Schubert_cells}

Let $\tilde{W}$ denote the extended affine Weyl group of $G_{\breve{E}}$ relative to the diagonal torus (as in Section \ref{sec:ext_aff_Weyl_group_IBdec}). Let $v \in \tilde{W}$ and let $\dot{v} \in G(\breve{E})$ be a lift. We denote by $C_v = I\dot{v}I/I \subseteq \cF$ the open Schubert cell attached to $v$. There is a parametrization (depending of $\dot{v}$)  of $C_v$ given by:

\[ \psi_{\dot{v}} \colon L_{[\mu,\mu+\ell(v) - 1]} \bG_a \stackrel{\sim}{\rar} C_v, \quad a \mapsto  e_{\pm}(a) \dot{v} I, \]

\noindent where $\mu \in \{0,1\}$ and the sign in $e_{\pm}$ depend on $v$, and $\ell(v)$ is the length of $v$. E.g., for $\dot{v} = \matzz{}{u^{-k}}{u^k}{}$ resp. $\dot{v} = \matzz{}{u^{-k}}{u^{k+1}}{}$, this parametrization is given by:

\begin{equation} \label{eq:Schubert_cell_param_GL2_special_v} 
\psi_{\dot{v}} \colon L_{[1,\ell(v)]} \bG_a \stackrel{\sim}{\rar} C_v, \quad a \mapsto e_-(a) \dot{v} I, 
\end{equation}

\noindent where $a = \sum_{i=1}^{\ell(v)} a_i u^i$ (note that $\ell(v) = 2k-1$ resp. $\ell(v) = 2k$). 


\subsubsection{Schubert cells in higher levels} 
For $m \geq 0$, let $\pr_m \colon \cF^m \rar \cF$ be the natural projection. Let $v \in \tilde{W}$ with lift $\dot{v}$ to $G(\breve{E})$. Let $C_v^m := \pr_m^{-1}(C_v)$. We give a parametrization of $C_v^m$ for $v,\dot{v}$ as in \eqref{eq:Schubert_cell_param_GL2_special_v} (for other $v \in \tilde{W}$ the parametrization is defined similarly). There is a well-defined injective morphism $L_{[1,\ell(v) + m]}\bG_a \rar C_v^m$ given by $a \mapsto e_-(a)\dot{v}I$. Using it we get a diagram

\centerline{
\begin{xy}\label{diag:character_isos_diag}
\xymatrix{
L_{[1,\ell(v) + m]}^{\leq \ell(v)} \bG_a \ar@{^{(}->}[r] & L_{[1,\ell(v) + m]} \bG_a \ar@{^{(}->}[r] & C_v^m \ar@{->>}[d]^{\pr_m} \\
& L_{[1,\ell(v)]} \bG_a \ar[r]^{\quad \, \sim} \ar[ul]^{\quad \, \sim} \ar@{^{(}->}[u] & C_v \ar@/^/[u]^s
}
\end{xy}
}

\noindent where the lower horizontal map is $\psi_{\dot{v}}$, the left vertical map is $s_{\ell(v),\ell(v)+m}$, and the section $s$ to $\pr_m$ is defined such that the diagram commutes. As $C_v^m \tar C_v$ is a $I/I^m$-torsor, $s$ induces the trivialization isomorphism $C_v \times I/I^m \stackrel{\sim}{\rar} C_v^m$ given by $x,i \mapsto s(x)i$. Using a parametrization of $I/I^m$, we obtain the following explicit parametrization of $C_v^m$ (depending on $\dot{v}$):

\begin{eqnarray}
\psi_{\dot{v}}^m \colon L^{\leq \ell(v)}_{[1, \ell(v) + m]} \bG_a \times L_{[0,m]} \bG_m^2 \times L_{[0,m-1]} \bG_a \times L_{[1,m]} \bG_a &\stackrel{\sim}{\longrar}& C_v^m = I \dot{v} I/I^m \nonumber \\ \label{eq:explicit_Cvm_param_GL_2_ram_case}
a, C,D,A,B &\mapsto& e_-(a) \dot{v} e_0(C,D) e_+(A)e_-(B) I^m.
\end{eqnarray}


\subsubsection{Spaces of double cosets}\label{sec:param_double_cosets_GL_2}
Let $m \geq 0$ be an integer and let $w \in \tilde{W}$. It can be seen that $C_w^m$ possesses a geometric quotient for the $I^m$-action by left multiplication. The set of its $\bar{k}$-valued points is the set of double cosets $D_{G_{\breve{E}},m}(w) = I^m \backslash IwI/I^m$. Let $\dot{w} = \matzz{}{u^{-\lengthofv}}{u^{\lengthofv}}{}$ with $\lengthofv > 0$. Let $w$ be the image of $\dot{w}$ in $\tilde{W}$. An explicit parametrization of $D_{G_{\breve{E}},m}(w)$ is given by 

\begin{eqnarray}\label{eq:explicit_DNK_param_GL_2_ram_case_all}
\phi_{\dot{w}}^m \colon L_{[0,m]} \bG_m^2 \times L_{[1,m]} \bG_a^2 &\stackrel{\sim}{\longrar}& D_{G_{\breve{E}},m}(w) = I^m\backslash I w I/I^m \nonumber \\
(C,D),(E,B) &\mapsto& I^m e_{-}(E) \dot{w} e_0(C,D) e_{-}(B) I^m.
\end{eqnarray}


\subsubsection{Bruhat-Tits buildings.}\label{sec:BT_stuff} (cf. Section \ref{sec:BT_general_stuff}) For any finite extension $M$ of $F$ or $\breve{F}$, the Bruhat-Tits building $\cB_M$ of $G$ over $M$ is an one-dimensional simplicial complex and it carries a $\Gal_{M/F}$-action if $M/F$ is Galois. We identify the subcomplex $\cB_{\breve{E}}^{\langle \sigma \rangle}$ of $\cB_{\breve{E}}$ with $\cB_E$. Moreover, as $\breve{E}/\breve{F}$ is tamely ramified, the embedding $\cB_{\breve{F}} \har \cB_{\breve{E}}$ identifies $\cB_{\breve{F}}$ with $\cB_{\breve{E}}^{\langle\tau\rangle}$ as subsets. The simplicial complex $\cB_{\breve{E}}^{\langle\tau\rangle}$ is obtained from $\cB_{\breve{F}}$ by adding an extra vertex in the middle of each alcove. Thus any alcove of $\cB_{\breve{F}}$ 'contains' two alcoves of $\cB_{\breve{E}}^{\langle\tau\rangle}$. Any vertex of $\cB_{\breve{F}}$ has an associated type in $\bZ/2\bZ = \{ 0,1\}$, which is defined as the $t$-valuation modulo $2$ of the determinant of the lattice, representing it. Similarly, we attach to any vertex of $\cB_{\breve{E}}$ its \emph{relative type} in $\frac{1}{2}\bZ/\bZ = v_t(\breve{E}^{\times}) / v_t(\breve{F}^{\times})$, defined as the class modulo $\bZ$ of the $t$-valuation of the determinant of the representing lattice. The same considerations also apply to the relationship between the $\sigma$-stable subcomplexes $\cB_F \rightsquigarrow \cB_{E}^{\langle\tau\rangle} \subseteq \cB_{E}$. 





\subsubsection{Vertex of departure} \label{sec:vertex_of_departure} 
In the proofs below we have to use the simple combinatorics of the tree $\cB_{\breve{E}}$. Therefore, following \cite{Re} we introduce the notion of the vertex of departure. Let $\cC \subseteq \cB_{\breve{E}}$ be a connected non-empty subcomplex. For any alcove $C$ of $\cB_{\breve{E}}$, which is not contained in $\cC$, there is a unique gallery $\Gamma = (C_0, C_1, \dots, C_d)$ of minimal length $d$, such that $C_0 = C$ and $C_d$ is not contained in $\cC$ and has a (unique) vertex which is contained in $\cC$. This vertex of $C_d$ is called the \emph{vertex of departure} of $C$ from $\cC$. The same considerations can also be applied to $\cB_{E}$ and a connected subcomplex. 


\subsubsection{Connected components of $\cF$}
It is well-known (\cite{PR} Theorem 5.1) that $v_u \circ \det$ induces an isomorphism 

\[ v_u \circ \det \colon \pi_0(\cF) \stackrel{\sim}{\rar} \bZ. \]

\noindent Denote the connected component of $\cF$ corresponding to the integer $i$ by $\cF^{(i)}$. Note that the alcoves of $\cB_{\breve{E}}$ can be identified with $\bar{k}$-points of $\cF^{(0)}$ and that there are (non-canonical) isomorphisms $\cF^{(0)} \stackrel{\sim}{\rar} \cF^{(i)}$. Further, denote by $\cF^{\equiv i}$ the preimage under $v_u \circ \det$ in $\cF$ of $2\bZ + i$. This divides $\cF$ in two disjoint sub-Ind-schemes $\cF = \cF^{\equiv 0} \dot{\cup} \cF^{\equiv 1}$. Note that the subgroup $G(F)$ of $G(\breve{E})$ acts transitively (by left multiplication) on the set of connected components of $\cF^{\equiv i}$ and that the action of $\matzz{}{1}{u}{} \in G(\breve{E})$ interchanges the two parts $\cF^{\equiv 0}$ and $\cF^{\equiv 1}$ of $\cF$.


\subsubsection{Commutation relations} We will need the following relations: for $a,b \in \breve{E}$ with $N := 1 + ab \neq 0$, we have 
\begin{equation}
\begin{aligned}\label{eq:commutativity_relations}
e_+(a) e_-(b) &= e_-(bN^{-1})e_+(aN) e_0(N,N^{-1}) \\
e_-(b) e_+(a) &= e_0(N^{-1},N) e_+(aN) e_-(bN^{-1}). 
\end{aligned}
\end{equation}


\subsection{Structure of $X_{\underline{w}}(1)$ at the Iwahori-level in some cases} \label{sec:GL_2_ADLV_at_Iwahori_level}\mbox{}

We write $X_{\underline{w}}(1)$ instead of $X_{\underline{w}}^{f_I}(1)$.

\begin{notat}\label{Def:D_v_tau}
Let $w := \matzz{}{u^{-\lengthofv}}{u^{\lengthofv}}{} \in \tilde{W}$ with some integer $\lengthofv > 0$ and let $\underline{w} \colon \Sigma \rar \tilde{W}$ be defined by $\underline{w}(\sigma) = 1$ and $\underline{w}(\tau) = w$.
Further, we set

\[ \dot{v} = \dot{v}(w) := \begin{cases} \matzz{}{u^{-k}}{u^k}{} & \text{if $\lengthofv = 2k - 1 > 0$ is odd,} \\ 
\matzz{}{u^{-k}}{u^{k+1}}{} & \text{if $\lengthofv = 2k > 0$ is even.}
\end{cases}
\]

\noindent In both cases let $v$ be the image of $\dot{v}$ in $\tilde{W}$ and let $D_w^{\tau}$ be the set of $k$-rational points of $C_v$ lying in the locus $a_1 \neq 0$ with respect to the coordinates \eqref{eq:Schubert_cell_param_GL2_special_v}. In particular, $D_w^{\tau}$ is just a finite discrete union of $k$-rational points.
\end{notat}

It will follow from the proof of Proposition \ref{prop:structure_GL_2_ADLV_Iwahori_level} (or can be seen directly) that $D_w^{\tau}$ is stable under the left multiplication action by $U_{\fJ}$ on $\cF$.

\begin{rem}\label{rem:descr_of_Dwtau_alcoves} If in Notation \ref{Def:D_v_tau}, $\lengthofv$ is odd, then $C_v$ is contained in the connected component $\cF^{(0)}$ of $\cF$, i.e., its points can be seen as alcoves in $\cB_{\breve{E}}$. Moreover, they all are $k$-rational, hence lie in $\cB_E$. Let $P_{1/2}$ be the vertex of the base alcove (= the alcove corresponding to $I$) of $\cB_{E}$ with relative type $\frac{1}{2}$ (see Section \ref{sec:BT_stuff}). Then $D_w^{\tau}$ corresponds to the set of the alcoves contained in $\cB_{E}$, having relative position $v$ to the base alcove and having $P_{1/2}$ as the vertex of departure from $\cB_{E}^{\langle \tau \rangle}$.
\end{rem}

\begin{prop}\label{prop:structure_GL_2_ADLV_Iwahori_level} Let $\underline{w}$ be as in Notation \ref{Def:D_v_tau}. There is an isomorphism 

\[ X_{\underline{w}}(1) \cong \coprod_{g \in G(F)/U_{\fJ}} g D_v^{\tau} \]

\noindent equivariant for the left $G(F)$-action. In particular, $X_{\underline{w}}(1)$ is a zero-dimensional reduced $k$-variety, containing only $k$-rational points. 
\end{prop}

\begin{proof}
Let first $\lengthofv$ be odd. The natural action of $G(F)$ on $X_{\underline{w}}(1)$ induces a transitive action of $G(F)$ on the set 

\[ \{ X_{\underline{w}}(1) \cap \cF^{(2i)} \colon i \in \bZ \} \] 

\noindent of subsets of $X_{\underline{w}}(1)$. This follows by taking any element $g \in G(F)$ with $v_u(\det(g)) = 2$. The stabilizer of $\cF^{(0)}$ in $G(F)$ is 

\[ H := (v_u \circ \det)^{-1}(0)  \subseteq G(F). \] 

\noindent We deduce

\begin{equation}\label{eq:dec_of_ADLV_wrt_val_det_special_for_GL2}
X_{\underline{w}}(1) \cap \cF^{\equiv 0} = \coprod_{g \in G(F)/H} g.(X_{\underline{w}}(1) \cap \cF^{(0)}).
\end{equation}

The $\bar{k}$-rational points of $X_{\underline{w}}(1) \cap \cF^{(0)}$ can be identified with alcoves in $\cB_{\breve{E}}$, which satisfy two conditions (defined by $\underline{w}(\sigma) = 1$ and $\underline{w}(\tau) = w$) on the relative position with respect to their $\sigma$- resp. $\tau$-translate. The $\sigma$-condition simply assures that each of the  alcoves contained in $X_{\underline{w}}(1)$ is $\sigma$-stable, i.e., is contained in $\cB_{E}$. Let $(\cB_{E}^{\langle \tau \rangle})^{(1/2)}$ be the set of all vertices of $\cB_{E}^{\langle \tau \rangle}$ of relative type $\frac{1}{2}$. For an alcove $C$ of $\cB_{E}$, which is not contained in $\cB_{E}^{\langle \tau \rangle}$, let $\Gamma_{C,\tau}$ denote the unique minimal gallery connecting $C$ with $\cB_{E}^{\langle \tau \rangle}$. Taking into account the types of the involved vertices, we deduce (exactly as in \cite{Iv1}) that 

\begin{equation}\label{eq:building_description_of_points} 
X_{\underline{w}}(1) \cap \cF^{(0)} = \coprod_{P \in (\cB_{E}^{\langle\tau\rangle})^{(1/2)} } \left\{ C \colon \begin{aligned} \text{ $C$ is an alcove in $\cB_{E}$ with vertex of departure from} \\ \text{$\cB_E^{\langle \tau \rangle}$ equal to $P$ and length of $\Gamma_{C,\tau}$ equal to $\lengthofv-1$ } \end{aligned} \right\},
\end{equation}  

\noindent with $\lengthofv$ as in Notation \ref{Def:D_v_tau}. Let $P_{1/2}$ be the vertex of type $\frac{1}{2}$ of the base alcove of $\cB_{E}$. Observe that for $P = P_{1/2}$, the set of alcoves $C$ on the right hand side of \eqref{eq:building_description_of_points} is simply $D_w^{\tau}$ (cf. Remark \ref{rem:descr_of_Dwtau_alcoves}). 

Now, $(\cB_{E}^{\langle \tau \rangle})^{(1/2)}$ can be canonically identified with the set of alcoves in $\cB_F$ (see Section \ref{sec:BT_stuff}). The natural action of $H$ on $\cB_F$ induces a transitive action of $H$ on the set of alcoves of $\cB_F$, and the stabilizer of the base alcove in $\cB_F$ is precisely $U_{\fJ} \subseteq H$. Combining these observations, we obtain a natural $H$-equivariant bijection 

\begin{equation} \label{eq:H_action_on_points_of_type_1/2} H/U_{\fJ} \cong (\cB_{E}^{\langle \tau \rangle})^{(1/2)}, \quad hU_{\fJ} \mapsto hP_{1/2}. 
\end{equation}



\noindent Combining \eqref{eq:dec_of_ADLV_wrt_val_det_special_for_GL2}, \eqref{eq:building_description_of_points} and \eqref{eq:H_action_on_points_of_type_1/2}, we deduce

\[ X_{\underline{w}}(1) \cap \cF^{\equiv 0} = \coprod_{g \in G(F)/H} g.\left(X_{\underline{w}}(1) \cap \cF^{(0)}\right) = \coprod_{g \in G(F)/H} g.\left(\coprod_{h \in H/U_{\fJ}} h D_w^{\tau}\right) = \coprod_{g \in G(F)/U_{\fJ}} gD_w^{\tau}. \]

\noindent It remains to show that $X_{\underline{w}}(1) \cap \cF^{\equiv 1} = \emptyset$. This can be done as follows: let $h = \matzz{}{1}{u}{}$. By Lemma \ref{lm:right_acting_I_general_G}, the map $xI \mapsto xhI$ defines an isomorphism

\begin{equation}\label{isom_by_superbasic_element_switching_parity} 
X_{\underline{y}}(1) \cap \cF^{\equiv 1} \stackrel{\sim}{\rar} X_{\underline{y}.h}(1) \cap \cF^{\equiv 0} 
\end{equation}

\noindent for any $\underline{y} \colon \Sigma \rar \tilde{W}$, where $(\underline{y}.h)(\gamma) := h^{-1}\underline{y}(\gamma)\gamma(h)$ for $\gamma \in \Sigma$. Thus it is enough to show that $X_{\underline{w}.h}(1) \cap \cF^{\equiv 0} = \emptyset$, where $(\underline{w}.h)(\sigma) = 1$, $(\underline{w}.h)(\tau) = h^{-1}w\tau(h) = \matzz{}{u^{\lengthofv - 1}}{u^{1- \lengthofv}}{} \in \tilde{W}$. This follows from \eqref{eq:dec_of_ADLV_wrt_val_det_special_for_GL2} and $X_{\underline{w}.h}(1) \cap \cF^{(0)} = \emptyset$. This last follows from the combinatorics of $\cB_{E}$ as  $\lengthofv-1$ is even: one has to use the fact that a vertex $P$ of $\cB_{E}$ of relative type $0$ cannot be the vertex of departure from $\cB_E^{\langle \tau \rangle}$ for a non-$\tau$-stable alcove $C$ of $\cB_{E}$, as all alcoves having $P$ as a vertex lie in $\cB_{E}^{\langle \tau \rangle}$.

Let now $\lengthofv$  be even. Applying the isomorphism \eqref{isom_by_superbasic_element_switching_parity} with $h$ replaced by $h^{-1}$, we reduce to determining $X_{\underline{y}}(1)$ with $\underline{y}(\sigma) = 1$ and $\underline{y}(\tau) = \matzz{}{u^{\lengthofv - 1}}{u^{1- \lengthofv}}{}$, where we can proceed exactly as in the case $\lengthofv$ odd (after replacing $\lengthofv$ by $-\lengthofv$). \qedhere
\end{proof}


\subsection{Structure of $X_{\underline{w}_m}^m(1)$ in some cases} \mbox{} \label{sec:expl_structure_of_ADLV_tam_ram} \mbox{}

Continuing with notations from preceding sections, we now study higher level covers of $X_{\underline{w}}(1)$. For $x \in G(\breve{E})$ we denote the image of $x$ in $D_{G_{\breve{E}},m}$ again by $x$, if no ambiguity can occur.  

\begin{notat}\label{notat:higher_level_w}
Let $\lengthofv$, $w$, $\underline{w}$ be as in Notation \ref{Def:D_v_tau}. We define the lift $\dot{w} \in G(\breve{E})$ of $w$ by 
\[ \dot{w} := \begin{cases} \matzz{}{(-1)^k u^{1-2k}}{(-1)^{k+1}u^{2k-1}}{} & \text{if $\lengthofv = 2k - 1 > 0$ is odd,} \\ 
\matzz{}{(-1)^k u^{-2k}}{(-1)^k u^{2k}}{} & \text{if $\lengthofv = 2k > 0$ is even.}
\end{cases} \] 

\noindent Moreover, let $m \geq 1$ be an odd integer. Let $\underline{w}_m \colon \Sigma \rar D_{G_{\breve{E}},m}$ be the lift of $\underline{w}$ defined by
\begin{equation} \label{eq:underline_w_m_here_vor_thm}
\underline{w}_m(\sigma) := 1, \qquad \underline{w}_m(\tau) := \dot{w}.
\end{equation}
\end{notat}

We have $J_1(F) = G(F)$ and hence by Section \ref{sec:ext_adlv_higher_level} we obtain the group actions

\begin{equation}\label{eq:actions_on_ADLV_GL2} 
G(F) \,\, \rotatebox[origin=c]{-90}{$\circlearrowright$} \,\, X_{\underline{w}_m}^m(1) \,\, \rotatebox[origin=c]{90}{$\circlearrowleft$} \,\, \tilde{I}_{m, \underline{w}_m}/I^m.
\end{equation}

\begin{lm}\label{lm:right_group_action_on_gen_ADLV_ram_GL_2} Let $m \geq 1$ be an odd integer. There is an isomorphism 

\[ \tilde{I}_{m, \underline{w}_m}/I^m \stackrel{\sim}{\rar} \left\{ \matzz{i_1}{i_2}{0}{\tau(i_1)} \colon i_1 \in E^{\times}/U_E^{m+1}, i_2 \in E/\fp_E^m, v_u(i_2) \geq v_u(i_1) \right\} \subset Z(\breve{E})I/I^m. \]

\noindent In particular, there is a surjection induced by the projection onto the diagonal part 
\begin{equation} \label{eq:right_group_and_mod_rad_surjection}
\tilde{I}_{m, \underline{w}_m}/I^m \tar E^{\times}/U_E^{m+1},
\end{equation}

\noindent under which $I_{m,\underline{w}_m}/I^m$ maps onto $U_E/U_E^{m+1}$.
\end{lm}

\begin{proof} The proof is an easy computation.
\end{proof}

Recall from Section \ref{sec:subgroups_of_GF} that we see $E^{\times}$ as a subgroup of $G(F)$. This defines a left multiplication action of $E^{\times}$ on $X^m_{\underline{w}_m}(1)$ (do not confuse this $E^{\times}$ with the quotient $E^{\times}$ of $\tilde{I}_{m,\underline{w}_m}$ acting on the right).


\begin{Def} \label{def:disc_subschemes_Yvm2} With notation from Notations \ref{Def:D_v_tau},\ref{notat:higher_level_w}, we define the discrete subscheme $Y_{\dot{w}}^m$ of $C_v^m \subseteq \cF^m$ as follows. Let $a  = \sum_{i = 1}^{\ell(v)} a_i u^i \in L^{\leq \ell(v)}_{[1, \ell(v) + m]} \bG_a (\bar{k})$ be as used in the parametrization \eqref{eq:explicit_Cvm_param_GL_2_ram_case} of $C_v^m$. Put $R := u^{-1}(\tau(a) - a) \mod u^{m+1}$. 
We define $Y_{\dot{w}}^m$ to be the subscheme of $C_v^m$ defined in coordinates $\psi_{\dot{v}}^m$ from \eqref{eq:explicit_Cvm_param_GL_2_ram_case} by the following conditions: 

\begin{eqnarray} 
a,A,C && \text{are $k$-rational} \nonumber \\
a_1 &\neq& 0 \text{\qquad  (in particular, $R$ is invertible)} \label{eq:Def_of_Y_m}  \\
B &=& C\tau(C)^{-1} u^{\lengthofv} \nonumber  \\
D &=& R^{-1} \tau(C)(1 + C\tau(C)^{-1}A u^{\lengthofv} - C^{-1}\tau(C) \tau(A) u^{\lengthofv}) \nonumber 
\end{eqnarray}

\noindent (both last equations take place in $k[u]/(u^{m+1})$). In particular, $Y_{\dot{w}}^m$ is just a finite discrete union of $k$-rational points. Moreover, let $y_i := e_0(u^i,(-u)^i)$ and define $\tilde{Y}_{\dot{w}}^m \subseteq \cF^m$ to be (disjoint) union 
\[ \tilde{Y}_{\dot{w}}^m := \coprod_{i \in \bZ} Y_{\dot{w}}^m \cdot y_i. \]

It will follow from the proof of Theorem \ref{thm:structure_thm_GL_2_ram} that the right multiplication action of $I/I^m$ on $C_v^m$ restricts to an action of $I_{m, \underline{w}_m}/I^m$ on $Y_{\dot{w}}^m$, which in turn extends to a right $\tilde{I}_{m, \underline{w}_m}/I^m$-action on $\tilde{Y}_{\dot{w}}^m$.


\end{Def}

\begin{rem} The varieties $Y_{\dot{w}}^m, \tilde{Y}_{\dot{w}}^m$ depend on $\dot{w}$, not only on $w$, but the choice of the lift $\dot{w}$ of $w$ is not essential: another choices would give either empty varieties or varieties isomorphic to those attached to $\dot{w}$. The full study of these choices is not relevant for the goals of this article, so we restrict our attention to our choice $\dot{w}$.
\end{rem}

\begin{thm} \label{thm:structure_thm_GL_2_ram}
Let $m\geq 1$ be an odd integer. With notation as in Definition \ref{def:disc_subschemes_Yvm2} assume that $m \leq \ell(w) = 2\lengthofv - 1$. Then $\tilde{Y}_{\dot{w}}^m$ (resp. $Y_{\dot{w}}^m$) is invariant under the left $E^{\times}U_{\fJ}$- (resp. $U_{\fJ}$-)action and the right $\tilde{I}_{m,\underline{w}_m}/I^m$- (resp. $I_{m, \underline{w}_m}/I^m$-)action and there is an isomorphism

\[ X^m_{\underline{w}_m}(1) \cong \coprod_{g \in G(F)/E^{\times}U_{\fJ}} g \tilde{Y}_{\dot{w}}^m \]

\noindent equivariant for the left $G(F)$- and right $\tilde{I}_{m, \underline{w}_m}/I^m$-actions. In particular, $X^m_{\underline{w}_m}(1)$ is a zero-dimensional reduced $k$-variety, containing only $k$-rational points. 

\end{thm}

\begin{proof} We claim that $X^m_{\underline{w}_m}(1) \cong \coprod_{g \in G(F)/U_{\fJ}} gY_{\dot{w}}^m$. As the natural projection $\cF^m \tar \cF$ restricts to a $G(F)$-equivariant projection ${\rm p}_m \colon X_{\underline{w}_m}^m(1) \rar X_{\underline{w}}(1)$, Proposition \ref{prop:structure_GL_2_ADLV_Iwahori_level} shows

\[ X_{\underline{w}_m}^m(1) \cong \coprod_{g \in G(F)/U_{\fJ}} {\rm p}_m^{-1}(gD_w^{\tau}) = \coprod_{g \in G(F)/U_{\fJ}} g.{\rm p}_m^{-1}(D_w^{\tau}). \]

\noindent Now, Lemma \ref{lm:ADLV_ram_GL_2_key_computation} implies that ${\rm p}_m^{-1}(D_w^{\tau}) = X_{\underline{w}_m}^m(1) \cap C_v^m = Y_{\dot{w}}^m$, hence the isomorphism claimed in the theorem. As ${\rm p}_m^{-1}(D_w^{\tau}) \subseteq X_{\underline{w}_m}^m(1)$ is stable under the right $I_{m,\underline{w}_m}$- and left $U_{\fJ}$-actions, the above shows that $Y_{\dot{w}}^m$ also is. As $\tilde{I}_{m,\underline{w}_m} = \coprod_i I_{m,\underline{w}_m}y_i$, with $y_i := e_0(u^i,(-u)^i)$ the theorem now follows from Lemma \ref{lm:tildeYwm_is_EUJ_stable}.
\end{proof}

\begin{lm}\label{lm:ADLV_ram_GL_2_key_computation}
Let $\dot{x}I^m = \psi_{\dot{v}}^m(a,C,D,A,B)$ be a point of $C_v^m$. Assume $m \leq \ell(w) = 2\lengthofv - 1$. 

\begin{itemize}
\item[(i)] Let $R := u^{-1}(\tau(a) - a) \mod u^{m+1}$. Then 

\begin{equation}\label{eq:tau_condition_level_m/2} \inv^m(\dot{x}I^m, \tau(\dot{x}I^m)) = \dot{w} \LRar \begin{cases} a_1 &\neq 0  \quad \text{(i.e., $R$ is invertible)} \\ B &= u^{\lengthofv} C \tau(C)^{-1} \\ D &= R^{-1} \tau(C)(1 + u^{\lengthofv} C\tau(C)^{-1}A + u^{\lengthofv} (-1)^{\lengthofv} C^{-1}\tau(C)\tau(A)) \end{cases} 
\end{equation}

\noindent (the equations on the right hand side take place in $k[u]/u^{m+1}$).

\item[(ii)] Suppose, $\dot{x}I^m$ satisfies the equations on the right hand side of \eqref{eq:tau_condition_level_m/2}. Then:

\[ \inv^m(\dot{x}I^m, \sigma(\dot{x}I^m)) = 1 \LRar a,A,B,C,D \text{ are $k$-rational}. \]
\end{itemize}
\end{lm}

\begin{proof}
Choose some lifts of $a,A,B,C,D$ to elements of $k\llbracket u \rrbracket$. We denote them by the same letters. (i): A computation shows that the $I$-double coset of $\dot{x}^{-1}\tau(\dot{x})$ is equal to the $I$-double coset of the element $e_+(u^{-\lengthofv}R)$, and $\dot{w}$ lies in this double coset if and only if $R$ is invertible. This is clearly necessary for the left hand side of part (i) to hold. Thus we can assume in the following that $a_1 \neq 0$, i.e., that $R$ is invertible. In $G(\breve{E})$ one easily computes (independently of the parity of $\lengthofv$):

\begin{equation}\label{eq:nebenrechnung_1und2_in_der_key_comp}
\dot{v}^{-1}e_-(-a)e_-(\tau(a)) \tau(\dot{v}) = e_-(u^{\lengthofv}R^{-1}) e_0(R,R^{-1})\dot{w} e_-((-1)^{\lengthofv+1}u^{\lengthofv}R^{-1}).
\end{equation}



\noindent In the rest of the proof we write $x \sim y$ to express that $x,y$ lie in the same $I^m$-double coset. Using \eqref{eq:nebenrechnung_1und2_in_der_key_comp} we compute:


\begin{eqnarray}
\nonumber \dot{x}^{-1} \tau(\dot{x}) &\sim& e_-(-B) e_+(-A) e_0(C^{-1},D^{-1}) \cdot [ e_-(u^{\lengthofv}R^{-1}) e_0(R,R^{-1}) \cdot \dot{w} \cdot e_-((-1)^{\lengthofv+1}u^{\lengthofv}R^{-1})] \cdot \dots \\ 
\nonumber &\dots& \cdot e_0(\tau(C),\tau(D)) e_+(\tau(A)) e_-(\tau(B))  \\
\label{eq:some_intermidiate_eq_for_double_coset_in_keycomp} &\sim&  e_-(-B) e_+(-A) e_-(u^{\lengthofv}CD^{-1}R^{-1}) \cdot \dot{w} \cdot e_0(D^{-1}R^{-1}\tau(C), C^{-1}R\tau(D)) \dots \\
\nonumber &\dots& e_-((-1)^{\lengthofv+1}u^{\lengthofv}\tau(C)\tau(D)^{-1}R^{-1}) e_+(\tau(A)) e_-(\tau(B)).
\end{eqnarray}

\noindent Let $N := 1 - u^{\lengthofv}CD^{-1}R^{-1}A$. We apply formulas \eqref{eq:commutativity_relations} to deduce:

\begin{align} \label{eq:some_intermidiate_eq_for_double_coset_in_keycomp_2} I^m e_-(-B) e_+(-A) e_-(u^{\lengthofv}CD^{-1}R^{-1}) &= I^m e_-(-B + u^{\lengthofv}CD^{-1}R^{-1}N^{-1})e_+(-AN)e_0(N,N^{-1}) \\
\nonumber &=  I^m e_-(-B + u^{\lengthofv}CD^{-1}R^{-1})e_+(-AN)e_0(N,N^{-1}),
\end{align}

\noindent where the last equation is true, since $u^{\lengthofv}N^{\pm 1} \equiv u^{\lengthofv} \mod u^{m+1}$, which in turn follows from $2\lengthofv - 1 \geq m$. Noting that the product of the last three matrices in the last expression in \eqref{eq:some_intermidiate_eq_for_double_coset_in_keycomp} is equal to $\tau$ applied to the inverse of the product of the first three (use $\tau(R) = R$), we deduce from \eqref{eq:some_intermidiate_eq_for_double_coset_in_keycomp} and \eqref{eq:some_intermidiate_eq_for_double_coset_in_keycomp_2}:

\begin{eqnarray*}
\dot{x}^{-1} \tau(\dot{x}) &\sim& e_-(-(B - u^{\lengthofv}CD^{-1}R^{-1})) e_+(-NA) e_0(N,N^{-1}) \cdot \dot{w} \cdot \dots \\
&\dots& e_0(\tau(C)D^{-1}R^{-1},C^{-1}\tau(D)R) e_0(\tau(N)^{-1},\tau(N)) e_+(\tau(NA)) \dots \\
&\dots& e_-(\tau(B - u^{\lengthofv}CD^{-1}R^{-1})).
\end{eqnarray*}

\noindent Now we bring the term $e_+(-NA)$ to the right side of $\dot{w}$, without modifying the other terms and it can be canceled there, as it lands in $I^m$ and $I^m$ is normal in $I$. Here we again used $2\lengthofv - 1 \geq m$.  Analogously, we cancel the term $e_+(\tau(NA))$ by bringing it to the left side of $\dot{w}$. Now put the three $e_0$-terms together and obtain

\begin{eqnarray}
\nonumber \dot{x}^{-1} \tau(\dot{x}) &\sim&  e_-(-(B - u^{\lengthofv}CD^{-1}R^{-1})) \cdot \dot{w} \cdot \dots \\
\label{eq:some_intermidiate_eq_for_double_coset_in_keycomp_3} &\dots& e_0(\tau(C)D^{-1}R^{-1}N^{-1}\tau(N)^{-1}, C^{-1}\tau(D) R N\tau(N)) e_-(\tau(B - u^{\lengthofv}CD^{-1}R^{-1})).
\end{eqnarray}

\noindent The left hand side of \eqref{eq:tau_condition_level_m/2} is equivalent to $\dot{x}^{-1} \tau(\dot{x}) \sim \dot{w}$, which by \eqref{eq:some_intermidiate_eq_for_double_coset_in_keycomp_3} and Section \ref{sec:param_double_cosets_GL_2} is equivalent to

\begin{eqnarray}
\label{eq:for_uB_tau_cond} B - u^{\lengthofv} C D^{-1} R^{-1} &\equiv& 0 \mod u^{m+1} \nonumber \\
\tau(C)D^{-1} R^{-1}N^{-1}\tau(N)^{-1} &\equiv& 1 \mod u^{m+1} \label{eq:tau_cond_zwischengl_bla} \\
C^{-1}\tau(D) R N\tau(N) &\equiv& 1 \mod u^{m+1}. \nonumber
\end{eqnarray}

\noindent Using $\tau^2 = 1$ and $\tau(R) = R$, we see that the second and the third equations are equivalent. Hence the third can be ignored. Assume first $\lengthofv \geq m + 1$. Then it is trivial to see that \eqref{eq:for_uB_tau_cond} is equivalent to the right hand side of \eqref{eq:tau_condition_level_m/2}. Assume now $m \geq \lengthofv$. Then, as $\lengthofv \geq m + 1 - \lengthofv > 0$ and $N \equiv 1 \mod u^{\lengthofv}$, the second equation of \eqref{eq:for_uB_tau_cond} shows

\begin{equation}\label{eq:D_mod_m1-ell}
D \equiv \tau(C) R^{-1} \mod u^{m + 1 - \lengthofv}.
\end{equation}

\noindent Using this and $N = 1 - u^{\lengthofv}CD^{-1}R^{-1}A$ it is now easy to deduce the equivalence of \eqref{eq:tau_cond_zwischengl_bla} and the right hand side of \eqref{eq:tau_condition_level_m/2}.


(ii): The implication '$\Leftarrow$' is immediate. We prove '$\Rar$'. Assume $\dot{x}^{-1} \sigma(\dot{x}) \in I^m$. In particular, the $I$-double coset of $\dot{x}^{-1}\sigma(\dot{x})$ is $I$. This is equivalent to $a$ being $k$-rational, and we deduce $\dot{v}^{-1} e_-(\sigma(a) - a) \sigma(\dot{v}) = 1$. Setting $G := 1 + AB$, we compute

\begin{eqnarray}\label{eq:sigma_cond_rat_matr} 
&\dot{x}^{-1} \sigma(\dot{x}) = e_-(-B)e_+(-A) e_0(C^{-1},D^{-1}) e_0(\sigma(C),\sigma(D)) e_+(\sigma(A)) e_-(\sigma(B)) \\
&= \matzz{ C^{-1}\sigma(C)\sigma(G) - D^{-1}\sigma(D)\sigma(B)A }{ C^{-1}\sigma(C)\sigma(A) - D^{-1}\sigma(D)A }{ D^{-1}\sigma(D)G\sigma(B) - C^{-1}\sigma(C)B \sigma(G) }{ D^{-1}\sigma(D)G - C^{-1}\sigma(C)B\sigma(A) }. \nonumber
\end{eqnarray}

\noindent We have to show that $B,C,D$ resp. $A$ are $\sigma$-stable $\!\!\!\! \mod u^{m+1}$ resp. $\!\!\!\! \mod u^m$. If $\lengthofv \geq m+1$, we have $B \equiv 0 \mod u^{m+1}$ and $G \equiv 1 \mod u^{m+1}$ by assumption and part (i), and the claimed equivalence is trivial. Assume $m \geq \lengthofv$. By assumption and as $\lengthofv \geq m+1 - \lengthofv > 0$, we know that

\begin{eqnarray*}
G &=& 1 + AB \equiv 1+ u^{\lengthofv} C\tau(C)^{-1}A \mod u^{m+1} \\ 
B &=& u^{\lengthofv} C \tau(C)^{-1} \equiv 0 \, \mod u^{\lengthofv},
\end{eqnarray*}

\noindent and we deduce from \eqref{eq:sigma_cond_rat_matr}

\begin{align}
C^{-1} \sigma(C) \sigma(1 + u^{\lengthofv} C \tau(C)^{-1} A) &\equiv 1 + D^{-1}\sigma(D)u^{\lengthofv}\sigma(C\tau(C)^{-1})A \mod u^{m+1} \label{eq:sigma_cond_rat_equations_1} \\
C^{-1} \sigma(C) \sigma(A) &\equiv D^{-1}\sigma(D) A \mod u^m \label{eq:sigma_cond_rat_equations_2} \\
D^{-1} \sigma(D) (1 + u^{\lengthofv} C \tau(C)^{-1} A ) &\equiv 1 + C^{-1} \sigma(C) u^{\lengthofv} C \tau(C)^{-1} \sigma(A) \mod u^{m+1} \label{eq:sigma_cond_rat_equations_3} 
\end{align}
\begin{align} 
D^{-1} \sigma(D) (1 + u^{\lengthofv} C  &\tau(C)^{-1} A ) \sigma( u^{\lengthofv} C\tau(C)^{-1}) \equiv \dots \nonumber \\
\label{eq:sigma_cond_rat_equations_4}	 &\dots \equiv C^{-1}\sigma(C) u^{\lengthofv} C \tau(C)^{-1} (1 + u^{\lengthofv} \sigma(C)\sigma(\tau(C))^{-1}\sigma(A)) \mod u^{m+1}. 
\end{align}

\noindent From \eqref{eq:sigma_cond_rat_equations_1}, \eqref{eq:sigma_cond_rat_equations_3} and $m \geq \lengthofv$, we deduce $C \equiv \sigma(C) \mod u^{\lengthofv}$ and $D \equiv \sigma(D) \mod u^{\lengthofv}$. Using this and $m \geq \lengthofv$, we deduce from \eqref{eq:sigma_cond_rat_equations_2} that $A \equiv \sigma(A) \mod u^{\lengthofv}$. Using these congruences and $\lengthofv \geq m+1 - \lengthofv > 0$, we may replace $\sigma(A),\sigma(C),\sigma(D)$ by $A,C,D$ in all terms which are $\equiv 0 \mod u^{\lengthofv}$ in equations \eqref{eq:sigma_cond_rat_equations_1}-\eqref{eq:sigma_cond_rat_equations_4}. Then \eqref{eq:sigma_cond_rat_equations_1} simplifies to $C^{-1}\sigma(C) \equiv 1  \mod u^{m+1}$ and \eqref{eq:sigma_cond_rat_equations_3} to $D^{-1} \sigma(D) \equiv 1 \mod u^{m+1}$. Using this, we deduce $\sigma(A) \equiv A \mod u^m$ from equation \eqref{eq:sigma_cond_rat_equations_2}. The $\sigma$-stability of $B$ follows by assumption and \eqref{eq:tau_condition_level_m/2}. This finishes the proof of the lemma. \qedhere
\end{proof}

\begin{rem} \label{rem:forget_unram_and_comps}\mbox{}
\begin{itemize}
\item[(i)] Lemma \ref{lm:ADLV_ram_GL_2_key_computation}(ii) shows, that one could have started directly with $E/F$ and $\Sigma = \{\tau\}$, instead of $\breve{E}/F$ and $\Sigma = \{\sigma, \tau\}$ as in the text, to obtain the same results. However, the approach in the text seems to the author to be more flexible.
\item[(ii)] The computations in the proof of Lemma \ref{lm:ADLV_ram_GL_2_key_computation} get significantly simpler under the stronger assumption $\lengthofv \geq m + 1$. 
However, it is the 'hardest' case $m = 2\lengthofv - 1$ of this theorem, which is necessary to realize the automorphic induction in a pure way, see Theorems \ref{thm:hard_version_main_result},  \ref{thm:small_level_case_morph_determining}. 
\end{itemize}
\end{rem}

\begin{lm}\label{lm:tildeYwm_is_EUJ_stable}
$\tilde{Y}_{\dot{w}}^m$ is stable under the left action of $E^{\times}U_{\fJ}$.
\end{lm}

\begin{proof} As $Y_{\dot{w}}^m$ is $U_{\fJ}$-stable (see the proof of Theorem \ref{thm:structure_thm_GL_2_ram}), $\tilde{Y}_{\dot{w}}^m$ also is. As $E^{\times}U_{\fJ}$ is generated by $U_{\fJ}$ and $\varpi$ ($\varpi$ as in Section \ref{sec:subgroups_of_GF}), it is enough to show $\varpi\tilde{Y}_{\dot{w}}^m = \tilde{Y}_{\dot{w}}^m$, which in turn follows from $\varpi Y_{\dot{w}}^m = Y_{\dot{w}}^m y_1$. Let $\pr_m \colon \cF^m \tar \cF$ denote the natural projection. Lemma \ref{lm:action_of_beta_varpi} shows $\varpi D_w^{\tau} = D_w^{\tau} y_1$. Using $\varpi$-(resp. $y_1$-)equivariance of $\pr_m$ and $\varpi$-(resp. $y_1$-)invariance of $X^m_{\underline{w}_m}(1)$, we deduce from this 

\begin{eqnarray*}
\varpi Y_{\dot{w}}^m &=& \varpi (\pr_m^{-1}(D_w^{\tau}) \cap X^m_{\underline{w}_m}(1)) = \pr_m^{-1}(\varpi D_w^{\tau}) \cap X^m_{\underline{w}_m}(1) = \pr_m^{-1}(D_w^{\tau}y_1) \cap X^m_{\underline{w}_m}(1) \\ &=& (\pr_m^{-1}(D_w^{\tau}) \cap X^m_{\underline{w}_m}(1))y_1 = Y_{\dot{w}}^m y_1. \qedhere
\end{eqnarray*}
%
%
%
\end{proof}

\begin{lm}\label{lm:action_of_beta_varpi} Let $\psi_{\dot{v}}(a) \in C_v$ be a point. Write $a = u a^{\prime}$ and assume that $v_u(a^{\prime}) = 0$. The point $\varpi \psi_{\dot{v}}(a) y_1^{-1}$ of $\cF$ (with $y_1$ as in Definition \ref{def:disc_subschemes_Yvm2}) lies in $C_v$. Moreover,
\[ \varpi \psi_{\dot{v}}(a) y_1^{-1} = \psi_{\dot{v}}(u a^{\prime, -1}). \]
\end{lm}

\begin{proof}
A computation shows that the $I$-cosets $\varpi e_-(a) \dot{v} y_1^{-1} I$ and $e_-(u a^{\prime,-1})I$ coincide.
\end{proof}


\section{Representation Theory}\label{sec:repth1}

Recall that $G = \GL_2$ and $\charac k \neq 2$. We use the notation from Section \ref{sec:GL_2_expl_computations}. Further, we fix a prime $\ell \neq \charac k$. All representations considered below are smooth $\overline{\bQ}_{\ell}$-representations.

\subsection{Some preparations} \mbox{}

\subsubsection{Filtrations on $U_{\fJ}$ and $U_E$}
Recall the $\caO_F$-algebra $\fJ$ from Section \ref{sec:subgroups_of_GF}. Then 

\[U_{\fJ}^{\lengthofv} := 1+ \varpi^{\lengthofv} \fJ = \matzz{1 + \fp_F^{\lfloor \frac{\lengthofv + 1}{2} \rfloor}}{\fp_F^{\lfloor \frac{\lengthofv}{2} \rfloor}}{\fp_F^{\lfloor \frac{\lengthofv}{2} \rfloor + 1}}{1 + \fp_F^{\lfloor \frac{\lengthofv + 1}{2} \rfloor}} \]

\noindent for $\lengthofv \geq 0$ form a filtration of $U_{\fJ}^0 := U_{\fJ}$ by open subgroups. Moreover, for $\lengthofv \geq 0$, we denote by $U_E^{\lengthofv}$ the $\lengthofv$-units of $E$. Note that via $\iota$ we have $U_{\fJ}^{\lengthofv} \cap E^{\times} = U_E^{\lengthofv}$.


\subsubsection{Some notation} \label{sec:some_charac_notations} For a locally compact group $H$, we denote by $H^{\vee}$ the set of all smooth $\overline{\bQ}_{\ell}^{\times}$-valued characters of $H$. For an additive character $\psi$ of $F$, we let $\psi_E := \psi\circ \tr_{E/F}$ be the corresponding character of $E$, where $\tr_{E/F}$ is the trace of $E/F$. Let $\fM := M_2(\caO_F)$. We denote by $\psi_{\fM} := \psi \circ \tr_{\fM}$ the corresponding character of $\fM$. For a character $\phi$ of $F^{\times}$ we set $\phi_E := \phi \circ \N_{E/F}$ be the corresponding character of $E^{\times}$, where $\N_{E/F}$ is the norm of $E/F$. For a $G(F)$-representation $\pi$ we denote by $\phi\pi$ the $G(F)$-representation $g \mapsto \phi(\det(g))\pi(g)$. 

\subsubsection{Characters of $U_{\fJ}$} Let $\psi$ be an additive character of $F$ of level $1$ (i.e., $\psi(\fp_F) = 1$, but $\psi$ non-trivial on $\caO_F$). Let $0 \leq k < r \leq 2k+1$ be integers. By \cite{BH} 12.5 Proposition we have isomorphisms

\begin{equation} \label{eq:characs_of_UJ_fromBH}
\varpi^{-r}\fJ/\varpi^{-k}\fJ \stackrel{\sim}{\longrar} (U_{\fJ}^{k+1}/U_{\fJ}^{r+1})^{\vee}, \qquad a + \varpi^{-k}\fJ \mapsto \psi_{\fM,a}|_{U_{\fJ}^{m+1}}, 
\end{equation}

\noindent where $\psi_{\fM,a}$ denotes the function $x \mapsto \psi_{\fM}(a(x-1))$ and $\fM$ is as in Section \ref{sec:some_charac_notations}.


\subsubsection{Admissible pairs} Let $\chi$ be a character of $E^{\times}$. The \emph{level} $\ell(\chi)$ of $\chi$ is the least integer $m \geq 0$, such that $\chi|_{U_E^{m+1}}$ is trivial. The pair $(E/F, \chi)$ is said to be \emph{admissible} (\cite{BH} 18.2) if  $\chi|_{U_E^1}$ does not factor through the norm map $\N_{E/F}$. An admissible pair $(E/F, \chi)$ with $\chi$ of level $m$ is called \emph{minimal}, if $\chi|_{U_E^m}$ does not factor through $\N_{E/F}$. Note that if $(E/F, \chi)$ is minimal, then $\ell(\chi)$ is odd. Two pairs $(E/F,\chi)$, $(E/F,\chi^{\prime})$ are said to be $F$-isomorphic if there is some $\gamma \in \Gal_{E/F}$ such that $\chi^{\prime} = \chi \circ \gamma$. We denote by $\bP_2^{\rm tr}(F)$ the set of isomorphism classes of all admissible pairs attached to the tamely ramified extension $E/F$.


\subsubsection{Supercuspidal representations} Denote by $\cA_2^{\rm tr}(F)$ the set of all isomorphism classes of irreducible supercuspidal representations of $G(F)$, which are not unramified (i.e., are not attached to an unramified stratum. We use the definition of \emph{unramified} from \cite{BH} 20.1, see also 20.3 Lemma). The ramified part of the tame parametrization theorem (\cite{BH} 20.2 Theorem) states the existence of a certain bijection

\begin{equation} \label{eq:ramfied_part_of_tame_param_thm}
\pi \colon \bP_2^{\rm tr}(F) \stackrel{\sim}{\rar} \cA_2^{\rm tr}(F) \qquad (E/F, \chi) \mapsto \pi_{\chi}. 
\end{equation}


\subsubsection{Bushnell-Henniart construction of $\pi_{\chi}$}\label{sec:BH_constr_of_pi_chi} We recall the construction of $\pi_{\chi}$ from \cite{BH}\S 15,19. By twisting with a character of $F^{\times}$, it is enough to construct $\pi_{\chi}$ for minimal pairs. Fix an additive character $\psi$ of $F$ of level one. Let $(E/F,\chi)$ be a minimal admissible pair with $\chi$ of odd level $m = 2\lengthofv - 1 \geq 1$. Choose an element $\beta \in \fp_E^{-m}$ such that 

\begin{equation}\label{eq:BH_choice_identity_between_chi_and_psi} 
\chi(1 + x) = \psi_E(\beta x) \quad \text{for all $x \in \fp_E^{\lengthofv}$.} 
\end{equation}

\noindent Via $\iota$ we see $\beta$ as an element of $M_2(F)$. Then $(\fJ,m,\beta)$ is a \emph{ramified simple stratum} (see \cite{BH} 13.1). Via \eqref{eq:characs_of_UJ_fromBH}, $\beta$ defines a character $\psi_{\beta}$ of $U_{\fJ}^\lengthofv$, which is trivial on $U_{\fJ}^{m+1}$. Let $\Lambda$ be the character of $J_{\beta} := E^{\times}U_{\fJ}^{\lengthofv}$ defined by

\[ \Lambda|_{U_{\fJ}^{\lengthofv}} := \psi_{\beta}, \quad \Lambda|_{E^{\times}} := \chi \]

\noindent (by \eqref{eq:BH_choice_identity_between_chi_and_psi} this is a consistent definition, as $\tr_{\fM}|_{\caO_{E}} = \tr_{E/F}|_{\caO_E}$ and $E \cap U_{\fJ}^{\lengthofv} = U_E^{\lengthofv}$). Then $(\fJ,J_{\beta},\Lambda)$ is a \emph{cuspidal type} in $G(F)$ attached to $(E/F,\chi)$ (see \cite{BH} 15.5). The \emph{cuspidal inducing datum} attached to this cuspidal type is the pair $(U_{\fJ}, \Theta_{\chi})$, where $\Theta_{\chi} := \cIndd_{J_{\beta}}^{E^{\times}U_{\fJ}} \Lambda$. Then $\pi_{\chi}$ is defined to be the compact induction 

\[ \pi_{\chi} := \cIndd\nolimits_{J_{\beta}}^{G(F)} \Lambda = \cIndd\nolimits_{E^{\times}U_{\fJ}}^{G(F)} \Theta_{\chi}. \]

\noindent The isomorphism class of $\pi_{\chi}$ is independent of the choices of $\iota$, $\psi$ and $\beta$. We work with the fixed choice of $\iota$, but $\psi$ and $\beta$ can be arbitrary.


\subsubsection{Cohomology} For a scheme $X$ over $k$ we denote by $\coh_c^{\ast}(X,\overline{\bQ}_{\ell})$ the $\ell$-adic cohomology of $X$ with compact support.


\subsection{Automorphic induction from the ramified torus of $\GL_2$} \mbox{}

Let $m \geq 1$ be an odd integer. Let $\chi$ be character of $E^{\times}$ of level $m$. Let $\underline{w}_m$ be as in Notation \ref{notat:higher_level_w}. By inflation via \eqref{eq:right_group_and_mod_rad_surjection}, $\chi$ determines a character of $\tilde{I}_{m, \underline{w}_m}/I^m$ and hence we can consider the $\chi$-isotypic subspace $\coh_c^{\ast}(X_{\underline{w}_m}^m(1),\overline{\bQ_{\ell}})[\chi]$ of the cohomology of $X_{\underline{w}_m}^m(1)$. Analogously, we can consider the $\chi$-isotypic subspace in the cohomology of $\tilde{Y}_{\dot{w}}^m$.

\begin{Def}\label{Def:R_chi}
Let $(E/F, \chi)$ be a minimal pair of odd level $m \geq 1$. Let $w$, $\lengthofv$, $\underline{w}$ be as in Notation \ref{Def:D_v_tau} such that $\ell(w) = 2\lengthofv - 1 \geq m$ and take $\underline{w}_m$ as in Notation \ref{notat:higher_level_w} lying over $\underline{w}$. Define $R_{\chi,\lengthofv}$ to be the $G(F)$-representation

\[R_{\chi,\lengthofv} := \coh_c^0(X_{\underline{w}_m}^m(1), \overline{\bQ_{\ell}})[\chi] \]

\noindent and $\Xi_{\chi,\lengthofv}$ to be the $E^{\times}U_{\fJ}$-representation 

\[ \Xi_{\chi,\lengthofv} := \coh_c^0(\tilde{Y}_{\dot{w}}^m, \overline{\bQ_{\ell}})[\chi]. \] 

\noindent For an arbitrary admissible pair $(E/F, \chi)$ such that $\chi = \phi \chi^{\prime}$ with $(E/F, \chi^{\prime})$ minimal we define $R_{\chi,\lengthofv} := \phi R_{\chi^{\prime},\lengthofv}$, $\Xi_{\chi,\lengthofv} := \phi \Xi_{\chi^{\prime},\lengthofv}$. If $m = 2\lengthofv - 1$, write

\[R_{\chi} := R_{\chi,\lengthofv} \quad \text{ and } \quad \Xi_{\chi} := \Xi_{\chi,\lengthofv}. \]

We also denote by $V_{\chi}$ the space in which $\Xi_{\chi}$ acts.
\end{Def}

As $X_{\underline{w}_m}^m(1)$ is zero-dimensional, its cohomology in all positive degrees vanishes, and Definition \ref{Def:R_chi} is compatible with \eqref{eq:adlv_induction_map}. We suppose that $R_{\chi}$ for non-minimal pairs occurs also naturally in the zeroth cohomology of $X_{\underline{w}_m}^m(1)$ with $m$ even. The following theorem is our main result. 

\begin{thm}\label{thm:hard_version_main_result} 
Let $(E/F,\chi)$ be an admissible pair. The representation $R_{\chi}$ is irreducible, cuspidal, ramified, has level $\ell(\chi)$ and central character $\chi|_{F^{\times}}$. Moreover, $R_{\chi}$ is isomorphic to $\pi_{\chi}$, i.e., the map

\begin{equation} \label{eq:R_ramified_is_bijection}
R \colon \bP_2^{\rm tr}(F) \rar \cA_2^{\rm tr}(F) \qquad (E/F, \chi) \mapsto R_{\chi} 
\end{equation}

\noindent coincides with the map $\pi_{\chi}$ from \eqref{eq:ramfied_part_of_tame_param_thm} and is, in particular, a bijection.
\end{thm}

After necessary preparations, Theorem \ref{thm:hard_version_main_result} is shown in Sections \ref{sec:cuspidality}, \ref{sec:rel_to_cusp_ind_data}. We wish to point out, that the injectivity of \eqref{eq:R_ramified_is_bijection} follows from the results of Section \ref{sec:restr_to_units_of_torus} and essentially does not use cuspidal types and the isomorphism $R_{\chi} \cong \pi_{\chi}$. We need them to prove surjectivity of \eqref{eq:R_ramified_is_bijection}. From Theorem \ref{thm:structure_thm_GL_2_ram} we deduce:
 
\begin{lm} \label{eq:Rchi_is_induction_from_ZUJ_of_Xichi}  Let $(E/F,\chi)$ be an admissible pair. Then
\begin{equation*}
R_{\chi,\lengthofv} = \cIndd\nolimits_{E^{\times} U_{\fJ}}^{G(F)} \Xi_{\chi,\lengthofv}. 
\end{equation*}
\end{lm}
\begin{proof}
It follows from Theorem \ref{thm:structure_thm_GL_2_ram} and the commutativity of the left and the right group actions on $X_{\underline{w}_m}^m(1)$.
\end{proof}

In Section \ref{sec:proof_of_thm_small_level_case} we also study the representations $R_{\chi,\lengthofv}$ for $\lengthofv \geq m+1$, where $m$ is the (odd) level of $\chi$. We determine the structure of $R_{\chi,\lengthofv}$ and give a recipe how to reconstruct $\chi$ (up to $\tau$-conjugacy) from $R_{\chi,\lengthofv}$.


\subsection{Unipotent traces}\label{sec:unipotent_traces} \mbox{}

From now on and until the end of Section \ref{sec:rel_to_cusp_ind_data} we assume $2 \lengthofv - 1 = m$. 

\begin{lm}\label{lm:UJm1_acts_trivial}
The central character of $R_{\chi}$ is $\chi|_{F^{\times}}$. The subgroup $U_{\fJ}^{m+1}$ acts trivially in $V_{\chi}$ and $V_{\chi}$ has dimension $(q-1)q^{\lengthofv - 1}$. 
\end{lm}

\begin{proof}
Elements of $F^{\times}$ act on $X_{\underline{w}_m}^m(1)$ in the same way from the left and from the right. As $R_{\chi}$ is the $\chi$-isotypic component of $\coh_c^0(X_{\underline{w}_m}^m(1),\overline{\bQ}_{\ell})$, the first statement of the lemma follows. The proof of the second statement is given in Section \ref{sec:computations_of_unipotent_traces}.
\end{proof}

By Lemma \ref{lm:UJm1_acts_trivial} we can consider $\Xi_{\chi}$ as a $E^{\times}U_{\fJ}/U_{\fJ}^{m+1}$-representation. Let $N_{\lengthofv}$ be the finite subgroup of $E^{\times}U_{\fJ}/U_{\fJ}^{m + 1}$, equipped with a descending filtration by subgroups $N_{\lengthofv}^i$ for $1 \leq i \leq \lengthofv + 1$ defined by

\[ N_{\lengthofv}^i := \matzz{1}{}{\fp_F^{i}}{1} \Bigg/ \matzz{1}{}{\fp_F^{\lengthofv + 1}}{1} \subseteq N_{\lengthofv} := \matzz{1}{}{\fp_F}{1} \Bigg/ \matzz{1}{}{\fp_F^{\lengthofv + 1}}{1}  \subseteq U_{\fJ}/U_{\fJ}^{m + 1}. \]

\begin{prop}\label{prop:unip_rep}
As $N_{\lengthofv}$-representations one has
\[ \Xi_{\chi} \cong \Ind\nolimits_1^{N_{\lengthofv}} 1 - \Ind\nolimits_{N_{\lengthofv}^{\lengthofv}}^{N_{\lengthofv}} 1 \cong \bigoplus_{\substack{\psi \in N_{\lengthofv}^{\vee} \\ \psi|_{N_{\lengthofv}^{\lengthofv}} \text{non-trivial} }} \psi. \]

\noindent In particular, $\Xi_{\chi}$ does not contain the trivial character on $N_{\lengthofv}^{\lengthofv}$.

\end{prop}

\begin{proof} The proposition follows from Lemma \ref{eq:claim_for_unip_traces} by comparing the traces of $N_{\lengthofv}$-representations on the left and the right side. \end{proof}

\begin{lm}\label{eq:claim_for_unip_traces}
For $g \in N_{\lengthofv}$ we have: 
\begin{equation} 
\tr(g; \Xi_{\chi}) = \begin{cases} (q-1)q^{\lengthofv - 1} & \text{if $g = 1$} \\ -q^{\lengthofv - 1} & \text{if $g \in N_{\lengthofv}^{\lengthofv} \sm \{1\}$} \\ 0 & \text{if $g \in N_{\lengthofv} \sm N_{\lengthofv}^{\lengthofv}$.} \end{cases} 
\end{equation}
\end{lm}

\begin{proof}
The proof is given in Section \ref{sec:computations_of_unipotent_traces}. 
\end{proof}

\begin{cor}\label{cor:irreducibility_of_XI_chi}
$\Xi_{\chi}$ is irreducible as $B$-representation, where $B \subseteq U_{\fJ}$ is the subgroup consisting of lower triangular matrices. 
\end{cor}

\begin{proof}
The proof is the same as the proof of \cite{Iv2} Corollary 4.12 (using Proposition \ref{prop:unip_rep} instead of \cite{Iv2} Proposition 4.10).
\end{proof}


\subsection{Some character theory} \label{sec:some_character_theory} \mbox{}

In this section we work relative to a fixed character $\chi$ of $E^{\times}$ of odd level $m \geq 1$. We write $\chi^{\tau} := \chi \circ \tau$.

\subsubsection{Admissibility of $(E/F,\chi)$}  

\begin{lm}\label{lm:non_triv_same_as_norm_fact}
The following hold: 
\begin{itemize}
\item[(i)]   $\chi|_{U_E^m}$ does not factor through the norm $N_{E/F}$.
\item[(ii)] $\chi|_{U_E^m} \neq \chi^{\tau}|_{U_E^m}$.
\end{itemize}
\end{lm}

\begin{proof}
First we show (ii): Assume $\chi(1 + u^m x) = \chi(1 - u^m x)$ for all $x \in k$. As $(1 - u^m x)^{-1} \in (1+u^m x)U_E^{2m}$ and as $\chi$ has level $m \geq 1$, we deduce $1 = \chi((1 + u^m x)^2) = \chi(1+ u^m 2x)$ for all $x \in k$. As $\charac E \neq 2$, we obtain a conradiction to our assumption $\ell(\chi) = m$. Now we deduce (i) from (ii): assume that $\chi|_{U_E^m}$ factors through the norm, i.e., $\chi = \chi^{\prime} \circ N_{E/F}$ on $U_E^m$. Then $\chi^{\tau}(x) = \chi^{\prime}(\N_{E/F}((\tau(x))) = \chi^{\prime}(N_{E/F}(x)) = \chi(x)$, which contradicts (ii).
\end{proof}

\subsubsection{Filtration on $U_E$} \label{sec:subsub_filtration_on_UE}
We have the disjoint decomposition

\[ U_E = U_F U_E^{m+1} \cup \bigcup_{\alpha=0}^{\lengthofv-1} (U_F U_E^{2\alpha + 1} \sm  U_F U_E^{2\alpha + 3}).  \]


\noindent Note that $U_FU_E^{2\alpha + 1} = U_FU_E^{2\alpha}$. 


\subsubsection{Index of coincidence for characters}

\begin{Def}
For a character $\theta$ of $E^{\times}$, which coincides with $\chi$ on $F^{\times}U_E^{m+1}$, we define the integer $i(\theta) = i_{\chi}(\theta)$ to be the smallest integer $i \geq 0$, such that $\theta|_{F^{\times} U_E^i} = \chi|_{F^{\times} U_E^i}$ or $\theta|_{F^{\times} U_E^i} = \chi^{\tau}|_{F^{\times} U_E^i}$. 
\end{Def}

Observe that $0 \leq i(\theta) \leq m+1$ and $i(\theta)$ is always even. 


\subsubsection{Modifications of characters} Fix some integer $0 \leq \alpha < \lengthofv$. Consider the $k$-algebra 

\[R_{\alpha} := \caO_E/\fp_E^{m-2\alpha} = k[u]/u^{m- 2\alpha}.\] 

\noindent The $\tau$-invariants of it are $R_{\alpha}^{\langle \tau \rangle} = k[t]/t^{\lengthofv - \alpha}$. Consider the subset 

\[ R_{\alpha}^{\langle \tau \rangle, \prime} := \{ s \in R_{\alpha}^{\langle \tau \rangle} \colon s \equiv \pm 1 \mod u^{m + 1 - 2(2\alpha + 1)} \} \]

\noindent of $R_{\alpha}^{\langle \tau \rangle}$ (note that $R_{\alpha}^{\langle \tau \rangle, \prime}  = R_{\alpha}^{\langle \tau \rangle}$ if $2 \alpha + 1 \geq \lengthofv$, or equivalently, $\alpha \geq \lfloor \frac{\lengthofv}{2} \rfloor$).

\begin{prop}\label{lm:description_of_special_characters} 
Let $0 \leq \alpha < \lengthofv$. Let $s \in R_{\alpha}^{\langle \tau \rangle, \prime}$. There is a unique character $\chi_s$ of $F^{\times} U_E^{2\alpha + 1}$, such that the following hold:
\begin{itemize}
\item[(i)] $\chi_s$ coincides with $\chi$ on $F^{\times} U_E^{m+1}$.
\item[(ii)] if $\alpha < \lfloor \frac{\lengthofv}{2} \rfloor$, then $\chi_s$ coincides on $F^{\times} U_E^{2(2\alpha + 1)}$ with $\chi \circ \tau^i$, where $s \equiv (-1)^i \mod u$.
\item[(iii)] $\chi_s(1 + u^{2\alpha + 1}h) = \chi(1 + u^{2\alpha + 1}hs)$ for all $h \in \caO_F$.
\end{itemize}
Conversely, let $\theta$ be a character of $F^{\times} U_E^{2\alpha + 1}$, which coincides with $\chi$ or $\chi^{\tau}$ on $F^{\times} U_E^{\min\{ m+1, 2(2\alpha + 1)\}}$. Then there is a unique $s \in R_{\alpha}^{\langle \tau \rangle, \prime}$ such that $\theta = \chi_s$.
\end{prop}

Note that the expression $\chi(1 + u^{2\alpha + 1}hs)$ in (iii) is well-defined: Indeed, $\chi$ is trivial on $U_E^{m+1}$, and on the other hand if $\tilde{s}_1$, $\tilde{s}_2 \in \caO_F = k\llbracket t \rrbracket \subseteq k\llbracket u \rrbracket$ represent the same element $s$ in $R_{\alpha}^{\langle \tau \rangle,\prime}$, then $\tilde{s}_1 \equiv \tilde{s}_2 \mod u^{m - 2\alpha}$, hence $1 + u^{2\alpha + 1}h\tilde{s}_1 \equiv 1 + u^{2\alpha + 1}h\tilde{s}_2 \mod u^{m+1}$.

\begin{proof} 
Consider the subset

\[ U_E^{2\alpha + 1, \prime} := \{ x \in U_E^{2\alpha + 1} \colon \exists h \in \caO_F \text{ with } x \equiv 1 + u^{2\alpha + 1}h \mod U_E^{m+1} \} \subseteq U_E^{2\alpha + 1}. \]

\begin{lm}\label{sublm:zerl_in_UFUE2alpha1} 
Any element $x \in F^{\times} U_E^{2\alpha + 1}$ can be written as $x = z_x x^{\prime}$ with $z_x \in U_F$, $x^{\prime} \in U_E^{2\alpha + 1, \prime}$. Moreover, modulo $U_E^{m+1}$, $z_x, x^{\prime}$ are uniquely determined by $x$ and if $x = \sum_{i \geq 0} x_i u^i \in U_E^{2\alpha + 1}$, then $z_x \equiv \sum_{i \geq 0} x_{2i} u^{2i} \mod u^{m+1}$.
\end{lm}

\begin{proof}
Multiplying by an element in $F^{\times}$, we can assume $x \in U_E^{2\alpha + 1}$. Write $x = 1 + \sum_{i = 2\alpha + 1}^m x_i u^i + \cO(u^{m+1})$. As $x^{\prime}$ has to lie in $U_E^{2\alpha + 1,\prime} \subseteq U_E^{2\alpha + 1}$, also $z_x$ must lie in $U_E^{2\alpha + 1}$. Thus we  seek for two elements $z_x := 1 + \sum_{i = \alpha + 1}^{\lengthofv - 1} z_{2i}u^{2i} + \cO(u^{m+1})$ and $x^{\prime} = 1 + \sum_{i = \alpha}^{\lengthofv - 1} y_{2i + 1} u^{2i + 1} + \cO(u^{m+1})$ which have to satisfy 

\[ \left(1 + \sum_{i = \alpha + 1}^{\lengthofv - 1} z_{2i}u^{2i} + \cO(u^{m+1})\right)\left(1 + \sum_{i = \alpha}^{\lengthofv - 1} y_{2i + 1} u^{2i + 1} + \cO(u^{m+1})\right) = 1 + \sum_{i = 2\alpha + 1}^m x_i u^i + \cO(u^{m+1}). \]

\noindent Comparing the parity of the degrees we see that $z_{2i} = x_{2i}$. Further, a computation shows that $y_i$'s satisfying this equation exist and are uniquely determined by the $x_i$'s.
\end{proof}

Let now $s \in R_{\alpha}^{\langle \tau \rangle,\prime}$. For $x \in F^{\times} U_E^{2\alpha + 1}$ with decomposition $x = z_x x^{\prime}$ according to Lemma \ref{sublm:zerl_in_UFUE2alpha1}, set 
\[ \chi_s(x) := \chi(z_x)\chi(1 + u^{2\alpha + 1}h) \quad \text{where $x^{\prime} = 1 + u^{2\alpha + 1}h$ with $h \in \caO_F$}. \]
We show that $\chi_s$ is a character of $F^{\times} U_E^{2\alpha + 1}$. Let $x,y \in F^{\times} U_E^{2\alpha + 1}$ with decompositions $x = z_x x^{\prime}$, $y = z_y y^{\prime}$ as in Lemma \ref{sublm:zerl_in_UFUE2alpha1} and let $x^{\prime} = 1 + u^{2\alpha + 1}h_x$, $y^{\prime} = 1 + u^{2\alpha + 1}h_y$ (up to some elements in $U_E^{m+1}$). Write $A := u^{2\alpha + 1}(h_x + h_y)$, $B := u^{2(2\alpha + 1)}h_x h_y$. We compute

\begin{eqnarray*}
\chi_s(x)\chi_s(y) &=& \chi(z_x z_y (1 + u^{2\alpha + 1}sh_x) (1 + u^{2\alpha + 1}sh_y)) \\ 
&=& \chi(z_x z_y (1 + sA + s^2 B )) \\
&=& \chi(z_x z_y (1 + sA + B)),
\end{eqnarray*}

\noindent the last equation being true, as $s^2 \equiv 1 \mod u^{m+1 - 2(2\alpha + 1)}$. We have

\[ x^{\prime}y^{\prime} = 1 + A + B. \]

\noindent As $h_x,h_y \in \caO_F$, Lemma \ref{sublm:zerl_in_UFUE2alpha1} implies $z_{x^{\prime}y^{\prime}} = 1 + B$ (up to elements in $U_E^{m+1}$). We deduce

\[ (x^{\prime}y^{\prime})^{\prime} = x^{\prime}y^{\prime} z_{x^{\prime}y^{\prime}}^{-1} = 1 + A - AB (1 + B)^{-1}. \]

\noindent Now, $xy = z_x z_y z_{x^{\prime}y^{\prime}} (x^{\prime}y^{\prime})^{\prime}$ is the decomposition of $xy$ according to Lemma \ref{sublm:zerl_in_UFUE2alpha1} and we compute

\[ \chi_s(xy) = \chi(z_xz_yz_{x^{\prime}y^{\prime}}) \chi(1 + s (A - AB(1+B)^{-1})). \]

\noindent If $2\alpha + 1 \geq \lengthofv$, we have $B \in \cO(u^{m+1})$, hence all terms containing $B$ can be ignored and we deduce $\chi_s(x)\chi_s(y) = \chi_s(xy)$. Assume $2\alpha + 1 < \lengthofv$. Let $\sgn(s) := \pm 1$, if $s \equiv \pm 1 \mod u$. From the above, $s \equiv (-1)^{\sgn(s)} \mod u^{m+1 - 2(2\alpha + 1)}$ and $B \equiv 0 \mod u^{2(2\alpha + 1)}$ we deduce

\begin{eqnarray*} 
\chi_s(xy) &=& \chi(z_xz_y(1 + B)(1 + sA - (-1)^{\sgn(s)}AB(1+B)^{-1})) \\
&=& \chi(z_xz_y(1 + sA + B + sAB - (-1)^{\sgn(s)}AB)) \\
&=& \chi(z_xz_y(1 + sA + B)).
\end{eqnarray*}

\noindent This shows that $\chi_s$ is a character. Now, $\chi_s$ satisfies (i) and (iii) by definition and (ii) is immediate. This finishes the proof of the first part of the proposition. For the  converse statement, one shows by a simple computation that the map $s \mapsto \chi_s$ from $R_{\alpha}^{\langle \tau \rangle, \prime}$ to characters of $F^{\times} U_E^{2\alpha + 1}$ is injective. This completes the proof, as the number of elements in $R_{\alpha}^{\langle \tau \rangle, \prime}$ coincides with the number of characters $\theta$ of $F^{\times} U_E^{2\alpha + 1}$, which are equal to $\chi$ or $\chi^{\tau}$ on $F^{\times} U_E^{\min\{m+1,2(2\alpha + 1)\}}$ (if $2\alpha + 1 \geq \lengthofv$, then there are $q^{\lengthofv - \alpha}$ those, otherwise there are $2q^{\alpha + 1}$). \qedhere
\end{proof}


\subsubsection{Compatibility with changing $\alpha$}
Let $0 \leq \alpha < n$. Let $\theta$ be a character of $E^{\times}$, coinciding on $F^{\times}U_E^{\min \{ m+1,2(2\alpha + 1)\} }$ with $\chi$ or $\chi^{\tau}$. By Proposition \ref{lm:description_of_special_characters}, there is some $s(\theta,\alpha) \in R_{\alpha}^{\langle \tau \rangle, \prime}$ such that $\theta|_{F^{\times} U_E^{2\alpha + 1}} = \chi_{s(\theta,\alpha)}$. This construction is compatible with changing the level $\alpha$.

\begin{lm}\label{lm:character_modifications_techn_lemma} Let $0 \leq \alpha_1 \leq \alpha_2 < \lengthofv$. Let $\theta$ be a character of $E^{\times}$ coinciding on $F^{\times} U_E^{\min\{m+1,2(2\alpha_1 + 1)\}}$ with $\chi$ or $\chi^{\tau}$. Under the natural projection $R_{\alpha_1}^{\langle \tau \rangle} \rar R_{\alpha_2}^{\langle \tau \rangle}$,  $s(\theta,\alpha_1)$ maps to $s(\theta,\alpha_2)$. 
\end{lm}

\begin{proof}
Let $\overline{s}_1$ denote the image of $s_1$ in $R_{\alpha_2}^{\langle \tau \rangle, \prime}$. On $F^{\times} U_E^{2\alpha_2 + 1}$ we have 
\[ \theta(1 + u^{2\alpha_2 + 1}h) = \chi_{s_1}(1 + u^{2\alpha_2 + 1}h) = \chi(1 + u^{2\alpha_2 + 1}\overline{s}_1 h). \]
Thus on $F^{\times} U_E^{2\alpha_2 + 1}$ we have $\chi_{s_2} = \theta = \chi_{\overline{s}_1}$. By the uniqueness statement in Proposition \ref{lm:description_of_special_characters} we have $\overline{s}_1 = s_2$. \qedhere
\end{proof}


\subsubsection{Elementary modifications and distances}

\begin{Def}\label{def:el_mod_and_dist} For $s \in R_{\alpha}^{\langle \tau \rangle, \prime}$, we call the character $\chi_s$ of $F^{\times} U_E^{2\alpha + 1}$ constructed in Proposition \ref{lm:description_of_special_characters} an \emph{elementary modification} of $\chi$. Let $\theta$ be character of $E^{\times}$ coinciding with $\chi$ on $F^{\times} U_E^{m+1}$. Set

\[ \alpha_{\theta} := \min \{ \alpha \colon 0 \leq \alpha < \lengthofv, 2(2\alpha + 1) \geq i(\theta) \}, \]

\noindent i.e., $\alpha_{\theta}$ is the smallest integer such that $\theta$ restricted to $F^{\times} U_E^{2\alpha_{\theta} + 1}$ is an elementary modification of $\chi$. We define the \emph{distance} from $\chi$ to $\theta$ to be the (uniquely determined by Proposition \ref{lm:description_of_special_characters}) element $s(\theta) := s(\theta, \alpha_{\theta}) \in R_{\alpha_{\theta}}^{\langle \tau \rangle, \prime}$, such that $\theta$ coincides on $F^{\times} U_E^{2\alpha_{\theta} + 1}$ with $\chi_{s(\theta)}$.
\end{Def}

As $i(\theta)$ is an even integer $\leq m+1 = 2\lengthofv$, it follows easily that in any case $\alpha_{\theta} \leq \lfloor \frac{\lengthofv}{2} \rfloor$. (Moreover, one has $\alpha_{\theta} = \lfloor \frac{i(\theta)}{4} \rfloor$, but we will not use this). Further, $\alpha_{\chi} = \alpha_{\chi^{\tau}} = 0$ and $s(\chi) = 1$ and $s(\chi^{\tau}) = -1$.


\subsubsection{Quadratic distance}\label{sec:quadratic_distance}

Let $0 \leq \alpha < \lengthofv$. There is the norm map 

\[ \N_{\tau,\alpha} \colon R_{\alpha} \rar R_{\alpha}^{\langle \tau \rangle} \quad s \mapsto s\tau(s). \]

\begin{lm}\label{lm:image_of_norm_on_Ralpha}
The image $\im(N_{\tau, \alpha})$ of $\N_{\tau,\alpha}$ consists of precisely such elements $s \in R_{\alpha}^{\langle \tau \rangle}$ for which \\ $(-1)^{v_t(s)} \cdot s_{v_t(s)}$ is a square in $k^{\times}$, where $s_{v_t(s)}$ denotes the leading coefficient of $s$.
\end{lm}

\begin{proof} 
This follows immediately from Lemma \ref{lm:charac_of_squares}.\qedhere
\end{proof}

\begin{lm}\label{lm:charac_of_squares}
Let $x \in k\llbracket t \rrbracket \sm \{0\}$ with leading coefficient $x_{v_t(x)} \in k^{\times}$. Then $x$ is 
\begin{itemize}
\item[-] a square of an element of $k\llbracket t \rrbracket$ if  and only if $v_t(x)$ is even and $x_{v_t(x)}$ is a square in $k^{\times}$,
\item[-] a square of an element of $k\llbracket u \rrbracket$ if  and only if $x_{v_t(x)}$ is a square in $k^{\times}$,
\item[-] in the image of the norm map $\N_{E/F}$ if and only if $(-1)^{v_t(x)}x_{v_t(x)}$ is a square in $k^{\times}$.
\end{itemize}
\end{lm}

\begin{proof}
This is well-known.\qedhere
\end{proof}

Consider the following subset of $R_{\alpha}^{\langle \tau \rangle, \prime}$:

\[ Q_{\alpha} := \{ s \in R_{\alpha}^{\langle \tau \rangle, \prime} \colon s\tau(s) - 1 \in \N_{\tau,\alpha} R_{\alpha} \} \sm \{\pm 1 \}.  \]

\begin{Def}\label{def:distance_quadratic}
Let $\theta$ be a character of $E^{\times}$ coinciding with $\chi$ on $F^{\times}U_E^{m+1}$. We say that the distance from $\chi$ to $\theta$ is \emph{properly quadratic} if $s(\theta) \in Q_{\alpha_{\theta}}$, with $s(\theta)$ as in Definition \ref{def:el_mod_and_dist}.
\end{Def}


\subsubsection{Structure of $Q_{\alpha}$}\label{sec:str_of_Q_alpha}
Set $R_n^{\langle \tau \rangle} := \{1\}$. Let $\pr_{\alpha}$ be the natural projection
\[ \pr_{\alpha} \colon R_{\alpha}^{\langle \tau \rangle} \cong k[t]/t^{\lengthofv - \alpha} \rar k[t]/t^{\lengthofv - \alpha - 1} \cong R_{\alpha + 1}^{\langle \tau \rangle}. \]

\begin{lm} \label{lm:study_of_Qalpha12_sizes_and_maps} 
Let $0 \leq \alpha \leq \lengthofv - 1$. An element $s \in R_{\alpha}^{\langle \tau \rangle} \sm \{\pm 1\}$ lies in $Q_{\alpha}$ if and only if either
\begin{itemize}
\item[-] $\alpha \geq \lfloor \frac{\lengthofv}{2} \rfloor$, $s \not\equiv \pm 1 \mod u$ and $s^2 - 1 \mod u$ is a square in $k^{\times}$, or
\item[-] $s \equiv \pm 1 \mod u$, i.e., $s = \pm 1 + t^j s_0 + \cO(t^{j+1})$ for some $s_0 \in k^{\times}$, $\max\{1, \lengthofv - (2\alpha + 1)\} \leq j \leq n - \alpha - 1$ and $\pm (-1)^j 2 s_0$ is a square in $k^{\times}$.
\end{itemize}
Moreover, the following hold:
\begin{itemize}
\item[(i)]    Let $0 \leq \alpha \leq \lengthofv - 2$. The preimage of $1$ (resp. $-1$) under the composed map $Q_{\alpha} \har R_{\alpha}^{\langle \tau \rangle} \rar R_{\alpha + 1}^{\langle \tau \rangle}$ contains precisely $\frac{q-1}{2}$ elements.
\item[(ii)]   Assume $0 \leq \alpha \leq \lengthofv - 2$. Let $s_0 \in R_{\alpha}^{\langle \tau \rangle, \prime}$ with $\pr_{\alpha}(s_0) \neq \pm 1$. Then $\sharp \pr_{\alpha}^{-1}(\pr_{\alpha}(s_0)) = q$ and we have the equivalence $s_0 \in Q_{\alpha} \LRar \pr_{\alpha}^{-1}(\pr_{\alpha}(s_0)) \subseteq Q_{\alpha}$.
\item[(iii)]  We have $\sharp Q_{\lengthofv - 1} = \frac{q-3}{2}$.
\end{itemize}
\end{lm}

\begin{proof} 
The description of $Q_{\alpha}$ follows by an easy computation from Lemma \ref{lm:charac_of_squares}. (i),(ii) follow from this description (along with $\sharp k^{\times, 2} = \frac{q-1}{2}$).

(iii): Note that $Q_{\lengthofv - 1} = \{s \in k \colon s^2 - 1 \text{ is a square in $k$} \} \sm \{\pm 1\}$. Consider the affine curve $C: s^2 - 1 = y^2$ over $k$ and let $\overline{C}$ be the unique smooth projective curve over $k$ containing $C$ as an open subset. We have $\sharp(\overline{C}(k) \sm C(k)) = 2$. Further, $\overline{C}$ is a smooth quadric in $\bP^2$ over a finite field, hence isomorphic to $\bP^1$, i.e., $\sharp \overline{C}(k) = q+1$. We deduce $\sharp C(k) = q-1$. Now (iii) follows from the fact that the map $C(k) \sm \{(\pm 1 , 0)\} \rar Q_{\lengthofv - 1}$ given by $(s,y) \mapsto s$ is surjective and two-to-one.
\end{proof}


\subsection{Restriction to the ramified torus $E^{\times} \subseteq G(F)$}\label{sec:restr_to_units_of_torus} \mbox{}

\noindent For a finite finite group $H$, let $\langle, \rangle_H$ denote the inner product on the set of class functions of $H$. For a character $\theta$ let $\langle \theta, \Xi_{\chi} \rangle_{E^{\times}}$ denote the multiplicity of $\theta$ in $\Xi_{\chi}$.

\begin{thm}\label{thm:Xi_chi_restriction_to_Etimes}
Let $(E/F, \chi)$ be a minimal pair of odd level $m \geq 1$. A character $\theta$ of $E^{\times}$ can only occur in $\Xi_{\chi}$, if $\theta$ coincides with $\chi$ on $F^{\times}U_E^{m+1}$. In this case we have

\[ \langle \theta, \Xi_{\chi}\rangle_{E^{\times}} = \begin{cases} 1 & \text{if $\theta = \chi$ or $\theta = \chi^{\tau}$ or the distance from $\chi$ to $\theta$ is properly quadratic} \\ 0 & \text{otherwise.} \end{cases} \]
\end{thm}

We prove this theorem below. First we investigate the restriction of $\Xi_{\chi}$ to $U_E$. Note that $s(\theta)$ from Definition \ref{def:el_mod_and_dist} is in exactly the same way also defined for characters $\theta$ of $U_E$, which coincide with $\chi$ on $U_F U_E^{m+1}$.

\begin{prop}\label{thm:trace_comp_hard_part}
Let $(E/F, \chi)$ be a minimal pair of odd level $m \geq 1$. A character $\theta$ of $U_E$ can only occur in $\Xi_{\chi}|_{U_E}$, if $\theta$ coincides with $\chi$ on $U_F U_E^{m+1}$. In this case we have 

\[ \langle \theta, \Xi_{\chi}\rangle_{U_E} = \begin{cases} 1 & \text{if $\theta = \chi$ or $\chi^{\tau}$} \\ 2 & \text{if $\theta \neq \chi, \chi^{\tau}$ and $s(\theta) \in Q_{\alpha_{\theta}}$} \\ 0 & \text{otherwise.} \end{cases} \]
\end{prop}

The main ingredient in the proof of Proposition \ref{thm:trace_comp_hard_part} is the following trace computation.

\begin{prop}\label{cor:non_split_traces_via_elem_modifications}
Let $0 \leq \alpha < \lengthofv$. Let $g \in U_F U_E^{2\alpha + 1} \sm U_F U_E^{2\alpha + 3}$. Then
\[  \tr(g; \Xi_{\chi}) = 2q^{\alpha} \sum_{s \in Q_{\alpha}}\chi_s(g) + q^{\alpha}(\chi(g) + \chi^{\tau}(g)).  \]
\end{prop}

\begin{proof}
We can write $g = z g^{\prime}$ for $z \in U_F U_E^{m+1}$ and $g^{\prime} = \iota(1 + u^{2\alpha + 1}y)$ with $y \in U_F$. Then the result follows from Proposition \ref{prop:non_split_traces} applied to $g^{\prime}$ and the fact that $z$ acts in $V_{\chi}$ as the scalar $\chi(z)$ by Lemma \ref{lm:UJm1_acts_trivial}.
\end{proof}

\begin{proof}[Proof of Proposition \ref{thm:trace_comp_hard_part}]
As $U_F U_E^{m+1}$ acts in $V_{\chi}$ by $\chi|_{U_F U_E^{m+1}}$, the first statement is clear. Assume $\theta|_{U_F U_E^{m+1}} = \chi|_{U_F U_E^{m+1}}$. Now, $U_E^{m+1}$ acts trivial in $V_{\chi}$, thus we can equivalently consider $V_{\chi}$ as a $U_E/U_E^{m+1}$-representation. The filtration from Section \ref{sec:subsub_filtration_on_UE} induces a disjoint decomposition

\[ U_E/U_E^{m+1} = Z \cup \bigcup_{\alpha = 0}^{\lengthofv - 1} (H^{\alpha} \sm H^{\alpha + 1}), \]

\noindent where $H^{\alpha} := U_F U_E^{2\alpha + 1}/U_E^{m+1}$ and $Z := H^{\lengthofv} = U_F U_E^{m+1}/U_E^{m+1}$. We have $\sharp H^{\alpha} = (q-1)q^{m - \alpha}$  for $0 \leq \alpha \leq \lengthofv$. For $0 \leq \alpha < \lengthofv$ set 

\[ S_{\alpha} := \sum_{g \in H^{\alpha} \sm H^{\alpha + 1} } \theta(g^{-1}) \tr(g; \Xi_{\chi}). \]

\noindent Then the trace computation Proposition \ref{cor:non_split_traces_via_elem_modifications} shows that $S_{\alpha} = (q-1)q^{m-1} S_{\alpha}^{\prime}$ with

\begin{eqnarray*} S_{\alpha}^{\prime} := 2 \sum_{s \in Q_{\alpha}} \left(q \langle \theta, \chi_s \rangle_{H^{\alpha}} - \langle \theta, \chi_s \rangle_{H^{\alpha+1}} \right) + q \langle \theta, \chi + \chi^{\tau} \rangle_{H^{\alpha}} - \langle \theta, \chi + \chi^{\tau} \rangle_{H^{\alpha+1}},
\end{eqnarray*}

\noindent and using Lemma \ref{lm:UJm1_acts_trivial} we deduce ($H^0 = U_E/U_E^{m+1}$)

\begin{eqnarray*} \langle \theta, \Xi_{\chi}\rangle_{H^0} &=& \frac{1}{\sharp H^0} \left( \sum_{g \in Z} \theta(g^{-1})\tr(g; \Xi_{\chi}) + \sum_{\alpha = 0}^{\lengthofv - 1} S_{\alpha} \right) = \frac{1}{\sharp H^0} \left(\sharp Z \cdot \dim V_{\chi} + \sum_{\alpha = 0}^{\lengthofv - 1} S_{\alpha} \right) \\
&=& \frac{1}{q} \left(q-1 + \sum_{\alpha = 0}^{\lengthofv - 1} S_{\alpha}^{\prime} \right).
\end{eqnarray*}

\noindent Now the proposition follows from Lemma \ref{lm:sizes_of_S_alphas}, by considering the five cases $i(\theta) = 0$, $0 < i(\theta) < m + 1$ and $s(\theta) \in Q_{\alpha_{\theta}}$, $0 < i(\theta) < m + 1$ and $s(\theta) \not\in Q_{\alpha_{\theta}}$, $i(\theta) = m + 1$ and $s(\theta) \in Q_{\alpha_{\theta}}$, $i(\theta) = m + 1$ and $s(\theta) \not\in Q_{\alpha_{\theta}}$.\qedhere
\end{proof}

\begin{lm}\label{lm:sizes_of_S_alphas}
\begin{itemize}
\item[(i)] Assume $0 \leq \alpha \leq \lengthofv - 2$. Then 
\[ S_{\alpha}^{\prime} = \begin{cases} 	0 	& \text{if $i(\theta) \leq 2\alpha + 1$} \\
					q 	& \text{if $i(\theta) = 2\alpha + 2$ and $s(\theta, \alpha) \in Q_{\alpha}$} \\
					-q 	& \text{if $i(\theta) = 2\alpha + 2$ and $s(\theta, \alpha) \not\in Q_{\alpha}$} \\
					0 	& \text{if $i(\theta) \geq 2\alpha + 3$.}
   \end{cases}
\]
\item[(ii)] For $\alpha = \lengthofv - 1$ we have
\[ S_{\lengthofv - 1}^{\prime} = \begin{cases}1 	& \text{if $i(\theta) \leq m = 2\lengthofv - 1$} \\
					q + 1 	& \text{if $i(\theta) = m + 1 = 2\lengthofv$ and $s(\theta, \lengthofv - 1) \in Q_{\lengthofv - 1}$} \\
					- q + 1 & \text{if $i(\theta) = m + 1 = 2\lengthofv$ and $s(\theta, \lengthofv - 1) \not\in Q_{\lengthofv - 1}$.}
					\end{cases}
\]
\end{itemize}
\end{lm}

\begin{proof}

\noindent Let $\pr_{\alpha}$ be as in Section \ref{sec:str_of_Q_alpha}. (i): Assume $0 \leq \alpha \leq \lengthofv - 2$. 

\textbf{Case $i(\theta) \leq 2\alpha + 1$.} Then on $H^{\alpha}$ resp. on $H^{\alpha+1}$ the character $\theta$ is equal to exactly one of the characters $\chi$ or $\chi^{\tau}$ (as $\alpha \leq \lengthofv - 2$ and $\chi \neq \chi^{\tau}$ on $H^{\lengthofv - 1}$ by Lemma \ref{lm:non_triv_same_as_norm_fact}). Assume that this character is $\chi$ (the other case is similar). Thus
$\langle \theta, \chi_s \rangle_{H^{\alpha}} = 0$ for all $s \in Q_{\alpha}$, $\langle \theta, \chi + \chi^{\tau} \rangle_{H^{\alpha}} = \langle \theta, \chi + \chi^{\tau} \rangle_{H^{\alpha+1}} = 1$ and by Lemma \ref{lm:character_modifications_techn_lemma} we have

\begin{equation}\label{eq:psi_chi_on_some_sgr_tech_detail}
\langle \theta, \chi_s \rangle_{H^{\alpha+1}} = \begin{cases} 1 & \text{if $\pr_{\alpha}(s) = 1$} \\  
									0 & \text{otherwise.} 
							  \end{cases} 
\end{equation}

\noindent By Lemma \ref{lm:study_of_Qalpha12_sizes_and_maps}(i), the first case happens for exactly $(q-1)/2$ elements in $s \in Q_{\alpha}$. Altogether we obtain $S_{\alpha}^{\prime} = 2(0 - \frac{q-1}{2}) + q - 1 = 0$. 

\textbf{Case $i(\theta) = 2\alpha + 2$.} The character $\theta$ coincides on $H^{\alpha}$ neither with $\chi$ nor with $\chi^{\tau}$, hence $\langle \theta, \chi + \chi^{\tau} \rangle_{H^{\alpha}} = 0$. As $2\alpha + 3 \leq 2(\lengthofv - 2) + 3 = m$ by assumption, $\theta$ coincides on $H^{\alpha+1}$ with precisely one of the characters $\chi$ or $\chi^{\tau}$ and hence $\langle \theta, \chi + \chi^{\tau} \rangle_{H^{\alpha+1}} = 1$. As $2(2\alpha + 1) \geq 2\alpha + 2 = i(\theta)$, the quantity $s(\theta,\alpha) \in R_{\alpha}^{\langle \tau \rangle, \prime}$ is well-defined. Thus

\[ \sum_{s \in Q_{\alpha}} q \langle \theta, \chi_s \rangle_{H^{\alpha}} = \begin{cases} q & \text{if $s(\theta,\alpha) \in Q_{\alpha}$} \\ 0 & \text{otherwise.} \end{cases} \]

\noindent Moreover, \eqref{eq:psi_chi_on_some_sgr_tech_detail} holds also in this case, and again there are precisely $(q-1)/2$ elements of $Q_{\alpha}$ with image $1$ in $R_{\alpha + 1}^{\langle \tau \rangle}$. From this we deduce the result.

\textbf{Case $i(\theta) \geq 2\alpha + 3$.} Then $\langle \theta, \chi + \chi^{\tau} \rangle_{H^{\alpha}} = \langle \theta, \chi + \chi^{\tau} \rangle_{H^{\alpha+1}} = 0$. Assume first $2(2\alpha + 1) \geq i(\theta)$ (in particular, $\alpha > 0$). By Proposition \ref{lm:description_of_special_characters} there is a unique $s(\theta, \alpha) \in R_{\alpha}^{\langle \tau \rangle, \prime}$ such that $\theta$ coincides with $\chi_{s(\theta,\alpha)}$ on $H^{\alpha}$. Hence 
\[\sum_{s \in Q_{\alpha}} q \langle \theta, \chi_s \rangle_{H^{\alpha}} = \begin{cases}q  & \text{if $s(\theta,\alpha) \in Q_{\alpha}$} \\ 0 & \text{otherwise.} \end{cases} \]

\noindent On the other hand, note that $\theta$ coincides with $\chi_s$ on $H^{\alpha+1}$ if and only if $s \in \pr_{\alpha}^{-1}(\pr_{\alpha}(s(\theta,\alpha)))$. Thus using Lemma \ref{lm:study_of_Qalpha12_sizes_and_maps}(ii) we deduce
\[ 
\sum_{s \in Q_{\alpha}} \langle \theta, \chi_s \rangle_{H^{\alpha+1}} = 
		  \begin{cases} q & \text{if $s(\theta,\alpha) \in Q_{\alpha}$} \\  
				0 & \text{otherwise.} 
		  \end{cases}   
\]

\noindent In any case we compute $S_{\alpha}^{\prime} = 0$. Finally, assume that $i(\theta) > 2(2\alpha + 1)$, i.e., $i(\theta) \geq 2(2\alpha + 2)$ as $i(\theta)$ is even. Thus $\theta$ does not coincide with $\chi$ or $\chi^{\tau}$ on $U_FU_E^{\min\{m+1,2(2\alpha+1)\}}/U_E^{m+1} =  H^{2\alpha+1}$. On the other hand, for $s \in R_{\alpha}^{\langle \tau \rangle, \prime}$, the character $\chi_s$ coincides by definition with $\chi$ or $\chi^{\tau}$ on $H^{2\alpha + 1}$. Thus $\theta$ does not coincide with any of the characters $\chi_s$ on $H^{2\alpha + 1}$ and from $2\alpha + 1 \leq 2\alpha + 3 \leq 4\alpha + 3$ we deduce

\[ \langle \theta, \chi_s \rangle_{H^{\alpha}} = \langle \theta, \chi_s \rangle_{H^{\alpha+1}} = 0, \]

\noindent and hence also $S_{\alpha}^{\prime} = 0$.

(ii): \textbf{Case $i(\theta) \leq m$.} Then $\theta$ coincides with exactly one of the characters $\chi$, $\chi^{\tau}$ on $H^{\lengthofv - 1}$. Thus $\langle \theta, \chi + \chi^{\tau} \rangle_{H^{\lengthofv - 1}} = 1$, $\langle \theta, \chi + \chi^{\tau} \rangle_Z = 2$, and $\langle \theta, \chi_s \rangle_{H^{\lengthofv - 1}} = 0$, $\langle \theta, \chi_s \rangle_Z = 1$ for all $s \in Q_{\alpha}$. Using Lemma  \ref{lm:study_of_Qalpha12_sizes_and_maps}(iii) we compute

\[S_{\lengthofv - 1}^{\prime} = 2(q\cdot 0 - \frac{q-3}{2}) + (q\cdot 1 - 2) = 1. \]

\noindent \textbf{Case $i(\theta) = m + 1$.} Then $\langle \theta, \chi + \chi^{\tau} \rangle_{H^{\lengthofv - 1}} = 0$, $\langle \theta, \chi + \chi^{\tau} \rangle_Z = 2$, $\langle \theta, \chi_s \rangle_Z = 1$  for all $s \in Q_{\lengthofv - 1}$. Moreover, $s(\theta, \lengthofv - 1)$ is well-defined and

\[ \sum_{s\in Q_{\lengthofv - 1}} q \langle \theta, \chi_s \rangle_{H^{\lengthofv - 1}} = \begin{cases} q & \text{if $s(\theta,\lengthofv - 1) \in Q_{\lengthofv - 1}$} \\ 0 & \text{otherwise.} \end{cases} \]

\noindent Again we conclude by Lemma \ref{lm:study_of_Qalpha12_sizes_and_maps}(iii). \qedhere
\end{proof}

\begin{proof}[Proof of Theorem \ref{thm:Xi_chi_restriction_to_Etimes}]
 Let $\phi$ be any one of the two characters of $E^{\times}$ satisfying $\phi(U_E) = 1$ and $\phi(t) = \chi(t)^{-1}$. Consider the $E^{\times}$-representation $\phi\Xi_{\chi}$ given by $(\phi\Xi_{\chi})(e) = \phi(e)\Xi_{\chi}(e)$. By construction, it is trivial on the subgroup $\langle t, U_E^{m+1}\rangle$ of $E^{\times}$, and we consider it as a representation of the finite group $E^{\times}/\langle t, U_E^{m+1}\rangle \cong U_E/U_E^{m+1} \times \langle u \rangle/\langle u^2 \rangle$. Let $\theta$ be a character of $E^{\times}$. Then $\langle \theta,\Xi_{\chi}\rangle_{E^{\times}} = 0$, unless $\theta$ coincides with $\chi$ on $F^{\times}U_E^{m+1}$. Assume this holds. Then $\phi\theta$ also factors through a character of $U_E/U_E^{m+1} \times \langle u \rangle/\langle u^2 \rangle$ and its multiplicity in $\Xi_{\chi}$ can be computed as follows:

\[ \langle \theta,\Xi_{\chi}\rangle_{E^{\times}} = \langle \phi\theta,\phi\Xi_{\chi}\rangle_{U_E/U_E^{m+1} \times \langle u \rangle/\langle u^2 \rangle} = \frac{1}{2(q-1)q^m} \sum_{g \in U_E/U_E^{m+1} \times \langle u \rangle/\langle u^2 \rangle} \theta(g^{-1}) \tr(g; \phi\Xi_{\chi}). \]

\noindent Let $\lambda(\theta) \in \{0,1\}$ be such that $\theta(u) = (-1)^{\lambda(\theta)}\chi(u)$ and let $\sgn(\chi)$ be $0$ if $\chi$ is even, and $1$ otherwise. We deduce from the above and from Proposition \ref{prop:traces_of_val_1_elements}:

\[ \langle \theta,\Xi_{\chi}\rangle_{E^{\times}} = \frac{1}{2} (\langle \theta, \Xi_{\chi} \rangle_{U_E/U_E^{m+1}} + (-1)^{\lambda(\theta)} \langle \theta, \chi + (-1)^{\sgn(\chi)}\chi^{\tau} \rangle_{U_E/U_E^{m+1}}). \]

\noindent Now Theorem \ref{thm:Xi_chi_restriction_to_Etimes} follows from Proposition \ref{thm:trace_comp_hard_part} by a simple case-by-case study. \qedhere
\end{proof}

\begin{prop}\label{prop:traces_of_val_1_elements}
Let $g \in E^{\times}$ with $v_u(g) = 1$. Then 
\[ \tr(g; \Xi_{\chi}) = \chi(g) + \chi^{\tau}(g). \]
\end{prop}

\begin{proof}
The proof is given in Section \ref{sec:traces_of_elements_in_varpi_UJ}. 
\end{proof}

\begin{cor}\label{cor:recipy_for_chi}
The character $\chi$ can be reconstructed from the $E^{\times}$-representation $\Xi_{\chi}|_{E^{\times}}$.
\end{cor}

\begin{proof}
By Lemma \ref{lm:UJm1_acts_trivial}, $\Xi_{\chi}|_{E^{\times}}$ determines $\chi|_{F^{\times}U_E^{m+1}}$ uniquely. Consider the map

\[ f \colon A := \{ \theta \in (E^{\times})^{\vee} \colon \theta|_{F^{\times}U_E^{m+1}} = \chi|_{F^{\times}U_E^{m+1}} \} \tar \{ \theta^{\prime} \in U_E^{\vee} \colon \theta^{\prime}|_{U_F U_E^{m+1}} = \chi|_{U_F U_E^{m+1}} \}, \]

\noindent given by restricting characters of $E^{\times}$ to $U_E$. It is surjective and $2$-to-$1$. By Proposition \ref{thm:trace_comp_hard_part} and Theorem \ref{thm:Xi_chi_restriction_to_Etimes}, $\chi$ and $\chi^{\tau}$ are the two unique elements among all elements $\theta \in A$, with the following property: $\theta$ occurs in $\Xi_{\chi}$, but the unique element of $f^{-1}(f(\theta)) \sm \{ \theta \}$ does not occur in $\Xi_{\chi}$.
\end{proof}

\subsection{Relation to strata, cuspidality}\label{sec:cuspidality} \mbox{} 

Using the unipotent traces computed in Section \ref{sec:unipotent_traces}, we show the first part of Theorem \ref{thm:hard_version_main_result}. We use the terminology of intertwining and strata from \cite{BH}\S11 and Chapter 4. The following is analogous to \cite{Iv2} Proposition 4.22 and Corollary 4.23. 
Recall the notation $N_{\lengthofv}$, $N_{\lengthofv}^{\lengthofv}$ from Section \ref{sec:unipotent_traces}. Let $N$ resp. $N^{\lengthofv}$ denote the preimage of $N_{\lengthofv}$ resp. $N_{\lengthofv}^{\lengthofv}$ under the natural projection $U_{\fJ} \tar U_{\fJ}/U_{\fJ}^{m+1}$.

\begin{prop}\label{prop:cuspidality_general_stuff}
Let $m \geq 0$. Let $\Xi$ be an irreducible $E^{\times} U_{\fJ}$-representation, which is trivial on $U_{\fJ}^{m+1}$ and does not contain the trivial character on $N^{\lengthofv}$. Then the $G(F)$-representation $\Pi_{\Xi} = \cIndd_{E^{\times} U_{\fJ}}^{G(F)} \Xi$ is irreducible, cuspidal and admissible. Moreover, it contains a ramified simple stratum $(\fJ,m,\alpha)$ for some $\alpha \in \varpi^{-m}\fJ$. One has $\ell(\Pi_{\Xi}) = \frac{m}{2}$. For any character $\phi$ of $F^{\times}$ one has $0 < \ell(\Pi_{\Xi}) \leq \ell(\phi\Pi_{\Xi})$.
\end{prop}

From this we can deduce the first statement of Theorem \ref{thm:hard_version_main_result}.

\begin{cor}\label{cor:first_part_of_hard_thm}
Let $(E/F, \chi)$ be a minimal pair. The representation $R_{\chi}$ is irreducible, cuspidal and admissible. It contains a ramified simple stratum and is, in particular, ramified. Moreover, $\ell(R_{\chi}) = \frac{\ell(\chi)}{2}$ and for any character $\phi$ of $F^{\times}$ one has $0 < \ell(R_{\chi}) \leq \ell(\phi R_{\chi})$. 
\end{cor}

\begin{proof}
The assumptions of Proposition \ref{prop:cuspidality_general_stuff} are satisfied for the $E^{\times}U_{\fJ}$-representation $\Xi_{\chi}$ by Corollary \ref{cor:irreducibility_of_XI_chi} and Proposition \ref{prop:unip_rep}.
\end{proof}

\begin{proof}[Proof of Proposition \ref{prop:cuspidality_general_stuff}] 
Irreducibility and cuspidality of $\Xi$ follow from \cite{BH} Theorem 11.4, which assumptions are satisfied due to Lemma \ref{lm:not_intertwining}. To contain a stratum is defined with respect to an additive character. So fix some character $\psi$ of $F$ of level $1$. Make the isomorphism \eqref{eq:characs_of_UJ_fromBH} explicit for $k = m-1$, $r = m$:

\[ \varpi^{-m}\fJ / \varpi^{1-m} \fJ \stackrel{\sim}{\longrar} (U_{\fJ}^m/U_{\fJ}^{m+1})^{\vee}. \]

\noindent An element of $\varpi^{-m}\fJ / \varpi^{1-m} \fJ$ resp. of $U_{\fJ}^m/U_{\fJ}^{m+1}$ is represented by a matrix $a = \matzz{}{a_2 t^{-\lengthofv}}{a_3t^{1-\lengthofv}}{}$ resp. $x = \matzz{}{x_2 t^{\lengthofv - 1}}{x_3t^{\lengthofv}}{}$ with $a_2,a_3,x_2,x_3 \in k$ and $\psi_{\fM,a}(x) = \psi(a_2x_3 + a_3x_2)$. The restriction of $\Xi$ to $U_{\fJ}^m$ factors through a representation of the abelian group $U_{\fJ}^m/U_{\fJ}^{m+1}$, thus it decomposes as a sum of characters, each of which is of the form $\psi_{\fM,a}|_{U_{\fJ}^m}$ for some $a \in \varpi^{-m}\fJ$. With other words, for each $a$, such that $\psi_{\fM,a}|_{U_{\fJ}^m}$ is contained in $\Xi$, the ramified stratum $(\fJ,m,a)$ occurs in $\Pi_{\Xi}$. By definition, a ramified stratum is simple, if and only it is fundamental, i.e., the coset $a + \varpi^{1-m}\fJ$ does not contain a nilpotent element of $\fM$. Thus to show that $\Pi_{\Xi}$ contains a ramified simple stratum it is enough to show the following claim.

\textbf{Claim.} Let $a \in \varpi^{-m}\fJ$. Assume $\psi_{\fM,a}|_{U_{\fJ}^m}$ occurs in $\Xi$. Then $a + \varpi^{1-m}\fJ$ does not contain nilpotent elements of $\fM$, or with other words $a_2,a_3 \neq 0$ (with notations as above).

\begin{proof}[Proof of the claim] Assume $a_2 = 0$, then the restriction of $\psi_{\fM,a}$ to the subgroup $N^{\lengthofv}$ of $U_{\fJ}^m$ is the trivial character, which contradicts our assumptions on $\Xi$. Thus $a_2 \neq 0$. Assume $a_3 = 0$. As $\varpi \in E^{\times}U_{\fJ}$, the character $\psi_{\varpi a \varpi^{-1}}$ also occurs in $\Xi$ (proof as in \cite{Iv2} Lemma 4.25). But $\varpi a \varpi^{-1} = \matzz{}{a_3t^{-\lengthofv}}{a_2t^{1-\lengthofv}}{}$ and we deduce a contradiction as in the already proven part.
\end{proof}

Thus we have shown that $\Pi_{\Xi}$ contains a ramified fundamental stratum of the form $(\fJ,m,a)$. Then \cite{BH} Theorem 12.9 shows that $\ell(\Pi_{\Xi}) = \frac{m}{2}$. Furthermore, if an essentially scalar stratum would be contained in $\Pi_{\Xi}$, then by \cite{BH} Section 12.9, it would have to intertwine with $(\fJ,m,a)$. But by \cite{BH} 13.2 Proposition, no fundamental stratum of the form $(\fM,r,b)$ can intertwine with the fundamental ramified stratum $(\fJ,m,a)$. Thus no essentially scalar stratum is contained in $\Pi_{\Xi}$ and \cite{BH} 13.3 Theorem shows the last statement of the proposition.
\end{proof}

\begin{lm}\label{lm:not_intertwining}
Let $\Xi$ be an irreducible $E^{\times} U_{\fJ}$-representation, which is trivial on $U_{\fJ}^{m+1}$ and does not contain the trivial character on $N^{\lengthofv}$. An element $g \in G(F)$ intertwines $\Xi$ if and only if $g \in E^{\times}U_{\fJ}$.
\end{lm}

\begin{proof}
The double $E^{\times}U_{\fJ}$-cosets in $G(F)$ are represented by diagonal matrices with entries $t^{\alpha}$, $1$ for $\alpha \geq 0$. The rest of the proof works exactly as in \cite{Iv2} Lemma 4.24.
\end{proof}


\subsection{Relation to cuspidal inducing data}\label{sec:rel_to_cusp_ind_data} \mbox{}

We relate the representations $R_{\chi}$, $\pi_{\chi}$ to each other. The following proposition finishes the proof of Theorem \ref{thm:hard_version_main_result}.

\begin{prop}\label{prop:Rchi_isom_pichi}
Let $(E/F, \chi)$ be an admissible pair. Then $R_{\chi} \cong \pi_{\chi}$.
\end{prop}

\begin{proof} By twisting both sides with a character of $F^{\times}$, we can assume that $(E/F,\chi)$ is a minimal pair. By construction of $\pi_{\chi}$ and Lemma \ref{eq:Rchi_is_induction_from_ZUJ_of_Xichi}, it is enough to show that $\Xi_{\chi} \cong \Theta_{\chi}$ ($\Theta_{\chi}$ is as in Section \ref{sec:BH_constr_of_pi_chi}). From Corollary \ref{cor:first_part_of_hard_thm} and the proof of Proposition \ref{prop:cuspidality_general_stuff} it follows that there is a simple (ramified) stratum $(\fJ,m,\beta)$ such that $\Xi_{\chi}|_{U_{\fJ}^{m}}$ contains $\psi_{\beta}$. By \cite{BH} 15.8 Exercise it follows that $(E^{\times}U_{\fJ}, \Xi_{\chi})$ is a cuspidal inducing datum in $G(F)$, i.e., there is some $\chi^{\prime}$ with $\Xi_{\chi} \cong \Theta_{\chi^{\prime}}$. By the last statement of Corollary \ref{cor:first_part_of_hard_thm}, $(E/F, \chi^{\prime})$ has to be minimal. By Lemma \ref{lm:restrictions_to_Etimes_are_isom}, $\Theta_{\chi^{\prime}}|_{E^{\times}} \cong \Xi_{\chi^{\prime}}|_{E^{\times}}$. Thus $\Xi_{\chi^{\prime}}|_{E^{\times}} \cong \Xi_{\chi}|_{E^{\times}}$. Now, by Corollary \ref{cor:recipy_for_chi}, $\chi$ is (up to $\tau$-conjugacy) uniquely determined by $\Xi_{\chi}$, and we deduce $\chi^{\prime} = \chi$ or $\chi^{\prime} = \chi^{\tau}$. As $\Theta_{\chi} \cong \Theta_{\chi^{\tau}}$, the proposition follows. \qedhere
\end{proof}

\begin{lm}\label{lm:restrictions_to_Etimes_are_isom}
Let $(E/F,\chi)$ be a minimal pair. We have $\Theta_{\chi}|_{E^{\times}} \cong \Xi_{\chi}|_{E^{\times}}$.
\end{lm}

\begin{proof} 
The proof is given in Section \ref{sec:trace_on_induced_side}.
\end{proof}


\subsection{Small level case}\label{sec:proof_of_thm_small_level_case} \mbox{}

Let $\chi$ be a character of $E^{\times}$ of (odd) level $m \geq 1$. Let $\lengthofv \geq m+1$ be an integer. Then $\chi$ defines a character $\tilde{\chi}$ of the group $E^{\times}U_{\fJ}^{\lengthofv}$ by composition

\begin{equation} \label{eq:tilde_chi_def}
\tilde{\chi} \colon E^{\times}U_{\fJ}^{\lengthofv} \tar E^{\times}U_{\fJ}^{\lengthofv}/U_{\fJ}^{\lengthofv} \cong E^{\times}/U_E^{\lengthofv} \tar E^{\times}/U_E^{m+1} \stackrel{\chi}{\rar} \overline{\bQ}_{\ell}^{\times}. 
\end{equation}

\begin{prop}\label{thm:small_level_case} Let $\chi$ be a character of $E^{\times}$ of odd level $m \geq 1$ and let $\lengthofv \geq m+1$. Then
\[ \R_{\chi,\lengthofv} \cong \cIndd\nolimits_{E^{\times}U_{\fJ}^{\lengthofv}}^{G(F)} \tilde{\chi}. \]
\end{prop}

\begin{proof}[Proof of Proposition \ref{thm:small_level_case}]
By Lemma \ref{eq:Rchi_is_induction_from_ZUJ_of_Xichi} it is enough to show $\Xi_{\chi,n} \cong \cIndd_{E^{\times}U_E^{\lengthofv}}^{E^{\times}U_{\fJ}} \tilde{\chi}$. To do so, it is enough to show that the traces of each element of $E^{\times}U_{\fJ}$ in both spaces agree. Modulo center, which acts by $\chi|_{F^{\times}}$ in both spaces, any element of $E^{\times}U_{\fJ}$ is represented by an element in $U_{\fJ} \cup \varpi U_{\fJ}$, thus we can restrict to elements lying in this union. The required trace computations are covered by Lemmas \ref{lm:traces_on_ADLV_side_small_level_case} and \ref{lm:traces_on_induced_side_in_small_level_case}. \qedhere
\end{proof}

\begin{lm}\label{lm:traces_on_ADLV_side_small_level_case} Let $g \in U_{\fJ}$. Precisely one of the following cases occurs:
\begin{itemize}
\item[(i)] $g \in U_F U_{\fJ}^{\lengthofv}$. Then $\tr(g;\Xi_{\chi,n})[\chi]) = (q-1)q^{\lengthofv - 1} \tilde{\chi}(g)$. In particular, $U_{\fJ}^{\lengthofv}$ acts trivial in $\Xi_{\chi,n})[\chi]$ and $U_F$ acts through the character $\chi|_{F^{\times}}$. 

\item[(ii)] $g \in U_{\fJ} \sm U_F U_{\fJ}^{\lengthofv}$ is conjugate to an element $x$ of $U_F U_E^{2\alpha + 1} U_{\fJ}^{\lengthofv} \sm U_F U_E^{2\alpha + 3}U_{\fJ}^{\lengthofv}$, such that $2\alpha + 2 \leq \lengthofv$. Then
\[ \tr(g;\Xi_{\chi,n})[\chi]) =  q^{2\alpha + 1}(\tilde{\chi}(x) + \tilde{\chi}^{\tau}(x)). \]

\item[(iii)] $g \in U_{\fJ} \sm U_F U_{\fJ}^{\lengthofv}$ is not conjugate to an element of $U_E U_{\fJ}^{\lengthofv}$. Then $\tr(g;\Xi_{\chi,n})[\chi]) = 0$.
\end{itemize}

Let $g \in \varpi U_{\fJ}$. Precisely one of the following two cases can occur:

\begin{itemize}
 \item[(i)${}^{\prime}$]  $g$ is not conjugate to an element of $E^{\times}U_{\fJ}^{\lengthofv}$. Then $\tr(g;\Xi_{\chi,n})[\chi]) = 0$.
 \item[(ii)${}^{\prime}$] $g$ is conjugate to an element $x$ of $E^{\times}U_{\fJ}^{\lengthofv}$. Then $\tr(g;\Xi_{\chi,n})[\chi]) = \tilde{\chi}(x) + \tilde{\chi}^{\tau}(x)$.
\end{itemize}
\end{lm}

\begin{proof} 
The proof is given in Section \ref{sec:traces_small_level_case}.
\end{proof}

\begin{lm}\label{lm:traces_on_induced_side_in_small_level_case}
Lemma \ref{lm:traces_on_ADLV_side_small_level_case} holds with $\Xi_{\chi,n}$ replaced by $\cIndd_{E^{\times}U_E^{\lengthofv}}^{E^{\times}U_{\fJ}} \tilde{\chi}$. 
\end{lm}

\begin{proof}
The lemma follows by an explicit computation using the Mackey-formula in a way very similar to the proof of Lemma \ref{lm:restrictions_to_Etimes_are_isom}. We omit the details.
\end{proof}

Using notations from Definition \ref{def:distance_quadratic} we have the following structure result.

\begin{thm} \label{thm:small_level_case_morph_determining}
Let $\chi$ be a character of $E^{\times}$ of odd level $m \geq 1$ and let $\lengthofv \geq m+1$. Let $\theta$ be a character of level $\geq m$. There are no non-zero maps from $R_{\chi, \lengthofv}$ to $R_{\theta}$, unless $\theta$ coincides with $\chi$ on $F^{\times}U_E^{m+1}$. In this case, we have

\[ \Hom_{G(F)}(R_{\chi,\lengthofv}, R_{\theta}) = \begin{cases} \overline{\bQ}_{\ell} & \text{if $\theta = \chi$, or $\theta = \chi^{\tau}$, or the distance from $\theta$ to $\chi$ is properly quadratic} \\ 0 & \text{otherwise.} \end{cases} \]

\noindent In particular, the character $\chi$ can be reconstructed from $R_{\chi, \lengthofv}$.
\end{thm}

\begin{proof}
By Lemma \ref{lm:red_via_non_intertwinig_small_level} we may assume that $\theta$ and $\chi$ coincide on $F^{\times}U_{\fJ}^{m+1}$. Thus by our assumption on $\theta$, $(E/F,\chi)$ is a minimal pair and we compute

\begin{eqnarray*}
\Hom_{G(F)}(R_{\chi,\lengthofv}, R_{\theta}) &=& \Hom_{E^{\times}U_{\fJ}} (\Xi_{\chi,\lengthofv}, \Xi_{\theta}) \\
&=& \Hom_{E^{\times}U_{\fJ}^{\lengthofv}} (\tilde{\chi}, \Xi_{\theta}) \\
&=& \Hom_{E^{\times}} (\tilde{\chi}, \Xi_{\theta}) \\
&=& \Hom_{E^{\times}} (\tilde{\chi}, \bigoplus\nolimits_{\theta^{\prime}} \theta^{\prime}),
\end{eqnarray*}
\noindent where $\theta^{\prime}$ runs through the set of all characters of $E^{\times}$ coinciding with $\theta$ on $F^{\times}U_E^{m+1}$, such that either $\theta^{\prime} = \theta$, or $\theta^{\prime} = \theta^{\tau}$, or the distance from $\theta$ to $\theta^{\prime}$ is properly quadratic. Above, the first equality follows from Lemma \ref{lm:red_via_non_intertwinig_small_level}. The second is Frobenius reciprocity and Proposition \ref{thm:small_level_case}. The third follows, as $\lengthofv \geq m+1$ and hence $\tilde{\chi}$, $\Xi_{\theta}$ are trivial on $U_{\fJ}^{\lengthofv}$. Finally, the forth equality follows from Theorem \ref{thm:Xi_chi_restriction_to_Etimes}. The above computation shows the statement of the theorem about $\Hom_{G(F)}(R_{\chi,\lengthofv}, R_{\theta})$. It remains to show that $\chi$ can be reconstructed from $R_{\chi,\lengthofv}$. First, the above considerations characterize $m$ as the greatest odd integer, such that there are non-zero maps from $R_{\chi,\lengthofv}$ to $R_{\theta}$ for some $\theta$ of level $m$. The rest follows as in the proof of Corollary \ref{cor:recipy_for_chi}. 
\end{proof}

\begin{lm}\label{lm:red_via_non_intertwinig_small_level}
Let $\theta$ be a character of $E^{\times}$ of odd level $\ell(\theta) \geq m$. If $\theta$ does not coincide with $\chi$ on $F^{\times} U_E^{m+1}$, then there are no non-zero morphisms between $R_{\chi,n}$ and $R_{\theta}$. Assume $\theta$ coincides with $\chi$ on $F^{\times}U_E^{m+1}$ (in particular, $\ell(\theta) = m$). Then 
\[ \Hom_{G(F)}(R_{\chi,\lengthofv}, R_{\theta}) = \Hom_{E^{\times}U_{\fJ}} (\Xi_{\chi,\lengthofv}, \Xi_{\theta}). \]
\end{lm}

\begin{proof}
Applying twice the Frobenius reciprocity and once the Mackey formula, we see by Lemma \ref{eq:Rchi_is_induction_from_ZUJ_of_Xichi}
\[ \Hom_{G(F)}(R_{\chi,\lengthofv}, R_{\theta}) = \bigoplus_{g \in E^{\times} U_{\fJ} \backslash G(F) / E^{\times}U_{\fJ}} \Hom_{E^{\times}U_{\fJ} \cap {}^g(E^{\times}U_{\fJ})}({}^g\Xi_{\chi,\lengthofv} ,\Xi_{\theta}), \]

\noindent where ${}^g(E^{\times}U_{\fJ}) = g(E^{\times}U_{\fJ})g^{-1}$ and ${}^g\Xi_{\chi,\lengthofv}(x) = \Xi_{\chi,\lengthofv}(g^{-1}xg)$. The set $E^{\times} U_{\fJ} \backslash G(F) / E^{\times}U_{\fJ}$ is represented by the diagonal matrices $e_0(t^{\alpha},1)$ for $\alpha \geq 0$. Let  $g = e_0(t^{\alpha},1)$ with $\alpha > 0$. We show that the summand on the right side corresponding to $g$ vanish. Note that $E^{\times}U_{\fJ} \cap {}^g(E^{\times}U_{\fJ}) \supseteq e_-(\fp_F^{\frac{m+1}{2}})$. On the one hand, Proposition \ref{prop:unip_rep} shows that $\Xi_{\theta}$ does not contain the trivial character on $e_-(\fp_F^{\frac{m+1}{2}})$. On the other hand, $e_-(\fp_F^{\frac{m+1}{2} + \alpha}) \subseteq U_{\fJ}^{m+1}$ as $\alpha > 0$. As $U_{\fJ}^{m+1}$ is normal in $E^{\times} U_{\fJ}$, and $\tilde{\chi}$ is trivial on $U_{\fJ}^{m+1}$, we see by Proposition \ref{thm:small_level_case} that $\Xi_{\chi,\lengthofv}$ is trivial on $U_{\fJ}^{m+1}$ and, in particular, on $e_-(\fp_F^{\frac{m+1}{2} + \alpha})$. As 

\[ {}^g\Xi_{\chi,\lengthofv}(e_-(x)) = \Xi_{\chi,\lengthofv}(e_-(t^{\alpha}x)), \]

\noindent we deduce that ${}^g\Xi_{\chi,\lengthofv}$ is trivial on $e_-(\fp_F^{\frac{m+1}{2}})$. The claim follows, and hence 
\[ \Hom_{G(F)}(R_{\chi,\lengthofv}, R_{\theta}) = \Hom_{E^{\times}U_{\fJ}} (\Xi_{\chi,\lengthofv}, \Xi_{\theta}). \]

\noindent It remains to show that this space is $0$, unless $\theta$ coincides with $\chi$ on $F^{\times} U_E^{m+1}$. Asssume $\Hom_{E^{\times}U_{\fJ}} (\Xi_{\chi,\lengthofv}, \Xi_{\theta}) \neq 0$. By comparing central characters, we see that $\theta|_{F^{\times}} = \chi|_{F^{\times}}$. As above, we see that $\Xi_{\chi,n}$ is trivial on $U_{\fJ}^{m+1}$. Hence $\Xi_{\theta}$ is too. We deduce $\ell(\theta) = m$. This shows our claim.
\end{proof}


\section{Trace computations}\label{sec:applications_of_the_trace_formula}

In this section we use notations from Sections \ref{sec:GL_2_expl_computations} and \ref{sec:repth1}, especially from Notations \ref{Def:D_v_tau}, \ref{notat:higher_level_w} and Definition \ref{def:disc_subschemes_Yvm2}. For $g \in G(F)$ we always write $g = \matzz{g_1}{g_2}{g_3}{g_4}$.

\subsection{Left and right group actions on $X_{\underline{w}_m}^m(1)$} \label{sec:description_of_group_actions} \mbox{}

To apply a trace formula in what follows, we make here the actions \eqref{eq:actions_on_ADLV_GL2} explicit using the coordinates $\psi_{\dot{v}}^m$ from \eqref{eq:explicit_Cvm_param_GL_2_ram_case}. It is clear that $I$ acts on $C_v^m \subseteq \cF^m$ by left multiplication. The following proposition describes this action.

\begin{prop} \label{prop:left_I_action_mega_formula}
Let $\dot{v}$,$v$,$\lengthofv$ be as in Notation \ref{Def:D_v_tau} and $m \geq 1$ odd (we do not assume $2\lengthofv - 1 \geq m$ here). Let $\dot{x}I^m = \psi_{\dot{v}}^m(a,C,D,A,B)$ be a point of $C_v^m$. Then $g \in I$ acts on $\dot{x}I^m$ by

\begin{eqnarray*} 
g.\dot{x}I^m &=& \psi_{\dot{v}}^m(g.a|_{\lengthofv}, \frac{\det(g)}{g_2 a + g_1} CN^{-1}, (g_2 a + g_1)DN,  \dots \\
&&\dots AN + h(g,a) \frac{(g_2 a + g_1)^2DN^2}{\det(g)C}, B + u^{\lengthofv + 1} \frac{g_2}{(g_2 a + g_1)} CD^{-1}N^{-1}),
\end{eqnarray*}

\noindent where 
\begin{eqnarray*}
g.a &:=& \frac{g_4 a + g_3}{g_2 a + g_1} \in L_{[1,\lengthofv + m]} \bG_a (\bar{k}) \quad \text{and $\cdot|_{\lengthofv}$ is as in Section \ref{sec:slices_of_positive_loops}} \\
N &:=& 1 + u^{\lengthofv + 1} \frac{g_2}{g_2 a + g_1} C D^{-1} A \\
h(g,a) &:=& u^{-(\lengthofv+1)} (g.a - g.a|_{\lengthofv}) \in L_{[0,m-1]}\bG_a(\bar{k}).
\end{eqnarray*}
\end{prop}

\begin{proof} First, observe that the expressions in the proposition are well-defined, as $v_u(g_1) = v_u(g_4) = 0$, $v_u(g_2) \geq 0$, $v_u(g_3) > 0$ and $v_u(a) > 0$. We compute in $G(\breve{E})$ (with $a,g.a,C,D$ replaced by some representatives in $\bar{k}\llbracket u \rrbracket$):

\[ g e_-(a) \dot{v} e_0(C,D) = e_-(g.a) \dot{v} e_0(\frac{\det(g)C}{g_2 a + g_1}, (g_2 a+ g_1) D ) e_-(u^{\lengthofv + 1} \frac{g_2 CD^{-1}}{g_2 a +g_1} ). \]

\noindent Further, using \eqref{eq:commutativity_relations} we see that

\[ e_-(u^{\lengthofv + 1} \frac{g_2 CD^{-1}}{g_2 a +g_1}) e_+(A) e_-(B) = e_0(N^{-1},N) e_+(AN) e_-(B + u^{\lengthofv+1} \frac{g_2 CD^{-1}}{g_2 a +g_1}N^{-1}), \]

\noindent with $N$ as in the proposition. Combining the two last computations we see:

\begin{eqnarray}\label{eq:comut_for_group_action}
\nonumber g.\dot{x}I^m &=& g e_-(a) \dot{v} e_0(C,D) e_+(A) e_-(B) \\
&=& e_-(g.a) \dot{v} e_0(\frac{\det(g)C}{g_2 a + g_1}, (g_2 a+ g_1) D ) e_-(u^{\lengthofv + 1} \frac{g_2 CD^{-1}}{g_2 a +g_1} ) e_+(A) e_-(B) \\
\nonumber &=& e_-(g.a) \dot{v} e_0(\frac{\det(g)C}{g_2 a + g_1}, (g_2 a+ g_1) D ) e_0(N^{-1},N) e_+(AN) e_-(B + u^{\lengthofv+1} \frac{g_2 CD^{-1}}{g_2 a +g_1}).
\end{eqnarray}

\noindent Now the only thing we have to do, is to replace $g.a \in L_{[1,\lengthofv + m]} \bG_a (\bar{k})$ in the last expression by an element in $L_{[1,\lengthofv + m]}^{\leq \lengthofv} \bG_a (\bar{k})$. Therefore, note that $g.a$ and $g.a|_{\lengthofv}$ have the same image in $L_{[1,\lengthofv]}\bG_a(\bar{k})$, i.e., $g.a - g.a|_{\lengthofv}$ is divisible by $u^{\lengthofv + 1}$ and $h(g,a)$ is well-defined as an element of $L_{[0,m-1]}\bG_a(\bar{k})$ and that 

\[ e_-(g.a) \dot{v} = e_-(g.a|_{\lengthofv}) e_-(g.a - g.a|_{\lengthofv})\dot{v} = e_-(g.a|_{\lengthofv}) \dot{v} e_+(h(g,a)). \]

\noindent Combining this and \eqref{eq:comut_for_group_action} finishes the proof of the proposition. \qedhere
\end{proof}

We compute $h(g,a)$ from Proposition \ref{prop:left_I_action_mega_formula} in some cases of interest for us. We point out, that later we need to know $h(g,a)$ only modulo $u^{\lengthofv}$ (cf. Proposition \ref{lm:our_form_of_Boy_trace_formula}).

\begin{lm}\label{sublm:hga_comp_for_UJm1} Let $\lengthofv$,$m$ be as in Proposition \ref{prop:left_I_action_mega_formula}. Assume $m \leq 2\lengthofv - 1$. Let $g \in U_{\fJ}$ and $a \in L_{[1,\lengthofv + m]}^{\leq \lengthofv}(k)$ with $a_1 \neq 0$. 
\begin{itemize} 
\item[(i)] For $g \in U_{\fJ}^{m+1}$ we have $v_u(h(g,a)) \geq \lengthofv$.
\end{itemize}
Let $g = \iota(1 + yu^{2\alpha + 1})$ with $0 \leq \alpha \leq \lengthofv - 1$ and $y \in U_F$. Write $a = u a^{\prime}$
\begin{itemize}
\item[(ii)] If $\alpha \geq \lfloor \frac{\lengthofv}{2} \rfloor$, then $h(g,a) = u^{2\alpha + 1 - \lengthofv}y (1 - a^{\prime,2})(1 - u^{2\alpha + 1}ya^{\prime})$.
\item[(iii)] If $0 \leq \alpha < \lfloor \frac{\lengthofv}{2} \rfloor$ and $a^{\prime} = \pm 1 + u^{\lengthofv - 2\alpha - 1}b$ for some $b \in L_{[0;m+2\alpha + 1]}^{\leq 2\alpha} \bG_a(k)$, then $h(g,a) = \frac{y(\mp 2b - u^{\lengthofv - 2\alpha - 1}b^2)}{1 \pm u^{2\alpha + 1}y + u^{\lengthofv}yb}$.
\end{itemize}
\end{lm}

\begin{proof} In any of the three cases, a simple calculation shows $g.a|_{\lengthofv} = a$ (only this case is of interest for us, cf. \eqref{eq:Gl_1_v1}). Now (i) is an easy computation. For (ii) and (iii) we compute
\[ u^{\lengthofv + 1} h(g,a) = g.a - g.a|_{\lengthofv} = g.a - a = \frac{a + u^{2\alpha + 2}y}{1 + u^{2\alpha}ya} - a = \frac{u^{2\alpha + 2} y (1 - a^{\prime,2})}{1 + u^{2\alpha + 1}ya^{\prime}}. \]
From this the lemma follows. \qedhere
\end{proof}

Let $\varpi$ be as in Section \ref{sec:subgroups_of_GF} and $y_1 := e_0(u,-u)$. Left multiplication by $\varpi$ composed with right multiplication by $y_1^{-1}$ defines an automorphism $\tilde{\beta}_{\varpi}$ of $\tilde{Y}_{\dot{w}}^m$. By (the proof of) Lemma \ref{lm:tildeYwm_is_EUJ_stable}, $\tilde{\beta}_{\varpi}$ restricts to an automorphism 

\[ \beta_{\varpi} \colon Y_{\dot{w}}^m \stackrel{\sim}{\rar} Y_{\dot{w}}^m \quad \text{ given by } \quad \dot{x}I^m \mapsto \varpi \dot{x}y_1^{-1} I^m. \]

\begin{prop}\label{prop:action_of_alpha_expl}
Let $\dot{v}$, $v$, $\lengthofv$, $D_w^{\tau}$ be as in Notation \ref{Def:D_v_tau} and $2\lengthofv - 1 \geq m \geq 1$ odd. Let $\psi_{\dot{v}}^m(\pm u,C,D,A,B)$ be a point of $Y_{\dot{w}}^m$ lying over $\pm u \in D_w^{\tau}$. Then 

\[ \beta_{\varpi}(\psi_{\dot{v}}^m(\pm u,C,D,A,B)) = \psi_{\dot{v}}^m(\pm u,\mp CM^{-1},\mp D M,-AM,B), \]

\noindent where $M := 1 - 2u^{\lengthofv}C\tau(C)^{-1}A$.
\end{prop}

\begin{proof}
Write $\dot{x}I^m = \psi_{\dot{v}}^m(\pm u,C,D,A,B)$. Using formulas \eqref{eq:commutativity_relations} we compute

\begin{eqnarray*} 
\beta_{\varpi}(\dot{x}I^m) &=& e_-(\pm u)\dot{v} e_0(\mp 1, \mp 1) e_-(\mp u^{\lengthofv})e_0(C,D)e_+(-A)e_-(-B)I^m = \\
&=& e_-(\pm u) \dot{v} e_0(\mp C M^{\prime,-1}, \mp D M^{\prime}) e_+(-AM^{\prime}) e_-(\mp u^{\lengthofv}CD^{-1} - B),
\end{eqnarray*}

\noindent where $M^{\prime} := 1 \pm u^{\lengthofv}CD^{-1}A$. Now $a = \pm u$ gives $R = u^{-1}(\tau(a) - a) = u^{-1}(\mp u - (\pm u)) = \mp 2$ and as $\dot{x}I^m \in Y_{\dot{w}}^m$, we have $D^{-1} \equiv R \tau(C)^{-1} \equiv \mp 2\tau(C)^{-1}\mod u^{\lengthofv}$ and $B = u^{\lengthofv}C\tau(C)^{-1}$. This shows on the one hand $M^{\prime} = M$, and on the other hand $\mp u^{\lengthofv}CD^{-1} - B = 2B - B = B$.
\end{proof}

Now we make the right $I_{m, \underline{w}_m}/I^m$-action on $Y_{\dot{w}}^m$ explicit.

\begin{prop} \label{prop:right_action_formula}
Let $\dot{v}$,$v$, $\lengthofv$ be as in Notation \ref{Def:D_v_tau} and $m \geq 1$ odd (we do not assume $2\lengthofv - 1 \geq m$ here). Let $\psi_{\dot{v}}^m(a,C,D,A,B)$ be a point of $C_v^m$. Then $i = \matzz{i_1}{}{}{\tau(i_1)}\matzz{1}{i_2}{}{1} \in I_{m, \underline{w}_m}/I^m$ acts on $\psi_{\dot{v}}^m(a,C,D,A,B)$ by

\begin{eqnarray*} 
\psi_{\dot{v}}^m(a,C,D,A,B).i &= \psi_{\dot{v}}^m(a, C i_1 H^{-1}, D \tau(i_1) H, i_1^{-1} \tau(i_1) H^2A + i_2 H, i_1 \tau(i_1)^{-1}BH^{-1}),
\end{eqnarray*}

\noindent where $H = 1 + i_1 \tau(i_1)^{-1} i_2 B \in \bar{k}[u]/u^{m+1}$ (note that $i_2$ is only determined $\mod u^m$, but $B \equiv 0 \mod u$).
\end{prop}

\begin{proof}
The proof is a computation similar to (and simpler as) the proof of Proposition \ref{prop:left_I_action_mega_formula}. \qedhere
\end{proof}


\subsection{Generalities on the trace formula}\label{sec:gen_on_trace_formula} \mbox{}

We use the following trace formula due to Boyarchenko.

\begin{lm}[\cite{Bo} Lemma  2.12] \label{lm:Boy_trace_formula}
Let $X$ be a separated scheme of finite type over a finite field $\bF_Q$ with $Q$ elements, on which a finite group $A$ acts on the right. Let $g \colon X \rar X$ be an automorphism of $X$, which commutes with the action of $A$. Let $\psi \colon A \rar \overline{\bQ}_{\ell}^{\ast}$ be a character of $A$. Assume that $\coh_c^i(X)[\psi] = 0$ for $i \neq i_0$ and $\Frob_Q$ acts on $\coh_c^{i_0}(X)[\psi]$ by a scalar $\lambda \in \overline{\bQ}_{\ell}^{\ast}$. Then 

\[ {\rm Tr}(g^{\ast}, \coh_c^{i_0}(X)[\psi]) = \frac{(-1)^{i_0}}{\lambda \cdot \sharp{A}} \sum_{a \in A} \psi(a) \cdot \sharp S_{g,a}, \]

\noindent where $S_{g,a} = \{ x \in X(\overline{\bF_q}) \colon g(\Frob_Q(x)) = x \cdot a \}$.
\end{lm}

We adapt Lemma \ref{lm:Boy_trace_formula} to our situation.

\begin{prop}\label{lm:our_form_of_Boy_trace_formula} Let $\lengthofv \geq 1$, $m \geq 1$ two integers with $m \leq 2\lengthofv + 1$. Let $\chi$ be a character of $E^{\times}$ of level $m$.
Let $g \in U_{\fJ}$. Then
\[ \tr(g ; \coh_c^0(\tilde{Y}_{\dot{w}}^m)[\chi]) = \sum_{i_1 \in U_E/U_E^{m+1}} \sharp S_{g,i_1} \chi(i_1), \]

\noindent where $S_{g, i_1}$ is empty, unless $\det(g) \equiv i_1\tau(i_1) \mod u^{m+1}$, in which case it is the set of solutions of the equations

\begin{eqnarray}
\label{eq:Gl_1_v1} g_2 a^2 + (g_1 - g_4)a - g_3 &\equiv& 0 \mod u^{\lengthofv + 1} \\
\label{eq:Gl_2_v1} \tau(i_1) (1 + u^{\lengthofv}h(g,a) R^{-1}) &\equiv&  g_2 a + g_1 \mod u^{m+1}
\end{eqnarray}


\noindent in $a \in L^{\leq \lengthofv}_{[1, \lengthofv + m]} \bG_a(k)$ (with $a_1 \neq 0$), where 

\begin{eqnarray}
h(g,a) &=& u^{-(\lengthofv+1)} (g.a - g.a|_{u^{\lengthofv}}) \in L_{[0,m-1]}\bG_a(\bar{k}) \nonumber \\
R &=& u^{-1} (\tau(a) - a). 
\end{eqnarray}

\end{prop}

\begin{lm}\label{eq:def_of_xi_chi} Let $\chi$ be a character of $E^{\times}$. We have $\coh_c^0(\tilde{Y}_{\dot{w}}^m)[\chi] \cong \coh_c^0(Y_{\dot{w}}^m)[\chi|_{U_E}]$.
\end{lm}
\begin{proof}
The proof is the same as in \cite{Iv2} Lemma 4.5. 
\end{proof}

\begin{lm}\label{rem:cancelling_terms}
Let $\lengthofv \leq s \leq 2\lengthofv$ be positive integers. Let $f \in k[u]/(u^s)$ and let $h \colon k[u]/(u^{\lengthofv}) \rar k[u]/(u^{s-\lengthofv})$ be some map. Then for $x \in k[u]/(u^s)$ we have

\[ x = f + u^{\lengthofv} h(x \hspace{-0.2cm} \mod u^{\lengthofv})  \quad \LRar \quad  x = f + u^{\lengthofv} h(f \hspace{-0.2cm} \mod u^{\lengthofv}) \]

\noindent (both equalities take place in $k[u]/(u^s)$).
\end{lm}
\begin{proof}
This is trivial. \qedhere
\end{proof}

\begin{proof}[Proof of Proposition \ref{lm:our_form_of_Boy_trace_formula}] The action of $g$ on $\tilde{Y}_{\dot{w}}^m$ fixes $Y_{\dot{w}}^m$. By Lemma \ref{eq:def_of_xi_chi} we have 

\[ \tr(g ; \coh_c^0(\tilde{Y}_{\dot{w}}^m)[\chi]) = \tr(g ; \coh_c^0(Y_{\dot{w}}^m)[\chi|_{U_E}]).\] 

\noindent We have $\sharp I_{m,\underline{w}_m}/I^m = (q-1)q^{2m}$. Applying Lemma \ref{lm:Boy_trace_formula} to the left action of $U_{\fJ}$ and the right action of $I_{m,\underline{w}_m}/I^m$ on $Y_{\dot{w}}^m$ and $\Frob_q$ (this is possible, as only the zeroth cohomology is non-vanishing, and as the Frobenius acts as a scalar in $\coh_c^0$), we deduce

\[ \tr(g ; \coh_c^0(\tilde{Y}_{\dot{w}}^m)[\chi]) = \frac{1}{(q-1)q^{2m}} \sum_{i \in I_{m,\underline{w}_m}/I^m} \sharp S_{g,i} \chi(i), \] 

\noindent where $S_{g,i}$ is the set of points $y \in Y_{\dot{w}}^m$ with $g.y = y.i$ (note that any point in $Y_{\dot{w}}^m$ has coordinates in $k$, hence Frobenius acts trivial). Further, note that a point of $Y_{\dot{w}}^m$ is uniquely determined by its coordinates $a,C,A$ (cf. Definition \ref{def:disc_subschemes_Yvm2}). Write $i = \matzz{i_1}{}{}{\tau(i_1)} \matzz{1}{i_2}{}{1}$ with $i_1 \in U_E/U_E^{m+1}$, $i_2 \in k[u]/u^m$. As the determinant is multiplicative, we see that $S_{g,i} = \emptyset$, unless $\det(g) \equiv \det(i) = i_1 \tau(i_1) \mod u^{m+1}$. Assume this holds. By Propositions \ref{prop:left_I_action_mega_formula} and \ref{prop:right_action_formula}, we see that $\sharp S_{g,i}$ is equal to the number of solutions of the equations
\begin{eqnarray} 
\label{eq:hilfs_eq_for_Trace_det_sthnr_1} g.a|_{\lengthofv} &\equiv& a \mod u^{\lengthofv + 1} \\
\label{eq:hilfs_eq_for_Trace_det_sthnr_2}\frac{\det(g)}{g_2 a + g_1} C N^{-1} &\equiv& C i_1 H^{-1} \mod u^{m+1} \\
\label{eq:hilfs_eq_for_Trace_det_sthnr_3} AN + h(g,a) \frac{(g_2 a + g_1)^2 D N^2}{\det(g) C} &\equiv& i_1^{-1} \tau(i_1)H^2 A + i_2 H \mod u^m
\end{eqnarray}

\noindent in the variables $a = \sum_{i=1}^{\lengthofv} a_iu^i + \sum_{i=\lengthofv+1}^m 0 u^i \in L^{\leq \lengthofv}_{[1, \lengthofv + m]} \bG_a(k)$ (with $a_1 \neq 0$), $C \in (k[u]/u^{m+1})^{\times}$ and $A \in k[u]/u^m$, where 
\begin{eqnarray*}
B &=& u^{\lengthofv} C\tau(C)^{-1}  \nonumber  \\
D &=& R^{-1} \tau(C)(1 + u^{\lengthofv} C\tau(C)^{-1}A  - u^{\lengthofv} C^{-1}\tau(C) \tau(A) ) \nonumber 
\end{eqnarray*}

\noindent (as we are in $Y_{\dot{w}}^m$; here $R = u^{-1}(\tau(a) - a)$) and $h(g,a)$ and

\begin{eqnarray*} 
N &=& 1 + u^{\lengthofv + 1} \frac{g_2}{g_2 a + g_1} C D^{-1} A \equiv 1 + u^{\lengthofv + 1} \frac{g_2}{g_2 a + g_1} R C \tau(C)^{-1} A \mod u^{m+1} \\
H &=& 1 + i_1 \tau(i_1)^{-1} i_2 B = 1 + u^{\lengthofv} i_1 \tau(i_1)^{-1} i_2 C\tau(C)^{-1} 
\end{eqnarray*}

\noindent are as in Propositions \ref{prop:left_I_action_mega_formula} and \ref{prop:right_action_formula}. As the character $\chi$ of $I_{m,\underline{w}_m}/I^m$ is inflated from a character of $U_E/U_E^{m+1}$ (again denoted by $\chi$), we see that 

\[ \sum_{i \mapsto i_1} \sharp S_{g,i} \chi(i) = \sharp S_{g,i_1}^{\prime} \chi(i_1), \]

\noindent where $i$ varies through all elements of $I_{m,\underline{w}_m}/I^m$ lying over $i_1$ and $\sharp S_{g,i_1}^{\prime}$ is the number of solutions of equations \eqref{eq:hilfs_eq_for_Trace_det_sthnr_1}, \eqref{eq:hilfs_eq_for_Trace_det_sthnr_2}, \eqref{eq:hilfs_eq_for_Trace_det_sthnr_3} in the variables $a,C,A,i_2$. It is enough to show that $\sharp S_{g,i_1}^{\prime} = (q-1)q^{2m} \sharp S_{g,i_1}$. If $\lengthofv \geq m + 1$, then $N,H \equiv 1 \mod u^{m+1}$ and the proof is immediate. Assume $\lengthofv \leq m \leq 2\lengthofv$. We cancel $C$ in \eqref{eq:hilfs_eq_for_Trace_det_sthnr_2} and insert the condition on the determinant to bring it to the form

\begin{equation}\label{eq:Gl_2_einfachere_Form_allg} 
\frac{\tau(i_1)}{g_2 a + g_1} H \equiv N \mod u^{m+1}.
\end{equation}

\noindent By replacing $N$ by $\frac{\tau(i_1)}{g_2 a + g_1} H$ in \eqref{eq:hilfs_eq_for_Trace_det_sthnr_3} and canceling the invertible term $H$ we see that the equations \eqref{eq:hilfs_eq_for_Trace_det_sthnr_1}, \eqref{eq:hilfs_eq_for_Trace_det_sthnr_2}, \eqref{eq:hilfs_eq_for_Trace_det_sthnr_3} are equivalent to the three equations \eqref{eq:hilfs_eq_for_Trace_det_sthnr_1}, \eqref{eq:Gl_2_einfachere_Form_allg} and 

\begin{equation}\label{eq:Gl_3_intermediate_form}
i_2 \equiv \frac{h(g,a)\tau(i_1)H D}{i_1 C} + \frac{\tau(i_1)A}{g_2 a + g_1}  - \frac{\tau(i_1) HA}{i_1} \mod u^m. 
\end{equation}

\noindent Using $H \equiv 1 \mod u^{\lengthofv}$ and $D \equiv R^{-1}\tau(C) \mod u^{\lengthofv}$ equation \eqref{eq:Gl_3_intermediate_form} implies:

\begin{equation}\label{eq:Gl_3_mod_ell}
i_2 \equiv \frac{h(g,a) \tau(i_1) \tau(C)}{i_1 R C}   + \frac{\tau(i_1)}{g_2 a + g_1} A - \frac{\tau(i_1) A}{i_1} \mod u^{\lengthofv} 
\end{equation}

\noindent (the right hand side does not depend on $i_2$). We can replace $i_2$ occurring in the term $H$ in \eqref{eq:Gl_2_einfachere_Form_allg} by the right hand side of \eqref{eq:Gl_3_mod_ell} and hence our three original equations \eqref{eq:hilfs_eq_for_Trace_det_sthnr_1}, \eqref{eq:hilfs_eq_for_Trace_det_sthnr_2}, \eqref{eq:hilfs_eq_for_Trace_det_sthnr_3} are equivalent to \eqref{eq:hilfs_eq_for_Trace_det_sthnr_1}, 

\begin{eqnarray}\label{eq:Gl_2_neue_Form}
\tau(i_1) (1 + u^{\lengthofv}(\frac{h(g,a)}{R} + \frac{i_1}{g_2 a + g_1} C\tau(C)^{-1}A - C\tau(C)^{-1}A) ) \equiv \\ 
\nonumber \equiv (g_2 a + g_1) + u^{\lengthofv + 1} g_2 R C \tau(C)^{-1} A \mod u^{m+1}
\end{eqnarray}

\noindent and \eqref{eq:Gl_3_intermediate_form}. By Lemma \ref{rem:cancelling_terms} applied to $x = i_2$, equation \eqref{eq:Gl_3_intermediate_form} is just an expression of $i_2$ in terms of $g,i_1,a,C,A$, hence it can be ignored and we see that $\sharp S_{g,i_1}^{\prime}$ is the number of solutions of \eqref{eq:hilfs_eq_for_Trace_det_sthnr_1} and \eqref{eq:Gl_2_neue_Form} in the variables $a,C,A$. 

Now, \eqref{eq:Gl_2_neue_Form} implies $\tau(i_1) \equiv g_2 a + g_1 \mod u^{\lengthofv}$. Applying Lemma \ref{rem:cancelling_terms} to $x = \tau(i_1)$, we see that \eqref{eq:Gl_2_neue_Form} is equivalent to 

\begin{eqnarray}\nonumber 
\tau(i_1) &+& u^{\lengthofv} ( (g_2 a + g_1)h(g,a)R^{-1} + (g_2 \tau(a) + g_1) C\tau(C)^{-1}A - (g_2 a + g_1) C\tau(C)^{-1}A) \equiv \\ 
\label{eq:Gl_2_neue_Form_2} &\equiv& (g_2 a + g_1) + u^{\lengthofv + 1} g_2 R C \tau(C)^{-1} A \mod u^{m+1}.
\end{eqnarray}

\noindent Inserting on the right hand side $R = u^{-1}(\tau(a) - a)$, we immediately see that \eqref{eq:Gl_2_neue_Form_2} is equivalent to \eqref{eq:Gl_2_v1}. Moreover, \eqref{eq:hilfs_eq_for_Trace_det_sthnr_1} is immediately seen to be equivalent to \eqref{eq:Gl_1_v1}. As in \eqref{eq:Gl_1_v1}, \eqref{eq:Gl_2_v1} neither $C$, nor $A$ occur, and as $C$ lives in $(k[u]/u^{m+1})^{\times}$ and $A$ lives in $k[u]/u^m$, we deduce that $\sharp S_{g,i_1}^{\prime} = (q-1)q^{2m} \sharp S_{g,i_1}$. \qedhere
\end{proof}

We now examine solutions of the equation \eqref{eq:Gl_1_v1} in $a \in L^{\leq \lengthofv}_{[1, \lengthofv + m]} \bG_a(k)$ (with $a_1 \neq 0$). Recall that via the embedding $\iota$ (see Section \ref{sec:subgroups_of_GF}) we have the subgroups $U_F U_{\fJ}^{\lengthofv} \subseteq U_E U_{\fJ}^{\lengthofv} \subseteq U_{\fJ}$.

\begin{lm}\label{lm:on_equation_for_a}
Let $g \in U_{\fJ}$. Precisely one of the following cases occur:
\begin{itemize}
\item[(i)] $g \in U_F U_{\fJ}^{\lengthofv}$. Then \eqref{eq:Gl_1_v1} has precisely $(q-1)q^{\lengthofv - 1}$ solutions.
\item[(ii)] $g \in U_{\fJ} \sm U_F U_{\fJ}^{\lengthofv}$ is conjugate in $U_{\fJ}$ to an element of $U_E U_{\fJ}^{\lengthofv}$. In this case \eqref{eq:Gl_1_v1} has precisely $2q^{v_u(g_3) - 1}$ solutions.
\item[(iii)] $g \in U_{\fJ} \sm U_F U_{\fJ}^{\lengthofv}$ is not conjugate in $U_{\fJ}$ to an element of $U_E U_{\fJ}^{\lengthofv}$. Then \eqref{eq:Gl_1_v1} has no solutions. 
\end{itemize}
\end{lm}

\begin{proof}
Assume \eqref{eq:Gl_1_v1} has a solution $a$. As $g \in U_{\fJ}$, the integers $v_u(g_2), v_u(g_3), v_u(g_1 - g_4)$ are even. As $a_1 \neq 0$, we have $v_u(a) = 1$. We deduce that $v_u((g_1 - g_4)a)$ is odd and $v_u(g_2 a^2)$, $v_u(g_3)$ are even. Thus by Lemma \ref{lm:charac_of_conjugacy_to_nonsplit_ram_torus} we are either in the case $g \in U_F U_{\fJ}^{\lengthofv}$ of the lemma, where each of these three integers is $\geq \lengthofv+1$ and each element of $L^{\leq \lengthofv}_{[1, \lengthofv + m]} \bG_a(k)$ solves equation \eqref{eq:Gl_1_v1}, or we are forced to have $v_u(g_3) = v_u(g_2) + 2 < \lengthofv + 1$ and $v_u(g_3) \leq v_u(g_1 - g_4) + 1$ (this last is, using parity, equivalent to $v_u(g_3) \leq v_u(g_1 - g_4)$). In the last case write $g_2 = g_2^{\prime} u^{v_u(g_2)}$, $g_3 = g_3^{\prime} u^{v_u(g_2) + 2}$, $a = a^{\prime}u$ and $g_1 - g_4 = g_{1,4}^{\prime} u^{v_u(g_2) + 2}$ with $g_2^{\prime}, g_3^{\prime},a^{\prime} \in k\llbracket t\rrbracket^{\times}$ and $g_{1,4}^{\prime} \in k\llbracket t\rrbracket$. After canceling $u^{v_u(g_3)} = u^{v_u(g_2) + 2}$, \eqref{eq:Gl_1_v1} is equivalent to

\begin{equation}\label{eq:eine_gl_hilf_mod_bla}
g_2^{\prime} a^{\prime,2} + g_{1,4}^{\prime}a^{\prime} u - g_3^{\prime} \equiv 0 \mod u^{\lengthofv + 1 - v_u(g_3)}, 
\end{equation}

\noindent where $\lengthofv + 1 - v_u(g_3) \geq 1$. Reducing modulo $u$, we deduce $a_1^2 \equiv \frac{g_3^{\prime}}{g_2^{\prime}} \mod u$, which shows that $\frac{g_3^{\prime}}{g_2^{\prime}} \mod u$ must be a square of an element of $k^{\times}$, or, equivalently (cf. Lemma \ref{lm:charac_of_squares}), that $\frac{g_3}{t g_2} \in k\llbracket t \rrbracket^{\times}$ is a square. Thus by Lemma \ref{lm:charac_of_conjugacy_to_nonsplit_ram_torus} we deduce that we must be in case (ii) of the lemma and that in case (iii) there are no solutions. In case (ii) with notations as above, we have to determine how many solutions in $a^{\prime} = a_1 + a_2u + \dots + a_{\lengthofv}u^{\lengthofv - 1}$ equation \eqref{eq:eine_gl_hilf_mod_bla} has. Using induction, one now easily deduces that there are exactly two possibilities for $a_1$, exactly $1$ possibility for each $a_2, \dots, a_{\lengthofv + 1 - v_u(g_3)}$ and exactly $q$ possibilities for each $a_{\lengthofv + 2- v_u(g_3)}, \dots, a_{\lengthofv}$.
\end{proof}

\begin{lm}\label{lm:charac_of_conjugacy_to_nonsplit_ram_torus}
Let $g \in U_{\fJ}$ and $\lengthofv \geq 1$. Then
\begin{itemize}
\item[(i)] $g \in U_F U_{\fJ}^{\lengthofv} \LRar v_u(g_2) \geq \lengthofv - 1$, $v_u(g_3) \geq \lengthofv + 1$, $v_u(g_1 - g_4) \geq \lengthofv$. 
\item[(ii)] $g \in U_{\fJ} \sm U_F U_{\fJ}^{\lengthofv}$ and $g$ is conjugate to an element of $U_E U_{\fJ}^{\lengthofv}$ if and only if $v_u(g_3) = v_u(g_2) + 2 < \lengthofv + 1$, $v_u(g_3) \leq v_u(g_1 - g_4)$ and $\frac{g_3}{tg_2} \in k\llbracket t \rrbracket^{\times}$ is a square of an element in $k\llbracket u \rrbracket^{\times}$
\end{itemize} 
\end{lm}

\begin{proof}
(i): is an easy computation (use that $v_u(g_j)$ is always even). (ii): In the $\caO_F$-algebra $\fJ$ the subset $\varpi^{\lengthofv} \fJ$ form a two-sided ideal and $U_{\fJ}/U_{\fJ}^{\lengthofv} = (\fJ/\varpi^{\lengthofv}\fJ)^{\times}$. Assume $g \in U_{\fJ} \sm U_F U_{\fJ}^{\lengthofv}$ and $v_u(g_3) = v_u(g_2) + 2 < \lengthofv + 1$, $v_u(g_3) \leq v_u(g_1 - g_4)$ and $\frac{g_3}{tg_2} \in k\llbracket t \rrbracket^{\times}$ is a square of an element in $k\llbracket u \rrbracket^{\times}$. We replace $U_{\fJ}$ (resp. $\fJ$) by $U_{\fJ}/U_{\fJ}^{\lengthofv}$ (resp. $\fJ/\varpi^{\lengthofv}\fJ$) and $g$ by its image there. We show that $g$ is conjugate to an element of $U_E/U_E^{\lengthofv} = U_E/U_E \cap U_{\fJ}^{\lengthofv}$. Replace $g$ by the difference of $g$ and the scalar matrix with entries $\frac{1}{2}(g_1 + g_4)$. Thus we can assume that $g$ has trace zero and we must show that there is some $b \in \caO_F$ such that $g$ is conjugate in $\fJ/\varpi^{\lengthofv}\fJ$ to the image of $\matzz{}{b}{tb}{}$. Consider $r_{y,\lambda}$ from Lemma \ref{lm:rep_system_for_UJ_UEUJell}. Note that 

\[ r_{y,\lambda} \matzz{}{b}{tb}{} r_{y,\lambda}^{-1} = \matzz{b\lambda t}{by^{-1}(1-\lambda^2 t)}{byt}{-b\lambda t} \]

\noindent By our assumptions we can write $g_2 = t^{\alpha}g_2^{\prime}$, $g_3 = t^{\alpha + 1}g_3^{\prime}$, $g_1 = - g_4 = t^{\alpha+1}g_1^{\prime}$ with $\alpha + 1 \leq \lfloor \frac{\lengthofv}{2} \rfloor$ and $g_2^{\prime},g_3^{\prime} \in k \llbracket t \rrbracket^{\times}$. Thus we can conclude, if we find appropriate $y \in U_F/U_F^{\lfloor \frac{\lengthofv+1}{2} \rfloor }$, $\lambda \in \caO_F/\caO_F^{\lfloor \frac{\lengthofv}{2} \rfloor}$ and $b = b_0 t^{\alpha}\in \caO_F$ with $b_0 \in U_F$ such that
\begin{eqnarray}
\nonumber b_0\lambda &\equiv& g_1^{\prime} \mod t^{\lfloor \frac{\lengthofv + 1}{2} \rfloor - (\alpha + 1)} \\
\label{eq:irgendwelche_hilfsgleichungen} b_0y &\equiv& g_3^{\prime}  \mod t^{\lfloor \frac{\lengthofv}{2}\rfloor - \alpha} \\ 
\nonumber b_0y^{-1}(1- \lambda^2 t) &\equiv& g_2^{\prime} \mod t^{\lfloor \frac{\lengthofv}{2} \rfloor - \alpha}
\end{eqnarray}

\noindent Using the first and the second equations to eliminate $b_0$ and $\lambda$, the only remaining equation is 
\[ y^2 \equiv g_3^{\prime}g_2^{\prime,-1} (1 - g_1^{\prime,2}g_3^{\prime,-2}y^2 t ) \mod t^{\lfloor \frac{\lengthofv}{2} \rfloor - \alpha} \]
This equation has a solution in $y$ by Hensel's lemma and our assumption on $\frac{g_3}{tg_2}$. The other direction in (ii) is an immediate computation.\qedhere
\end{proof}


\subsection{Traces of unipotent elements}\label{sec:computations_of_unipotent_traces} \mbox{}

In Sections \ref{sec:computations_of_unipotent_traces}-\ref{sec:trace_on_induced_side} we assume $m = 2\lengthofv - 1$.

\begin{proof}[Proof of Lemma \ref{lm:UJm1_acts_trivial}]

We use notations of Proposition \ref{lm:our_form_of_Boy_trace_formula}. Let $g \in U_{\fJ}^{m+1}$. Thus $v_u(g_1 - 1),v_u(g_2),v_u(g_4 - 1) \geq 2\lengthofv = m+1$ and $v_u(g_3) \geq m+3$. This, Proposition \ref{lm:our_form_of_Boy_trace_formula} and Lemma \ref{sublm:hga_comp_for_UJm1}(i) show that $\sharp S_{g,i_1} = 0$ for $i_1 \in U_E/U_E^{m+1} \sm \{1\}$. Lemma \ref{lm:on_equation_for_a} implies $\sharp S_{g,1} = (q-1)q^{\lengthofv - 1}$. Proposition \ref{lm:our_form_of_Boy_trace_formula} shows $\tr(g;\Xi_{\chi}) = (q-1)q^{\lengthofv - 1}$. \qedhere

%
%
\end{proof}

\begin{proof}[Proof of Lemma \ref{eq:claim_for_unip_traces}]
We use notations from Proposition \ref{lm:our_form_of_Boy_trace_formula}. The case $g = 1$ of Lemma \ref{eq:claim_for_unip_traces} follows from Lemma \ref{lm:UJm1_acts_trivial}. Write $\delta := \lfloor \frac{\lengthofv + 1}{2} \rfloor - \lfloor \frac{\lengthofv}{2} \rfloor$. For $0 \leq  \alpha \leq \lfloor \frac{\lengthofv + 1}{2} \rfloor - 1$ consider the subgroup

\[A_{\alpha} := \{1 + u^{(\lengthofv - \delta) + 2\alpha + 1} y \colon y \text{ is $\tau$-invariant} \} \]

\noindent of $U_E^{(\lengthofv - \delta) + 2\alpha + 1}/U_E^{m+1}$, and let $A_{\lfloor \frac{\lengthofv + 1}{2} \rfloor } := \{ 1 \} \subseteq U_E/U_E^{m+1}$.

\begin{lm}\label{lm:sizes_of_sets_S_unip_traces}
Let $g \in N_{\lengthofv} \sm \{1\}$. If $g \not\in N_{\lengthofv}^{\lfloor \frac{\lengthofv}{2} \rfloor + 1}$, then $S_{g,i_1} = \emptyset$ for all $i_1 \in U_E/U_E^{m+1}$. Otherwise, let $g \in N_{\lengthofv}^{\lfloor \frac{\lengthofv}{2} \rfloor + 1 + \alpha} \sm N_{\lengthofv}^{\lfloor \frac{\lengthofv}{2} \rfloor + 2 + \alpha}$ for some $0 \leq  \alpha \leq \lfloor \frac{\lengthofv + 1}{2} \rfloor - 1$ and $i_1 \in U_E/U_E^{m+1}$. Then 

\[ \sharp S_{g,i_1} = \begin{cases} c(\alpha) & \text{if $i_1 \in A_{\alpha} \sm A_{\alpha + 1}$} \\ 0 & \text{otherwise}, \end{cases} \]

\noindent where $c(\alpha)$ depends only on $\alpha$, not on $i_1$. Moreover, $c(\lfloor \frac{\lengthofv + 1}{2} \rfloor - 1) = q^{\lengthofv - 1}$.
\end{lm}

\begin{proof}
If $g \in N_{\lengthofv} \sm N_{\lengthofv}^{\lfloor \frac{\lengthofv}{2} \rfloor + 1}$, then $g$ is not conjugate to an element of $U_EU_{\fJ}^{\lengthofv}$ by Lemma \ref{lm:charac_of_conjugacy_to_nonsplit_ram_torus}, so $S_{g,i_1} = \emptyset$ for all $i_1 \in U_E/U_E^{m+1}$ by Lemma \ref{lm:on_equation_for_a}, and the first statement of the lemma follows from Proposition \ref{lm:our_form_of_Boy_trace_formula} (alternatively, look at equation \eqref{eq:Gl_1_v1} for $g$). Let $g \in N_{\lengthofv}^{\lfloor \frac{\lengthofv}{2} \rfloor + 1 + \alpha} \sm N_{\lengthofv}^{\lfloor \frac{\lengthofv}{2} \rfloor + 2 + \alpha}$ for some $0 \leq  \alpha \leq \lfloor \frac{\lengthofv + 1}{2} \rfloor - 1$ and $i_1 \in U_E/U_E^{m+1}$. Write $g = \matzz{1}{}{t^{\lfloor \frac{\lengthofv}{2}\rfloor + 1 + \alpha }x}{1}$ with $v_t(x) = 0$. Then equation \eqref{eq:Gl_1_v1} is trivially satisfied for each $a$ and equation \eqref{eq:Gl_2_v1} takes the form

\begin{equation} \label{eq:Gl_1_unip_version} 
i_1 \equiv 1 + u^{(\lengthofv - \delta) + 2\alpha + 1} x R^{-1} \mod u^{m+1}
\end{equation}

\noindent (one easily computes $h(g,a) = u^{2\alpha + 1 - \delta}x$). Write $a = \sum_{i=1}^{\lengthofv} a_i u^i$. Then $R = u^{-1}(\tau(a) - a) = -2(a_1 + a_3 u^2 + \dots)$ and $x$ are $\tau$-invariant and we have $v_t(R) = v_t(x) = 0$. Hence $S_{g,i_1} = \emptyset$ unless $i_1 \in A_{\alpha} \sm A_{\alpha + 1}$. On the other hand, from the explicit form of $R$, it is clear that for any $i_1 \in A_{\alpha} \sm A_{\alpha + 1}$ the set $S_{g,i_1}$ of solutions $a$ of \eqref{eq:Gl_1_unip_version} has the same cardinality. The second statement of the lemma follows. To see the last statement, put $\alpha = \lfloor \frac{\lengthofv + 1}{2} \rfloor - 1$. Then $(\lengthofv - \delta) + 2\alpha + 1 = 2\lengthofv - 1 = m$ and for a fixed $i_1 \in A_{\lfloor \frac{\lengthofv + 1}{2} \rfloor - 1} \sm A_{\lfloor \frac{\lengthofv + 1}{2} \rfloor} = U_E^m/U_E^{m+1} \sm \{1\}$ equation \eqref{eq:Gl_1_unip_version} amounts to a condition on $R \mod u$, or, which is the same, on $a_1$. It determines $a_1$ uniquely and $a_2, a_3, \dots, a_{\lengthofv}$ can be chosen arbitrarily. Thus \eqref{eq:Gl_1_unip_version} has exactly $q^{\lengthofv - 1}$ solutions.
\end{proof}

Now we can finish the proof of the Lemma \ref{eq:claim_for_unip_traces}. Let $g \in N_{\lengthofv} \sm N_{\lengthofv}^{\lengthofv}$. If $g \not\in N_{\lengthofv}^{\lfloor \frac{\lengthofv}{2} \rfloor + 1}$, then Proposition \ref{lm:our_form_of_Boy_trace_formula} and the first statement in Lemma \ref{lm:sizes_of_sets_S_unip_traces} immediately show $\tr(g;\Xi_{\chi}) = 0$. Otherwise, there is some $\alpha$ with $0 \leq \alpha < \lfloor \frac{\lengthofv + 1}{2} \rfloor - 1$, such that $g \in N_{\lengthofv}^{\lfloor \frac{\lengthofv}{2} \rfloor + 1 + \alpha} \sm N_{\lengthofv}^{\lfloor \frac{\lengthofv}{2} \rfloor + 2 + \alpha}$, and we deduce from Proposition \ref{lm:our_form_of_Boy_trace_formula} and Lemma \ref{lm:sizes_of_sets_S_unip_traces}

\[ \tr(g;\Xi_{\chi}) = \sum_{i_1 \in A_{\alpha} \sm A_{\alpha + 1}} c(\alpha) \chi(i_1) = c(\alpha) \sum_{i_1 \in A_{\alpha}} \chi(i_1) - c(\alpha) \sum_{i_1 \in A_{\alpha + 1}}  \chi(i_1) = 0, \]

\noindent as $A_{\alpha}$, $A_{\alpha + 1}$ both are subgroups containing $U_E^m/U_E^{m+1}$ and $\chi$ is a non-trivial character on $U_E^m/U_E^{m+1}$. 
Now assume $g \in N_{\lengthofv}^{\lengthofv} \sm \{ 1 \}$. This corresponds to $\alpha = \lfloor \frac{\lengthofv + 1}{2} \rfloor - 1$ and $A_{\lfloor \frac{\lengthofv + 1}{2} \rfloor - 1} = U_E^m/U_E^{m+1}$. By Proposition \ref{lm:our_form_of_Boy_trace_formula} and Lemma \ref{lm:sizes_of_sets_S_unip_traces} we compute

\[ \tr(g;\Xi_{\chi}) = \sum_{i_1 \in U_E^m/U_E^{m+1} \sm \{1\}} q^{\lengthofv - 1} \chi(i_1) = - q^{\lengthofv - 1}, \]

\noindent as $\chi$ is non-trivial on $U_E^m/U_E^{m+1}$. This finishes the proof of Lemma \ref{eq:claim_for_unip_traces}.
\end{proof}


\subsection{Traces of some non-split elements} \mbox{}

\begin{prop}\label{prop:non_split_traces}
Let $0 \leq \alpha \leq \lengthofv - 1$. Let $g = \iota(1 + u^{2\alpha + 1}h)$ for some $h \in U_F$. Then 
\begin{equation} \label{eq:traces_of_non_split_elements_in_prop}
\tr(g;\Xi_{\chi}) = q^{\alpha}(\chi(g) + \chi^{\tau}(g)) + 2q^{\alpha} \cdot \sum_{ \substack{i_1 = 1 + u^{2\alpha + 1}hs \\ s \in Q_{\alpha} }} \chi(i_1), 
\end{equation}
\noindent with $Q_{\alpha}$ as in Section \ref{sec:quadratic_distance}.
\end{prop}

\begin{proof}
This follows immediately from Proposition \ref{lm:our_form_of_Boy_trace_formula} and Lemma \ref{lm:traces_of_non_split_elements}.
\end{proof}

\begin{lm}\label{lm:traces_of_non_split_elements}
Let $\alpha$, $g$, $h$ be as in Proposition \ref{prop:non_split_traces}. Then $S_{g,i_1} = \emptyset$, unless $i_1 = 1 + u^{2\alpha + 1}hs$ for some $s \in R_{\alpha}$. Assume this holds. Then
\[ S_{g,i_1} = \begin{cases} 2q^{\alpha} & \text{ if } s \in Q_{\alpha} \\ q^{\alpha} & \text{ if } s = \pm 1 \\ 0 & \text{otherwise.} \end{cases} \]
\end{lm}

\begin{proof}
Assume first $\alpha \geq \lfloor \frac{\lengthofv}{2} \rfloor$, or equivalently $2\alpha + 1 \geq \lengthofv$. In this case $\sharp S_{g,i_1}$ is equal to the number of solutions of \eqref{eq:Gl_2_v1} in the variable $a$. Using Lemma \ref{sublm:hga_comp_for_UJm1}(ii), we see that $S_{g,i_1} = \emptyset$, unless $i_1 = 1 + u^{2\alpha + 1}hs$ for some $s \in R_{\alpha}$. Assume this holds. As $h$ is $\tau$-invariant, it follows that the condition $\det(g) \equiv i_1 \tau(i_1) \mod u^{m+1}$ (necessary for the non-emptiness of $S_{g,i_1}$) is equivalent to $s \in R_{\alpha}^{\langle \tau \rangle}$. Thus we can assume that $i_1 = 1 + u^{2\alpha + 1}hs$ with $s \in R_{\alpha}^{\langle \tau \rangle}$. Then, using Lemma \ref{sublm:hga_comp_for_UJm1}(ii), \eqref{eq:Gl_2_v1} is seen to be equivalent to

\[ s \equiv (1 - a^{\prime,2})(1 - u^{2\alpha + 1}ha^{\prime})R^{-1} - a^{\prime} \mod u^{m - 2\alpha}, \]

\noindent where we write $a = ua^{\prime}$. By assumption $2\alpha + 1 \geq m - 2\alpha$, and moreover, $R = u^{-1}(\tau(a) - a) = -(\tau(a^{\prime}) + a^{\prime})$. Hence \eqref{eq:Gl_2_v1} is equivalent to

\begin{equation}\label{eq:Gl2incase_alphagreater_l2} 
s \equiv - \frac{1 + a^{\prime}\tau(a^{\prime})}{a^{\prime} + \tau(a^{\prime})} \mod u^{m - 2\alpha}.
\end{equation}
\noindent Assume this equation has a solution in $a^{\prime}$. Then we deduce 

\[ s\tau(s) - 1 = s^2 - 1 = \left(\frac{1 + a^{\prime}\tau(a^{\prime})}{a^{\prime} + \tau(a^{\prime})}\right)^2 - 1 = \N_{\tau, \alpha} \left( \frac{1 - a^{\prime,2}}{a^{\prime} + \tau(a^{\prime})} \right) \]

\noindent in $R_{\alpha}$. This shows that if $S_{g,i_1} \neq \emptyset$, then $s \in Q_{\alpha} \cup \{\pm 1\}$. Conversely, assume that $s \in Q_{\alpha} \cup \{\pm 1\}$. Write $a^{\prime} = a_1 + a_2u + \dots + a_{\lengthofv}u^{\lengthofv - 1}$. We differ between three cases. 

\textbf{Case 1.} $s \not\equiv \pm 1 \mod u$. Let $s_0 := s \mod u$. By Lemma \ref{lm:study_of_Qalpha12_sizes_and_maps}, $s_0^2 - 1$ is a square in $k^{\times}$. We can rearrange the equation \eqref{eq:Gl2incase_alphagreater_l2} and bring it to the form
\begin{eqnarray*} 
(1 + a_1^2) &+& (2a_1 a_3 - a_2^2)u^2 + \dots + (2\sum_{j = 0}^{i-1} (-1)^j a_{j+1} a_{2i - j + 1} + (-1)^i a_{i+1}^2) u^{2i} + \dots \\
  \dots &+& (2a_1 a_{2\lfloor \frac{\lengthofv+1}{2} \rfloor - 1} + \dots ) u^{2\lfloor \frac{\lengthofv+1}{2} \rfloor - 2} \equiv -2s(a_1 + a_3 u^2 + \dots + a_{2\lfloor \frac{\lengthofv + 1}{2} \rfloor - 1}u^{2\lfloor \frac{\lengthofv + 1}{2} \rfloor - 2} ) \mod u^{m - 2\alpha}. 
\end{eqnarray*}

\noindent Taking this equation modulo $u$, we obtain the equation  $a_1^2 + 2 s_0 a_1 + 1 = 0$ in $k$. It has precisely two different solutions in $a_1$ as $s_0^2 - 1$ is a square in $k^{\times}$. Note that both solutions satisfy $a_1 \neq - s_0$ due to $s_0 \neq \pm 1$. Taking the above equation iteratively modulo $u^3, \dots, u^{m-2\alpha}$ and using $a_1 \neq - s_0$, we see that there are exactly $q$ possibilities to choose any of the pairs $(a_2,a_3), \dots, (a_{m-2\alpha - 1}, a_{m-2\alpha})$ and we obtain $q$ possibilities for each of the remaining variables $a_{m-2\alpha}, \dots, a_{\lengthofv}$ (note that $m- 2\alpha \leq 2\lfloor \frac{\lengthofv + 1}{2} \rfloor - 1$). Altogether we obtain $2q^{\lengthofv - \alpha - 1} q^{\lengthofv - (m-2\alpha)} = 2q^{\alpha}$ solutions.

\textbf{Case 2.} $v_u(s + 1) = 2j$ or $v_u(s - 1) = 2j$ with $0 < 2j < m - 2\alpha$ (note that the $v_u(s \pm 1)$ has to be even, as $s$ is $\tau$-invariant). We assume $v_u(s - 1) = 2j$ (the other case is similar). Then we can write $s = 1 + u^{2j}s^{\prime}$ for some $\tau$-invariant unit $s^{\prime}$. Then \eqref{eq:Gl2incase_alphagreater_l2} is equivalent to 

\[ (1 + a^{\prime})(1 + \tau(a^{\prime})) \equiv -u^{2j} s^{\prime} (a^{\prime} + \tau(a^{\prime})) \mod u^{m - 2\alpha}, \]

\noindent and we deduce that a solution $a^{\prime}$ must satisfy $v_u(1 + a^{\prime}) = j$ (as $s^{\prime}, a^{\prime} + \tau(a^{\prime})$ are necessarily units and $v_u(1 + a^{\prime}) = v_u(1 + \tau(a^{\prime}))$). Set $a^{\prime} = - 1 + u^j b$ with some $b = \sum_{i = 0}^{\lengthofv - j - 1} b_i u^i \in (k[u]/u^{\lengthofv - j})^{\times}$. The number of solutions of \eqref{eq:Gl2incase_alphagreater_l2} in $a^{\prime}$ is equal to the number of solutions of 

\begin{equation} \label{eq:version_of_Gl_2_in_some_interm_case}
(-1)^j b\tau(b) \equiv s^{\prime}(2 - u^j(b + (-1)^j \tau(b) )) \mod u^{m -2\alpha - 2j}
\end{equation}

\noindent in the variable $b \in (k[u]/u^{\lengthofv - j})^{\times}$. Taking this equation modulo $u$ we get the equation $(-1)^j b_0^2 \equiv 2s^{\prime} \mod u$. As $s = 1 + u^{2j}s^{\prime} \in Q_{\alpha}$, Lemma \ref{lm:study_of_Qalpha12_sizes_and_maps} shows that $(-1)^j 2s^{\prime} \mod u$ is a square in $k^{\times}$, and thus this equation has exactly two solutions in $b_0$. Similarly as in case 1 above, taking \eqref{eq:version_of_Gl_2_in_some_interm_case} iteratively modulo $u^3, u^5, \dots, u^{m - 2\alpha - 2j}$, we get per step exactly one condition which determines $b_2, b_4, \dots, b_{m-2\alpha - 2j - 1}$ uniquely (note: the set of these conditions also can be empty). For each $b_i$ with $i \not\in \{0,2,4, \dots, m-2\alpha - 2j - 1\}$ there are $q$ possible choices. Thus the number of solutions of \eqref{eq:version_of_Gl_2_in_some_interm_case} in $b$ is equal to $2 q^{(\lengthofv - j - 1) - (\lengthofv - \alpha - j - 1)} = 2q^{\alpha}$.

\textbf{Case 3.} $s = \pm 1$. Assume $s = 1$ (the other case is similar). Then \eqref{eq:Gl2incase_alphagreater_l2} is equivalent to

\[(1 + a^{\prime})(1 + \tau(a^{\prime})) \equiv 0 \mod u^{m - 2\alpha}, \]

\noindent which in turn is equivalent to $v_u(1 + a^{\prime}) \geq \frac{m - 2\alpha + 1}{2} = \lengthofv - \alpha$. We easily deduce that the number of solutions of this equation in $a^{\prime}$ is equal to $q^{\alpha}$. This finishes the case $\alpha \geq \lfloor \frac{\lengthofv}{2} \rfloor$.

Assume now $0 \leq \alpha < \lfloor \frac{\lengthofv}{2} \rfloor$. Then $2\alpha + 1 < \lengthofv$. The quantity $\sharp S_{g,i_1}$ is equal to the number of solutions of \eqref{eq:Gl_1_v1} and \eqref{eq:Gl_2_v1} in $a$. We again write $a = u a^{\prime}$. Equation \eqref{eq:Gl_1_v1} is immediately seen to be equivalent to $a^{\prime} \equiv \pm 1 \mod u^{\lengthofv - 2\alpha - 1}$ and we write $a^{\prime} = \pm 1 + u^{\lengthofv - 2\alpha - 1}b$ for $b \in k[u]/u^{2\alpha + 1}$. We deduce 

\begin{equation} \label{eq:some_expression_for_R_in_low_level_case}
R \equiv - (a^{\prime} + \tau(a^{\prime})) \equiv \mp 2 - u^{\lengthofv - 2\alpha - 1} ( b + (-1)^{\lengthofv + 1}\tau(b) ) \mod u^{m+1}.  
\end{equation}

\noindent Let us denote the 'automorphic factor' $g_2 a + g_1$ by
\begin{equation} \label{eq:def_of_automorphic_factor_in_low_level_case}
f := g_2 a + g_1 = 1 \pm u^{2\alpha + 1} h + u^{\lengthofv}hb. 
\end{equation}

\noindent By Lemma \ref{sublm:hga_comp_for_UJm1}(iii), the quantity $ \sharp S_{g,i_1}$ is equal to the number of solutions in the variable $b \in k[u]/u^{2\alpha + 1}$ of the equation

\[
\tau(i_1)(1 + u^{\lengthofv} R^{-1} \frac{h(\mp 2b - u^{\lengthofv - 2\alpha - 1} b^2)}{f}) \equiv f \mod u^{m+1},
\]

\noindent or equivalently,

\begin{equation}\label{eq:form_of_Gl_2_in_low_level_case}
\tau(i_1) \equiv f - u^{\lengthofv} R^{-1} h(\mp 2b - u^{\lengthofv - 2\alpha - 1} b^2)   \mod u^{m+1}.
\end{equation}

\noindent Taking this equation modulo $u^{m - 2\alpha} = u^{2\lengthofv - 2\alpha - 1}$ and using \eqref{eq:some_expression_for_R_in_low_level_case} and \eqref{eq:def_of_automorphic_factor_in_low_level_case}, we deduce that $S_{g,i_1} = \emptyset$, unless $i_1 \equiv 1 \mp u^{2\alpha + 1}h \mod u^{m-2\alpha}$, or with other words, $i_1 = 1 + u^{2\alpha + 1} h s$ with $s \in R_{\alpha}$ satisfying $s \equiv \mp 1 \mod u^{m - 4\alpha - 1}$. Assume that this holds. An easy computation shows now that $\det(g) \equiv i_1 \tau(i_1) \mod u^{m+1}$ is equivalent to $s \in R_{\alpha}^{\langle \tau \rangle, \prime}$, so we also can assume this (otherwise, $S_{g,i_1} = \emptyset$). Let us write $s = \mp 1 + u^{m - 4\alpha - 1} \cdot (u^{2j}s_0)$, with $s_0 \in (k[u]/u^{2\alpha - 2j + 1})^{\times}$ $\tau$-invariant with $0 \leq j \leq \alpha + 1$ ($j = \alpha + 1$ corresponds to $s = \pm 1$). Straightforward rearrangements of terms show that  \eqref{eq:form_of_Gl_2_in_low_level_case} is equivalent to

\begin{equation}\label{eq:form_Gl_2_mit_s0_low_level_case}
(\mp 2 - (b + (-1)^{\lengthofv + 1}\tau(b))u^{\lengthofv - 2\alpha - 1} ) u^{2j} s_0 \equiv (-1)^{\lengthofv + 1} b\tau(b) \mod u^{2\alpha + 1}. 
\end{equation}

\noindent If $j = \alpha + 1$, then $s = \pm 1$ and \eqref{eq:form_Gl_2_mit_s0_low_level_case} is equivalent to $b\tau(b) \equiv 0 \mod u^{2\alpha + 1}$. This is equivalent to $b \equiv 0 \mod u^{\alpha + 1}$, and hence \eqref{eq:form_Gl_2_mit_s0_low_level_case} has precisely $q^{\alpha}$ solutions in $b$. Assume $j \leq \alpha$. A potential solution $b$ of \eqref{eq:form_Gl_2_mit_s0_low_level_case} must satisfy $b \equiv 0 \mod u^j$, hence we can write $b = u^j b^{\prime}$ for a $b^{\prime} \in k[u]/u^{2\alpha + 1 - j}$ and rewrite \eqref{eq:form_Gl_2_mit_s0_low_level_case} as 

\begin{equation}\label{eq:form_Gl_2_mit_s0_low_level_case_rewritten}
(\mp 2 - (b^{\prime} + (-1)^{\lengthofv + j + 1}\tau(b^{\prime}))u^{\lengthofv - 2\alpha - 1 + j} ) s_0 \equiv (-1)^{\lengthofv + j + 1} b^{\prime}\tau(b^{\prime}) \mod u^{2\alpha + 1 - 2j}. 
\end{equation}

\noindent Assume first $S_{g,i_1} \neq \emptyset$, i.e., \eqref{eq:form_Gl_2_mit_s0_low_level_case_rewritten} has at least one solution. Taking \eqref{eq:form_Gl_2_mit_s0_low_level_case_rewritten} modulo $u$, we deduce that $\pm (-1)^{\lengthofv + j} 2s_0 \mod u$ is a square in $k^{\times}$, which is by Lemma \ref{lm:study_of_Qalpha12_sizes_and_maps} equivalent to $s \in Q_{\alpha}$. Thus $S_{g,i_1} \neq \emptyset$ implies $s \in Q_{\alpha}$. Conversely, if $s \in Q_{\alpha}$, we can deduce that $\sharp S_{g,i_1} = 2q^{\alpha}$ in the same way as in the case $\alpha \geq \lfloor \frac{\lengthofv}{2} \rfloor$. \qedhere
\end{proof}

We are convinced that there must be a more elegant proof of Lemma \ref{lm:traces_of_non_split_elements}, but we still can not find it. 


\subsection{Traces of elements in $E^{\times}$ with $u$-valuation $1$}\mbox{} \label{sec:traces_of_elements_in_varpi_UJ}

\begin{proof}[Proof of Proposition \ref{prop:traces_of_val_1_elements}] Put $y_1 := e_0(u,-u)$. Consider the automorphism $\tilde{\beta}_g \colon \tilde{Y}_{\dot{w}}^m \rar \tilde{Y}_{\dot{w}}^m$ given by $\tilde{\beta}_g(\dot{x}I) = g \dot{x}y_1^{-1}I^m$. Then $\tilde{\beta}_g$ induces an automorphism of $\coh_c^0(\tilde{Y}_{\dot{w}}^m)$ and hence also an automorphism $\tilde{\beta}_g^{\ast} \colon V_{\chi} \rar V_{\chi}$ of its $\chi$-isotypic quotient. As $y_1$ acts in $V_{\chi}$ as the scalar multiplication\footnote{A subtlety: we suppressed our choice of an identification of $E^{\times}$ with the diagonal quotient of $\tilde{I}_{m,w}$, for which we silently have chosen that $u$ corresponds to $y_1$. This choice determines on the one hand that $y_1$ acts in $V_{\chi}$ by $\chi(u)$, and on the other hand, that we have to evaluate the trace formula using the identifications $\varpi \leftrightarrow u \leftrightarrow y_1 = e_0(u,-u)$.} with $\chi(u)$, we have $\tr(g; \Xi_{\chi}) = \chi(u) \tr(\tilde{\beta}_g^{\ast}; V_{\chi})$. We determine $\tr(\tilde{\beta}_g^{\ast}; V_{\chi})$. As $v_u(g) = 1$, Lemma \ref{lm:action_of_beta_varpi} shows that $gY_{\dot{w}}^m = Y_{\dot{w}}^m y_1$. With other words, $\tilde{\beta}_g$ restricts to an automorphism $\beta_g$ of $Y_{\dot{w}}^m$. Further, $\beta_g$ induces an automorphism $\beta_g^{\ast}$ of $\coh_c^0(Y_{\dot{w}}^m)[\chi|_{U_E}]$. Moreover, the isomorphism from Lemma \ref{eq:def_of_xi_chi} induces a commutative diagram

\centerline{
\begin{xy}\label{diag:character_isos_diag}
\xymatrix{
\coh_c^0(Y_{\dot{w}}^m)[\chi|_{U_E}] \ar[r]^{\sim} \ar[d]^{\beta_g^{\ast}} & \coh_c^0(\tilde{Y}_{\dot{w}}^m)[\chi] \ar[d]^{\tilde{\beta}_g^{\ast}} \\
\coh_c^0(Y_{\dot{w}}^m)[\chi|_{U_E}] \ar[r]^{\sim} & \coh_c^0(\tilde{Y}_{\dot{w}}^m)[\chi]
}
\end{xy}
}

\noindent from which we deduce $\tr(\tilde{\beta}_g^{\ast}; V_{\chi}) = \tr(\beta_g^{\ast}; \coh_c^0(Y_{\dot{w}}^m)[\chi|_{U_E}])$. Lemma \ref{prop:our_form_of_Boy_trace_formula_for_varpi} finishes the proof. \qedhere
\end{proof}

\begin{lm}\label{prop:our_form_of_Boy_trace_formula_for_varpi}
Let $g \in E^{\times}$ with $v_u(g) = 1$. Let $\beta_g$ be the automorphism of $Y_{\dot{w}}^m$ defined by $\beta_g(\dot{x}I^m) = g\dot{x}y_1^{-1}I^m$. Write $g = g_0 u$. Then

\[ \tr(\beta_g^{\ast}; \coh_c^0(Y_{\dot{w}}^m)[\chi|_{U_E}]) = \chi(g_0) + \chi^{\tau}(-g_0). \]
\end{lm}

\begin{proof}
Multiplying with some central element in $U_F \subseteq U_E$ (those act as scalars in $V_{\chi}$) we can assume that $g_0 = 1 + u^{2\alpha + 1}h$ for some $h 
\in U_F$ (and $0 \leq \alpha < \lengthofv$). We proceed analogously as in the proof of Proposition \ref{lm:our_form_of_Boy_trace_formula}. Let $i \in I_{m,\underline{w}_m}/I^m$. A point $\dot{x}I \in Y_{\dot{w}}^m$ can lie in the set $S_{\beta_g,i} = \{ \dot{x}I \in Y_{\dot{w}}^m \colon \beta_g(\dot{x} I^m) = \dot{x} i I^m \}$ from Lemma \ref{lm:Boy_trace_formula} only if $\beta_g$ fixes its $a$-coordinate. By Lemma \ref{lm:action_of_beta_varpi}, $\beta_u$ acts on the coordinate $a$ by $a = ua^{\prime} \mapsto u a^{\prime,-1}$. From this and Proposition \ref{prop:left_I_action_mega_formula} one easily deduces that $\beta_g^{\ast}(a) = a$ is equivalent to $a = \pm u$ (for any $g_0$) and that $h(g,\pm u) = 0$. Apply Propositions \ref{prop:left_I_action_mega_formula}, \ref{prop:action_of_alpha_expl}, \ref{prop:right_action_formula} to determine the actions of $\beta_g$ and $i$ on $Y_{\dot{w}}^m$. Exactly as in the proof of Proposition \ref{lm:our_form_of_Boy_trace_formula} we see that

\[ \tr(\beta_g^{\ast}; \coh_c^0(Y_{\dot{w}}^m)[\chi|_{U_E}]) = \frac{1}{(q-1)q^{2m}} \sum_{i_1 \in U_E/U_E^{m+1}} \sharp S_{\beta_g,i_1} \chi(i_1), \]

\noindent where $S_{\beta_g,i_1}$ is the set of all solutions of the equations 

\begin{eqnarray}
\label{eq:varpi_eq_1} (1 \mp u^{2\alpha + 1}h)(\mp C) M^{-1}N^{-1} &\equiv& C i_1 H^{-1} \mod u^{m+1} \\
\label{eq:varpi_eq_2} -AMN &\equiv& i_1^{-1}\tau(i_1)H^2 A + i_2 H \mod u^m
\end{eqnarray}

\noindent in the variables $C \in (k[u]/u^{m+1})^{\times}$, $A, i_2 \in k[u]/u^m$ (the sign $\pm$ corresponds to the two possibilities $a = \pm u$), where

\begin{eqnarray*}
M &=& 1 - 2u^{\lengthofv}C\tau(C)^{-1}A \quad \text{as in Proposition \ref{prop:action_of_alpha_expl}, and} \\
N &=& 1 + u^{\lengthofv + 1} \frac{g_2}{g_2 (\pm 1) + g_1} (\mp CM^{-1})(\mp D M)^{-1} (-AM) = 1 \pm 2u^{\lengthofv} \frac{u^{2\alpha + 1}h}{1 \pm u^{2\alpha + 1}h} C \tau(C)^{-1}A \\
H &=& 1 + i_1 \tau(i_1)^{-1}i_2 B = 1 + u^{\lengthofv} i_1 \tau(i_1)^{-1}i_2 C \tau(C)^{-1}.
\end{eqnarray*}

\noindent Canceling $C$ in \eqref{eq:varpi_eq_1} we see that it is equivalent to

\begin{equation} \label{eq:varpi_eq_1_v2}
MN i_1 \equiv  \mp(1 \mp u^{2\alpha + 1}h) H \mod u^{m+1}.
\end{equation}

\noindent Taking \eqref{eq:varpi_eq_1_v2} modulo $u^{\lengthofv}$, we see that $S_{\beta_g,i} = \emptyset$, unless $i_1 \equiv \mp (1 \mp u^{2\alpha + 1}h) \mod u^{\lengthofv}$. Assume the last holds. Taking equation \eqref{eq:varpi_eq_2} modulo $u^{\lengthofv}$ and inserting $MN$ from \eqref{eq:varpi_eq_1_v2} and $i_1 \equiv \mp (1 \mp u^{2\alpha + 1}h) \mod u^{\lengthofv}$ we deduce

\[ i_2 \equiv -A \frac{2}{1 \mp u^{2\alpha + 1}h} \mod u^{\lengthofv}. \]

\noindent This allows to compute 

\begin{equation}\label{eq:H_for_varpi}
H = 1 - 2u^{\lengthofv}\frac{1}{1 \pm u^{2\alpha + 1}h} C\tau(C)^{-1}A.
\end{equation}

\noindent As in the proof of Proposition \ref{lm:our_form_of_Boy_trace_formula} we eliminate $i_2$ and ignore equation \eqref{eq:varpi_eq_2}. Thus $\sharp S_{\beta_g,i_1}$ is equal to the number of solutions of \eqref{eq:varpi_eq_1_v2} in $C$ and $A$. Finally, we compute $MN = H$ (this uses $i_1 \equiv \mp (1 \mp u^{2\alpha + 1}h) \mod u^{\lengthofv}$ and \eqref{eq:H_for_varpi}) and canceling these terms in \eqref{eq:varpi_eq_1_v2} shows that $S_{\beta_g,i_1} = \emptyset$, unless $i_1 = \mp (1 \mp u^{2\alpha + 1}h)$, in which case $\sharp S_{\beta_g,i_1} = (q-1)q^{2m}$, finishing the proof.
\end{proof}


\subsection{Traces on the induced side}\label{sec:trace_on_induced_side}

\begin{proof}[Proof of Lemma \ref{lm:restrictions_to_Etimes_are_isom}]
In $\Xi_{\chi}$ and $\Theta_{\chi}$ the central characters are $\chi|_{F^{\times}}$ and $U_E^{m+1}$ acts trivial. Thus it is enough to show

\begin{eqnarray} \label{eq:claimed_equality_of_traces}
\tr(g; \Xi_{\chi}) &=& \tr(g; \Theta_{\chi}) \quad \text{$\forall$ $g = \iota(1 + u^{2\alpha + 1}h)$ with $h \in U_F$ and $0 \leq \alpha \leq \lengthofv - 1$,} \\
\tr(g; \Xi_{\chi}) &=& \tr(g; \Theta_{\chi}) \quad \text{$\forall$ $g \in \varpi U_E$}. \label{eq:claimed_eq_for_trace_of_varpi}
\end{eqnarray}


\begin{lm}\label{lm:rep_system_for_UJ_UEUJell}
Let $\lengthofv \geq 1$. For $y \in U_F/U_F^{\lfloor \frac{\lengthofv+1}{2} \rfloor }$, $\lambda \in \caO_F/\caO_F^{\lfloor \frac{\lengthofv}{2} \rfloor}$ the matrices

\begin{equation}
r_{y,\lambda} := \matzz{1}{}{}{y} \matzz{1}{\lambda}{}{1} \in U_{\fJ}/U_{\fJ}^{\lengthofv}
\end{equation}

\noindent for a system of representatives for the left (and right) cosets of $U_E U_{\fJ}^{\lengthofv}$ in $U_{\fJ}$ and hence also of $J_{\beta} = E^{\times}U_{\fJ}^{\lengthofv}$ in $E^{\times}U_{\fJ}$.
\end{lm}

\begin{proof}
We have $U_{\fJ} / U_E U_{\fJ}^{\lengthofv} \cong E^{\times}U_{\fJ}/E^{\times}U_{\fJ}^{\lengthofv}$. The rest is an immediate computation. 
\end{proof}

\noindent We use notations from Section \ref{sec:BH_constr_of_pi_chi} and compute the traces $\tr(g; \Theta_{\chi})$. Let $g = \iota(1 + u^{2\alpha + 1}h)$ be as in \eqref{eq:claimed_equality_of_traces}. Applying the Mackey formula to $\Theta_{\chi} = \Indd_{J_{\beta}}^{E^{\times}U_{\fJ}} \Lambda$ we see

\begin{equation}
\tr(g; \Theta_{\chi}) = \sum_{y,\lambda} \Lambda(r_{y,\lambda} g r_{y,\lambda}^{-1}), 
\end{equation}

\noindent where the sum is taken over all representatives $r_{y,\lambda}$ of $E^{\times} U_{\fJ}/J_{\beta}$ (from Lemma \ref{lm:rep_system_for_UJ_UEUJell}), such that $r_{y,\lambda} g r_{y,\lambda}^{-1} \in J_{\beta} = E^{\times} U_{\fJ}^{\lengthofv}$. We compute:

\[r_{y,\lambda} g r_{y,\lambda}^{-1} = \matzz{1 + \lambda h t^{\alpha + 1}}{y^{-1}h(1 - \lambda^2 t) t^{\alpha}}{yht^{\alpha + 1}}{1 - \lambda h t^{\alpha + 1}}. \]

\noindent Write $\beta = (b + uc)u^{-m}$ with some $b,c \in \caO_F$. 
Assume first $\alpha \geq \lfloor \frac{\lengthofv}{2} \rfloor$. Then $r_{y,\lambda} g r_{y,\lambda}^{-1} \in U_{\fJ}^{\lengthofv} \subseteq U_E U_{\fJ}^{\lengthofv}$ for all $r_{y,\lambda}$ and we compute:

\[ \tr(g; \Theta_{\chi}) = \sum_{y,\lambda} \psi_{\fM,\beta}(r_{y,\lambda} g r_{y,\lambda}^{-1}) = \sum_{y,\lambda} \psi(t^{\alpha + 1 - \lengthofv}bh(y + y^{-1}(1 - \lambda^2 t))). \]

\noindent Taking some lifts of $y,\lambda$ to $E$ and setting $n := \frac{1}{2} h(y + y^{-1}(1 - \lambda^2 t)) u^{2\alpha + 1} \in E$, we see that $\beta n + \tau(\beta n) = t^{\alpha + 1 - \lengthofv}bh(y + y^{-1}(1 - \lambda^2 t))$, i.e., using \eqref{eq:BH_choice_identity_between_chi_and_psi}, we deduce

\[ \tr(g; \Theta_{\chi}) = \sum_{y,\lambda} \psi_E(\beta n) = \sum_{y,\lambda} \chi(1 + \frac{1}{2} h(y + y^{-1}(1 - \lambda^2 t)) u^{2\alpha + 1}). \]

\noindent This does not depend on the choice of the lifts $y,\lambda$ to $E$, as $\chi$ is of level $m$. Interpreting $1 + \frac{1}{2} h(y + y^{-1}(1 - \lambda^2 t)) u^{2\alpha + 1}$ as an element of $U_E/U_E^{m+1}$, we have to show that the summand $\chi(i_1)$ for $i_1 \in U_E/U_E^{m+1}$ occurs in this sum if and only if and with the same multiplicity as it occurs in the sum \eqref{eq:traces_of_non_split_elements_in_prop}. Therefore, writing $i_1 = 1 + u^{2\alpha + 1}hs$, it is enough to show that for a fixed $s \in R_{\alpha}^{\langle \tau \rangle}$ the equation 

\begin{equation}\label{eq:for_trace_comparison} 
\frac{1}{2} (y + y^{-1}(1 - \lambda^2 t)) \equiv s \mod t^{\lengthofv - \alpha} 
\end{equation}

\noindent in the variables $y \in U_F/U_F^{\lfloor \frac{\lengthofv+1}{2} \rfloor }$, $\lambda \in \caO_F/\caO_F^{\lfloor \frac{\lengthofv}{2} \rfloor}$ is equivalent to the equation \eqref{eq:Gl2incase_alphagreater_l2} in the variable $a^{\prime} = a_1 + a_2 u + \dots + a_{\lengthofv}u^{\lengthofv - 1} \in (k[u]/u^{\lengthofv})^{\times}$. Indeed, write $a^{\prime} = -b^{\prime} + c^{\prime}u$ with $b^{\prime} = -\frac{1}{2}( \tau(a^{\prime}) + a^{\prime} ) = \sum_{j = 0}^{\lfloor \frac{\lengthofv+1}{2} \rfloor - 1} b_j t^j$ and $c^{\prime}u = \frac{1}{2}(\tau(a^{\prime}) - a^{\prime}) = u \sum_{j = 0}^{\lfloor \frac{\lengthofv}{2} \rfloor - 1}c_j t^j$ with $b^{\prime},c^{\prime}$ $\tau$-invariant. Then \eqref{eq:Gl2incase_alphagreater_l2} can be rewritten as

\[ s \equiv \frac{1 + b^{\prime, 2} - c^{\prime,2} t}{2b^{\prime}} \mod t^{\lengthofv - \alpha}, \]

\noindent which is evidently equivalent to \eqref{eq:for_trace_comparison} (replace $b^{\prime}$ by $y$ and $c^{\prime}$ by $\lambda$). The case $\alpha < \lfloor \frac{\lengthofv}{2} \rfloor$ can be done similarly. This shows \eqref{eq:claimed_equality_of_traces}.

To show \eqref{eq:claimed_eq_for_trace_of_varpi}, we let $g = \iota(u(1 + hu))$ for some $h \in \caO_F$ (restriction to this case is possible after multiplication with a central element). We compute 

\begin{equation} \label{eq:commut_of_varpi_mit_rlambday}
\varpi^{-1} r_{y,\lambda} g r_{y,\lambda}^{-1} = \matzz{y}{h-\lambda}{(h + \lambda) t}{y^{-1}(1- \lambda^2 t)}.
\end{equation}

\noindent Notice that $r_{y,\lambda}g r_{y,\lambda}^{-1} \in J_{\beta} = E^{\times} U_{\fJ}^{\lengthofv}$ if and only if $\varpi^{-1} r_{y,\lambda}g r_{y,\lambda}^{-1} \in E^{\times}U_{\fJ}^{\lengthofv} \cap U_{\fJ} = U_E U_{\fJ}^{\lengthofv}$. By \eqref{eq:commut_of_varpi_mit_rlambday}, this is the case if and only if $\lambda = 0$, $y = \pm 1$. Thus $\tr(g; \Theta_{\chi}) = \chi(g) + \chi^{\tau}(g)$. Together with Proposition \ref{prop:traces_of_val_1_elements} it shows \eqref{eq:claimed_eq_for_trace_of_varpi} and thus the lemma. \qedhere 
\end{proof}


\subsection{Computation of traces in the small level case}\label{sec:traces_small_level_case} \mbox{}

In this section we assume $\lengthofv \geq m+1$.

\begin{proof}[Proof of Lemma \ref{lm:traces_on_ADLV_side_small_level_case}] Let $g \in U_{\fJ}$. We apply Proposition \ref{lm:our_form_of_Boy_trace_formula}. Observe first that equation \eqref{eq:Gl_2_v1} reduces to
\begin{equation}\label{eq:Gl_2_v2} 
\tau(i_1) \equiv g_2 a + g_1 \mod u^{m+1}.
\end{equation}

\noindent (i): Then we are exactly in the case (i) of Lemma \ref{lm:on_equation_for_a}. As $v_u(a) = 1$ and $v_u(g_2) \geq 2\lfloor\frac{\lengthofv}{2}\rfloor$, we see that $g_2 a + g_1 \equiv g_1 \mod u^{\lengthofv}$ and hence also $g_2 a + g_1 \equiv g_1 \mod u^{m+1}$. Let $i_1 \in (k[u]/u^{m+1})^{\times}$. Then \eqref{eq:Gl_2_v2} simply says that either $i_1$ is $g_1 \mod u^{m+1}$ or $S_{g,i_1} = \emptyset$. By Proposition \ref{lm:our_form_of_Boy_trace_formula} we deduce

\[ \tr(g; \coh_c^0(\tilde{Y}_{\dot{w}})[\chi]) = (q-1)q^{\lengthofv - 1} \chi(g_1) = (q-1)q^{\lengthofv - 1}\tilde{\chi}(g), \]

\noindent showing the first statement of (i). The last statement of (i) follows immediately from the first, as $\tilde{\chi}$ is trivial on $U_{\fJ}^{\lengthofv}$. 

\noindent (ii): Conjugating $g$ into $U_E U_{\fJ}^{\lengthofv}$ and multiplying with an element of $U_F U_{\fJ}^{\lengthofv}$ (these elements act by part (i) as scalars), we can without loss of generality assume that $g = \iota(1 + u^{2\alpha + 1}h)$ with some $h \in U_F$ and with $2\alpha + 2 = v_u(g_3) \leq \lengthofv$.  Let $i_1 \in U_E/U_E^{m+1}$. We determine $\sharp S_{g,i_1}$. First of all \eqref{eq:Gl_1_v1} is equivalent to


\[ (t^{\alpha}h) a^2 - t^{\alpha + 1} h \equiv 0 \mod u^{\lengthofv + 1}. \]

\noindent Write $a = a^{\prime} u$ with $a^{\prime}$ invertible. This equation is equivalent to

\begin{equation}\label{eq:Gl_1_v2_easy_case} 
a^{\prime} \equiv \pm 1 \mod u^{\lengthofv + 1 - (2\alpha + 2)}.
\end{equation}

\noindent Equation \eqref{eq:Gl_2_v2} takes the form 

\[ \tau(i_1) = 1 + u^{2\alpha + 1}h a^{\prime} \mod u^{m+1}. \]

\noindent Thus \eqref{eq:Gl_1_v2_easy_case} and $\lengthofv \geq m+1$ shows that either $S_{g,i_1} = \emptyset$, or $i_1 = 1 \pm u^{2\alpha + 1}h$. Moreover, for each of this two choices of $i_1$, there are exactly $q^{v_u(g_3) - 1} = q^{2\alpha + 1}$ possible $a$'s satisfying equations \eqref{eq:Gl_1_v2_easy_case} and \eqref{eq:Gl_2_v2} (cf. Lemma \ref{lm:on_equation_for_a}(ii)). We obtain

\[ \tr(g; \coh_c^0(\tilde{Y}_{\dot{w}})[\chi]) = q^{v_u(g_3) - 1} \cdot (\tilde{\chi}(g) + \tilde{\chi}^{\tau}(g)). \]

\noindent (iii): By Lemma \ref{lm:on_equation_for_a}(iii) it is clear that $S_{g,i_1} = \emptyset$ for all $i_1$ in this case. 

Let now $g = g_0 \varpi \in \varpi U_E$. As in the proof of Proposition \ref{prop:traces_of_val_1_elements} we have the automorphism $\tilde{\beta}_g$ of $\tilde{Y}_{\dot{w}}^m$ defined by $\tilde{\beta}_g(\dot{x}I^m) = g\dot{x}y_1^{-1}I^m$, where $y_1 = e_0(u,-u)$ and its restriction $\beta_g$ to $Y_{\dot{w}}^m$. Again, we have

\[ \tr(g; \coh_c^0(\tilde{Y}_{\dot{w}}^m)[\chi]) = \chi(u) \tr(\tilde{\beta}_g^{\ast}; \coh_c^0(\tilde{Y}_{\dot{w}}^m)[\chi]) = \chi(u) \tr(\beta_g^{\ast}; \coh_c^0(Y_{\dot{w}}^m)[\chi]). \]

\noindent The right action of $I_{m,\underline{w}_m}/I^m$ does not affect the $a$-coordinate of a point $\dot{x}I \in Y_{\dot{w}}^m$, thus we see from the Lemma \ref{lm:Boy_trace_formula} that $\tr(\beta_g^{\ast}; \coh_c^0(Y_{\dot{w}}^m)[\chi]) = 0$, unless $\beta_g^{\ast}(a) = a$. A simple computation shows that this can only be the case if $g$ is conjugate to an element in $E^{\times}U_{\fJ}^{\lengthofv}$. This shows (ii)${}^{\prime}$. If $g$ is conjugate to an element of $E^{\times}U_{\fJ}$, then we can assume $g \in E^{\times}U_{\fJ}$ and (i)${}^{\prime}$ can be shown as in the proof of Proposition \ref{prop:traces_of_val_1_elements}.
\end{proof}


\end{document}